\documentclass{amsart}
\usepackage{times}
\usepackage{amsfonts}
\usepackage{amsmath}
\usepackage{amscd}
\usepackage{amssymb}
\usepackage{latexsym}
\usepackage{multirow}
\usepackage{amsthm} 
\input{xy}
\xyoption{all}

\theoremstyle{plain}

\newtheorem*{conjectuur*}{Conjecture}

\swapnumbers
\newtheorem{theorem}[subsection]{Theorem}
\newtheorem{corollary}[subsection]{Corollary}
\newtheorem{lemma}[subsection]{Lemma}
\newtheorem{proposition}[subsection]{Proposition}

\newtheorem{conjecture}[subsection]{Conjecture}

\theoremstyle{definition}
\newtheorem{definition}[subsection]{Definition}

\newtheorem{example}[subsection]{Example}

\theoremstyle{remark}

\newtheorem{remark}[subsection]{Remark}

\swapnumbers

\newcommand{\emptyprop}{q}

\newcommand \acf{algebraically closed field}
\newcommand \after{\circ}

\newcommand \binomial[2]{{\bigl( \begin{matrix} #1\cr#2\cr\end{matrix} \bigr)}}
\newcommand \ch{characteristic}

\newcommand \complet[1]{\widehat {#1}}

\newcommand \homo{homomorphism}
\newcommand \id{\mathfrak a}
\renewcommand\iff{if and only if}
\newcommand\into{\hookrightarrow}
\newcommand \inv[1]{{#1^{-1}}}
\newcommand \inverse[2]{{#1^{-1}(#2)}}
\newcommand \iso{\cong}
\newcommand \loc{{\mathcal {O}}}

\newcommand \map[1]{{\newcommand{\tmpprop}{#1q}  \if\tmpprop\emptyprop \to\else \xrightarrow{{\phantom{i}{#1}\phantom{i}}}\fi}} 
\newcommand \maxim{\mathfrak m}
\newcommand \nat{\mathbb N}
\newcommand \norm[1]{\left|#1\right|}
\newcommand \onto{\twoheadrightarrow}
\newcommand \op\operatorname
\newcommand \pol[2]{#1[#2]}
\newcommand \pow[2]{#1[[#2]]}

\newcommand \pr{\mathfrak p}

\newcommand \range [2]{#1,\dots,#2}
\newcommand \restrict [2]{\left.#1\right|_{{#2}}}
\newcommand \rij[2]{(#1_1,\dots,#1_{#2})}
\let\sub\subseteq

\newcommand \tensor{\otimes}

\newcommand \zet{\mathbb Z}

\newcommand\asa{\leftrightarrow}
\newcommand\dan{\to}
\newcommand\En{\bigwedge}
\newcommand\en{\wedge}
\newcommand\niet{\neg}
\newcommand\Of{\bigvee}
\newcommand\of{\vee}

\newcommand{\commdiagram}[9][]{%
\begin{equation}
{\newcommand{\tmpprop}{#1q} 
\if\tmpprop\emptyprop \relax\else \label{#1}\fi}
\begin{aligned}%
\mbox{
\begin{picture}(130,90)%
\put(120,70){\vector( 0,-1){50}}%
\put(10,80){\vector( 1, 0){100}}%
\put(0,70){\vector( 0,-1){50}}%
\put(10,10){\vector( 1, 0){100}}%
\put(115,80){\makebox(0,0)[l]{$#4$}}%
\put(5,80){\makebox(0,0)[r]{$#2$}}%
\put(115,10){\makebox(0,0)[l]{$#9$}}%
\put(5,10){\makebox(0,0)[r]{$#7$}}%
\put(-3,50){\makebox(0,0)[r]{$#5$}}
\put(123,50){\makebox(0,0)[l]{$#6$}}
\put(60,3){\makebox(0,0)[c]{$#8$}}
\put(60,88){\makebox(0,0)[c]{$#3$}}
\end{picture}}
\end{aligned}
\end{equation}}

\newcommand\commtrianglefront[7][]{%
\begin{equation}
{\newcommand{\tmpprop}{#1q} 
\if\tmpprop\emptyprop \relax\else \label{#1}\fi}
\begin{aligned}%
\mbox{
\begin{picture}(120,80)%
\put(55,68){\vector(-1,-2){30}}
\put(65,68){\vector(1,-2){30}}
\put(30,5){\vector(1,0){60}}
\put(60,75){\makebox(0,0)[c]{$#2$}}
\put(25,5){\makebox(0,0)[r]{$#4$}}
\put(95,5){\makebox(0,0)[l]{$#6$}}
\put(60,0){\makebox(0,0)[c]{$#5$}}
\put(37,43){\makebox(0,0)[r]{$#3$}}
\put(83,43){\makebox(0,0)[l]{$#7$}}
\end{picture}}
\end{aligned}
\end{equation}}

\newcommand\commtriangleback[7][]{%
\begin{equation}
{\newcommand{\tmpprop}{#1q}
\if\tmpprop\emptyprop \relax\else \label{#1}\fi}
\begin{aligned}%
\mbox{
\begin{picture}(120,80)%
\put(55,70){\vector(-1,-2){30}}
\put(65,70){\vector(1,-2){30}}
\put(30,5){\vector(1,0){60}}
\put(60,75){\makebox(0,0)[c]{$#2$}}
\put(25,5){\makebox(0,0)[r]{$#6$}}
\put(95,5){\makebox(0,0)[l]{$#4$}}
\put(60,0){\makebox(0,0)[c]{$#5$}}
\put(37,43){\makebox(0,0)[r]{$#7$}}
\put(83,43){\makebox(0,0)[l]{$#3$}}
\end{picture}}
\end{aligned}
\end{equation}}

\newcommand\NYCCT{\address{Department of Mathematics\\
City University of New York\\
365 Fifth Avenue\\ 
New York, NY 10016 (USA)}
\email{hschoutens@citytech.cuny.edu}}

\usepackage[bookmarks,colorlinks,pagebackref]{hyperref} 

\newcommand\fld{F}
\newcommand\recur[3]{#2\lhd_{#1}#3}

\newcommand\ord[2]{\op{ord}_{#1}(#2)}
\newcommand\round[2]{\lceil\frac{#1}{#2}\rceil}
\newcommand\tin[2]{\tilde{#1}^{#2}}
\newcommand\org[2]{\dot{#1}({#2})}

\renewcommand\dan{\Rightarrow}
\newcommand \affine[2]{{\mathbb A_{#1}^{#2}}}
\newcommand \gr{Gro\-then\-dieck ring}

\newcommand \sciss{scissor group}
\newcommand \grgroup{Gro\-then\-dieck group}
\newcommand\zar{\mathcal Sch}
\newcommand\pp{\mathcal {PP}}
\newcommand\qf{\mathcal {QF}}
\newcommand\expl{\mathcal Expl}
\newcommand\zariski{schemic}
\newcommand\Zariski{Schemic}

\newcommand \var{x}
\newcommand \vary{y}
\newcommand\tuple[1]{\mathbf{#1}}
\newcommand\fim[1]{\texttt{Im}(#1)}

\newcommand \inter[2]{#1(#2)}
\newcommand \coor[2]{\nabla_{\!#1}#2} 
\newcommand \parc[3]{\nabla_{\!#1}^{#2}#3} 
\newcommand \arc[2]{\nabla_{\!#1}#2} 
\newcommand \sym[1]{{\langle #1\rangle}}
\newcommand \class[1]{{\left[ #1\right]}}
\newcommand \arcbar[1]{\bar\nabla {#1}} 
\newcommand \form{formulary}
\newcommand \forms{formularies}

\newcommand \defcat[1]{\op {Def}({#1})}

\newcommand \grot[1]{{\mathbf {Gr}(#1)}}
\newcommand \grotlat[1]{{\mathbf {Sciss}(#1)}}
\newcommand \grotsch[1]{{\mathbf {Gr}(#1^{\op{sch}})}}

\newcommand \grotschzero[1]{{\mathbf {Gr}_0(#1^{\op{sch}})}}
\newcommand \grotschinv[1]{{\mathbf {Gr}(#1^{\op{sch}})}_\lef}
\newcommand \grotclass[1]{{\mathbf {Gr}(#1^{\op{var}})}}
\newcommand \grotart[2]{{\mathbf {Gr}_{#2}(#1^{\op{pp}})}}

\newcommand \grotartinv[1]{\mathbf {Gr}(#1^{\op{pp}})_{\lef}}
\newcommand \grotartinf[1]{\mathbf {Gr}^\infty(#1^{\op{pp}})} 
\newcommand \grotform[1]{{\mathbf {Gr}}(#1^{\op{form}})}

\newcommand \grotmod[3]{\mathbf {Gr}^{#1}_{#3}({#2})}
\newcommand \grotmon[2]{\mathbf {K}_{#2}(#1)}
\newcommand\ACF[1]{\theory {ACF}_{#1}}
\newcommand\arttheory[1]{\theory {Art}_{#1}}
\newcommand\loctheory[1]{\theory T_{#1}^{loc}}

\newcommand\integral[2]{\int #1\ d#2}
\newcommand\motint[2]{\int^{\op{mot}} #1\ d#2}

\newcommand \lef{{\mathbb L}}

\newcommand \theory[1]{\mathbf{#1}}

\newcommand \one{\top}
\newcommand \nul{\bot}

\newcommand \jet[3]{J_{#1}^{#2}#3}
\newcommand \Jet[2]{J_{#1}#2}
\newcommand \mor[3]{\op{Mor}_{#1}(#2,#3)}
\newcommand \jac[1]{\op{Jac}_{#1}}
\newcommand\igugeom[1]{\op{Igu}^{\op{geom}}_{#1}}
\newcommand\igutwist[2]{\op{Igu}^{#2}_{#1}}
\newcommand\igu[3]{\op{Igu}^{(#2,#3)} _{#1}}
\newcommand\igumot[3]{\op{Igu}^{(#2,#3)}_{ #1^{\op{mot}}}}

\title {Schemic Grothendieck rings and motivic rationality}
\author{Hans Schoutens}
\date\today
 \NYCCT

\begin{document}


\begin{abstract}  
We propose a suitable substitute for the classical \gr\ of an \acf, in which any quasi-projective scheme  is represented, while maintaining its  non-reduced structure. This yields a more subtle invariant, called the \zariski\ \gr, in which we can formulate a form of integration resembling Kontsevich's motivic integration via arc schemes. In view of its  more   functorial properties, we can present a \ch-free proof of the rationality of the geometric Igusa zeta series  for certain hypersurfaces, thus generalizing the ground-breaking work on motivic integration by Denef and Loeser. The construction  uses first-order formulae, and some infinitary versions, called formularies. 
\end{abstract}

\maketitle

\section{Introduction}

The classical  \gr\ $K_0(k)$ of an \acf\ $k$ is defined as the quotient of the free Abelian group on  varieties   over $k$, modulo the relations $\class X-\class {X'}$, if $X\iso X'$, and 
\begin{equation}\label{eq:scissvar}
\class X=\class{X- Y}+\class Y, 
\end{equation}
if $Y$ is a closed subvariety, for $Y,X,X'$ varieties (=reduced, separated schemes of finite type over $k$). We will refer to the former relations as \emph{isomorphism relations} and to the latter   as \emph{scissor relations}, in the sense that we    ``cut out $Y$ from $X$.'' In this way, we cannot just take the class  of a variety, but of any constructible subset. Multiplication on $K_0(k)$ is then induced by the fiber product. In sum, the three main ingredients for building the \gr\ are:  an isomorphism relation,  scissor relations, and a product. Only the former causes problems if one wants to generalize the construction of the \gr\ to include not just classes of varieties, but also of   finitely generated schemes (with their nilpotent structure). Put bluntly,  we cannot cut a scheme in two, as there is no notion of a scheme-theoretic complement.  To describe what this ought to be, we turn to model-theory.

To model-theorists, constructible subsets are nothing else than definable subsets (in view of quantifier elimination for \acf{s}). Moreover, union and intersection correspond to disjunction and conjunction of the corresponding formulae. Therefore, instead of working with the theory of \acf{s}, we could repeat the previous construction over any first-order theory $\theory T$. However, now it is less obvious what it means for two formulae to be isomorphic. The most straightforward way is to introduce the notion of a definable isomorphism. However, even for the theory of \acf{s}, this yields a priori a different notion of isomorphism than the geometric one: whereas the former allows for arbitrary quantifier free formulae, the latter is given by polynomials, that is to say, of formulae of the form $\vary=f(\var)$, which we will call \emph{explicit formulae}. This observation suggests that we should consider not necessarily all first-order formulae, but also some restricted classes. This general construction is discussed in \S\ref{s:grth}. 

It is beneficial to develop the theory in a relative setup, so we work over an arbitrary affine, Noetherian scheme $X=\op{Spec}A$, instead of just over an \acf. To construct a generalized \gr\ for schemes, a so-called   \emph{\zariski\ \gr}, we need to settle on a first-order theory $\theory T$. The classical \gr\ is obtained by taking for $\theory T$ the theory of \acf{s} that are also $A$-algebras. Alternatively, one could also have chosen   the theory of all $A$-algebras without zero-divisors, and so, to include all schemes, we could simply  replace this by the theory $\theory T_A$ of all $A$-algebras. Refinements lead to more relations and hence more manageable \gr{s}, the most important one of which is the  theory $\arttheory A$ of all Artinian local $A$-algebras (=algebras that have finite length as an $A$-module). Since we no longer have quantifier elimination, we also need to make a decision on which formulae we will allow, both for our definition of isomorphism as well as for the classes we want to study. Varieties, and more generally schemes, are given by equations, and so the family of formulae of the form $f_1=\dots=f_s=0$, with $f_i\in \pol A\var$ will provide the proper candidate for our generalization to schemes; we call such formulae therefore \emph{\zariski}. We show that there is a one-one correspondence between \zariski\ formulae in $n$ free variables up to $\theory T_A$-equivalence,\footnote{Two formulae are equivalent modulo a theory if they define the same subsets in each model of the theory.}  and closed subschemes of $\affine Xn$ (Theorem~\ref{T:fosch}). In fact, this result remains true when working in the theory $\arttheory A$. As for isomorphisms, we may take either the class of explicit isomorphisms, or the larger class of \zariski\ isomorphisms, both choices leading to the same  \emph{\zariski\ \gr} $\grotsch X$. There is an obvious ring \homo\ $\grotsch X\to K_0(X)$ to the classical \gr\ $K_0(X)$ of $X$. The main result, Theorem~\ref{T:classinv}, is that two affine schemes of finite type over $X$ are isomorphic \iff\ they have the same class in $\grotsch X$.

However, if we want more relations to hold in our \gr, we need to enlarge  the family of formulae, and work in the appropriate theory. In \S\ref{s:ppgr}, we explain how in order to define the class of a non-affine scheme, we need to work modulo the theory $\arttheory A$ in the larger class of \emph{pp-formulae}, that is to say,   existentially quantified \zariski\ formulae.  This is apparent already when dealing with basic open subsets: if $U=\op D(f)$ is the basic open subset of $\affine Xn$, that is to say, $U=\op{Spec}(\pol A\var_f)$, then as an abstract affine scheme, it is given by the \zariski\ formula $f(\var)z=1$, where $z$ is an extra variable, whereas as an open subset of $\affine Xn$, it is given by the pp-formula $(\exists z)f(\var)z=1$; the isomorphism between these two sets is only true modulo $\arttheory A$, and is   given by a (non-explicit) \zariski\ formula. This leads to the \emph{pp-\gr} $\grotart X{}$ of $X$, where instead of  quantifier free formulae, we take Boolean combinations of pp-formulae, up to \zariski\ isomorphisms. To any scheme of finite type over $X$, we can, by taking an open affine covering,  associate a unique element in $\grotart X{}$. 

Unfortunately,   the original scissor relation~\eqref{eq:scissvar} is no longer valid. Indeed, the complement of an open $U\sub Y$ does not carry a unique closed scheme structure anymore. The solution is to take the limit over all these structures, yielding the formal completion $\complet Y_Z$, where $Z$ is the underlying variety of $Y-U$. At the level of formulae (for simplicity, we assume $A=\fld$ is an \acf\ henceforth), the negation of the pp-formula defining a basic open subset is equivalent with an infinite disjunction of \zariski\ formulae, having the property that in any Artinian $\fld$-algebra, the set defined by the disjunction is already definable by one of the disjuncts (but different models may require different disjuncts). Such an infinitary (whence non-first-order) disjunction will be termed a \emph{formulary}. Replacing formulae by formularies in the definition of the pp-\gr, yields the \emph{infinitary \gr} $\grotartinf \fld$, in which every formal scheme over $\fld$ is represented by the class of some formulary, resurrecting the old scissor relation into a new one: for any closed immersion $Z\sub Y$ of schemes of finite type over $X$, we have $\class Y=\class {Y-Z}+\class{\complet Y_Z}$.  All this is explained in \S\ref{s:infgr}.

There is one more variant that will be considered here, called the \emph{formal \gr} $\grotform \fld$, in which we revert to the reduced situation by factoring out the ideal generated by all $\class {\complet Y_Z}-\class Z$ for all closed immersions $Z\sub Y$. However, we will only work in this latter quotient (in which any two schemes with the same underlying variety have the same class) after we have taken   arcs (see below). This does make a difference, as can be seen already on easy examples (Table~\ref{tab:1}). The advantage is that we get back the original scissor relation~\eqref{eq:scissvar}, which makes it easier to invoke inductive arguments when proving rationality of the motivic Igusa zeta series (to be discussed below). The relation between all these \gr{s} is given by the  following (ring) \homo{s}
\begin{equation}\label{eq:grhomo}
\grotsch \fld\to\grotart \fld{}\to \grotartinf \fld\to \grotform\fld\to\grotclass\fld.
\end{equation}

To discuss our main application, the motivic rationality of the geometric Igusa zeta series, we introduce a weak version of motivic integration in \S\ref{s:motint}. For any Artinian $\fld $-algebra $R$, and any affine scheme $X$ over $\fld $, we define the \emph{arc scheme} $\arc RX$ of $X$ along $\op{Spec}R$ as the scheme whose $S$-rational points correspond to the $R\tensor_\fld S$-rational points of $X$, for any Artinian $\fld$-algebra $S$. This generalizes the truncated arc space of a variety, which is obtained by taking $R=\pol \fld\xi/\xi^n\pol \fld \xi$ and ignoring the nilpotent structure. The \emph{arc integral} $\integral ZX$  is then defined as the class $\class{ \arc Z X }$ in $\grotart \fld{}$, and the main result is that it only depends only the classes of $X$ and $Z$ (unlike in the classical case). The \emph{geometric Igusa zeta series of  $X$ along the germ $(Y,P)$} is then defined as the formal power series
$$
\igu XYP(t):=\sum_n\left(\integral{\jet PnY}X\right)t^n
$$
in $\pow{\grotart \fld{} }t$, where the \emph{$n$-th jet} of a germ $(Y,P)$ is defined as the Artinian scheme $\jet PnY:=\op{Spec}(\loc_Y/\maxim^n_P)$, with $\maxim_P$   the maximal ideal  of the closed point $P$. For the remainder of this introduction, I will assume that  $(Y,P)$ is the germ of a point on a line, and simply write $\igugeom X$ for this zeta series. Under the \homo\ from \eqref{eq:grhomo}, this power series   becomes the Denef-Loeser geometric Igusa zeta series. The aim is to recover within the new framework their result that $\igugeom X$ is rational over the localization $\grotclass\fld_\lef$, where $\lef$ is the class of an affine  line, called the \emph{Lefschetz class}. Their proof relies on Embedded Resolution of Singularities, and hence works in positive \ch\ only for surfaces. In \S\S\ref{s:ratIgu} and \ref{s:DuVal}, I will give examples of   hypersurfaces, in any \ch, for which we can derive the rationality of the Igusa zeta series (in fact, over $\grotform\fld_\lef$), without any appeal to resolution of singularities. The proofs are, moreover, far more elementary and algorithmic in nature because of the functorial properties of our construction.

\section{The \grgroup\ of a lattice}\label{s:lat}
The most general setup in which one can define  a \grgroup\ is the category of semi-lattices. Recall that   a \emph{lattice} $\Lambda$  is a partially ordered set in which every finite subset has an infimum and a supremum.  For any two elements $a,b\in \Lambda$, we let $a\en b$   and $a\of b$ denote respectively the infimum and the supremum of $\{a,b\}$. If only infima exist, then we call $\Lambda$ a \emph{semi-lattice}. Given a semi-lattice $\Lambda$, we call a finite subset $S\sub\Lambda$ \emph{admissible} if it has a supremum $a$, in which case we call $S$ a \emph{covering} of $a$. 
 
\subsection*{Scissor relations}
 For each $n>1$, we define the \emph{$n$-th scissor polynomial}
$$
S_n(x):=1-\prod_{i=1}^n(1-\var_i)=\var_1+\dots+\var_n-\dots+(-1)^{n-1}\var_1\cdots\var_n\in\pol\zet\var
$$
where $\var=\rij\var n$.
Let $\mathfrak i_n$ be the \emph{generic idempotency ideal}, that is to say, the ideal in $\pol\zet\var$ generated by the   relations $\var_i^2-\var_i$, for $i=\range 1n$. The following identities among scissor polynomials will be useful later:

\begin{lemma}\label{L:scisind}
For each $n$, we have an equivalence relation
\begin{equation}\label{eq:scisind}
S_n\rij\var n+S_{n-1}(\var_1\var_n,\dots,\var_{n-1}\var_n)\equiv S_{n-1}\rij\var{n-1}+\var_n\mod\mathfrak i_n.
\end{equation}
 in $\pol\zet\var$, with $\var=\rij\var n$. More generally, for $\vary=\rij\vary m$, we have
\begin{multline}\label{eq:scisprod}
S_{n+m}(\var_1,\dots,\var_n,\vary_1,\dots,\vary_m)+S_{nm}(\var_1\vary_1,\var_1\vary_2,\dots,\var_n\vary_m)\equiv\\
S_n\rij\var n+S_m\rij\vary m\mod\mathfrak i_{n+m}.
\end{multline}
\end{lemma}
\begin{proof}
Note that \eqref{eq:scisind} is just a special case of \eqref{eq:scisprod}. Let us first prove that in any ring $D$, we have an identity
\begin{equation}\label{eq:idemprod}
\prod_{i=1}^n(1-ea_i)=1-e+e\prod_{i=1}^n(1-a_i)
\end{equation}
for $a_i\in D$ and $e$ an idempotent in $D$. Indeed, write $1-ea_i=1-e+e(1-a_i)$. Since $e(1-e)=0$, the expansion of the product on the left hand side of \eqref{eq:idemprod} yields $(1-e)^n+e^n\prod_i(1-a_i)$, and the claim follows since $e$ and $1-e$ are both idempotents.

To prove \eqref{eq:scisprod}, we   carry our calculations out in $C:=\pol\zet{\var,\vary}/\mathfrak i_{n+m}$. To simplify notation, let us write $P:=\prod_i(1-\var_i)$ and $Q:=\prod_j(1-\vary_j)$. Hence the first, third, and fourth term in \eqref{eq:scisprod} are respectively $1-PQ$, $1-P$ and $1-Q$. Let us therefore expand the second term. Applying \eqref{eq:idemprod}, for each $i$, with $\var_i$ as   idempotent in the product indexed by $j$, and then again in the last line, with  $1-Q$ as idempotent, we get
\begin{align*}
1-S_{nm}(\var_1\vary_1,\var_1\vary_2,\dots,\var_n\vary_m)&=\prod_{i=1}^n\left(\prod_{j=1}^m(1-\var_i\vary_j)\right)\\
&=\prod_{i=1}^n (1-\var_i+\var_iQ)\\
&=\prod_{i=1}^n (1-(1-Q)\var_i)=
Q+(1-Q)P
\end{align*}
From this, \eqref{eq:scisprod} now follows immediately.
\end{proof}

We can write any scissor polynomial $S_n$ as the difference $S_n^+-S_n^-$ of two polynomials $S_n^+$ and $S_n^-$ with positive coefficients, that is to say, $S_n^+$ is the sum of terms in $S_n$  of odd degree, and $S_n^-$ is minus the sum of all terms of even degree. Put differently, $S_n^+$ and $S_n^-$ are the respective  sums of all square-free monomials in the variables $\rij xn$ of respectively odd and (positive) even degree.

\subsection*{The \sciss\ of a lattice}
 Given a lattice $\Lambda$, let $\pol\zet\Lambda$ be the free Abelian group on $\Lambda$. Using the infimum of $\Lambda$ as multiplication, we get a ring structure on $\pol\zet\Lambda$, that is to say, 
$$
(\sum_{i=1}^n a_i)\cdot(\sum_{j=1}^m b_j):=\sum_{i=1}^n\sum_{j=1}^ma_i\en b_j
$$
for $a_i,b_j\in\Lambda$.
Let $\tuple a:=\rij an$ be a tuple in $\Lambda $. We will write $\Of\tuple a$ for $a_1\of\dots\of a_n$.  Substitution induces a ring \homo\  $\pol\zet\var\to \pol\zet\Lambda\colon\var_i\mapsto a_i$.  In particular, $S_n(\tuple a)$ is a well-defined element in $\pol\zet\Lambda$, and we may abbreviate it as $S(\tuple a)$, since the arity is clear from the context. Note that, since any element of $\Lambda$ is idempotent in $\pol\zet\Lambda$, the   kernel of this \homo\ $\pol\zet \var\to  \pol\zet\Lambda$ contains $\mathfrak i$, the generic idempotency ideal.  We define the \emph{scissor relation} on $\tuple a$ as the formal sum
\begin{equation}\label{eq:scissor}
(\Of \tuple a)-S(\tuple a).
\end{equation}
  For instance, if $n=2$, then   the (second) scissor relation $(a_1\of a_2)-S_2(a_1,a_2)$ is equal to
\begin{equation}\label{eq:2scis}
(a_1\of a_2)-a_1-a_2+(a_1\en a_2).
\end{equation}
Similarly, for $n=3$, we get 
$$
 (a_1\of a_2\of a_3)-a_1-a_2-a_3+(a_1 \en a_2)+ (a_1 \en a_3)+ (a_2 \en a_3)- (a_1 \en a_2 \en a_3).
$$
We define the \emph{\sciss} of $\Lambda$, denoted $\grotlat\Lambda$, as the quotient of $\pol\zet\Lambda$ by   the subgroup $N$ generated by all second scissor relations~\eqref{eq:2scis}. Although we later will make a notational distinction between an element and its   class  in the \sciss, at present, no such distinction is needed, and so we continue to write $a$ for the image of $a\in\Lambda$ in $\grotlat\Lambda$.

\begin{remark}\label{R:enmult}
The ring structure on $\pol\zet\Lambda$, given by $\en$, descends to a ring structure on $\grotlat\Lambda$, since  $N$ is in fact an ideal. When we apply this to formulae in the next section,  this ring structure on $\grotlat\Lambda$  will play a minor role, and instead, a different multiplication will be introduced.
\end{remark}

\begin{proposition}\label{P:scissor}
For each  tuple $\tuple a$ in $\Lambda$, we have a scissor relation
$$
\Of\tuple a=S(\tuple a)
$$
in $\grotlat \Lambda$.
\end{proposition}
\begin{proof}
We prove this by induction on the length $n\geq 2$ of $\tuple a=\rij an$, where the case $n=2$ is just the definition. Since $\mathfrak i$ lies in the kernel of $\pol\zet\var\to \pol\zet \Lambda $, the equivalence~\eqref{eq:scisind} in Lemma~\ref{L:scisind} becomes an identity in the latter group, that is to say, 
\begin{equation}\label{eq:sn}
S_n\rij a n= 
S_{n-1}\rij a{n-1}-S_{n-1}(a_1\en a_n,\dots,a_{n-1}\en a_n)+ a_n.
\end{equation}

Viewing $\Of\tuple a$ as the disjunction $b\of a_n$, where     $b:=a_1\of\dots\of a_{n-1}$, the defining (second) scissor relation  yields  an identity
\begin{equation}\label{eq:s2}
\Of\tuple a=b+a_n-(b\en a_n)=(a_1\of\dots\of a_{n-1})+{a_n}-{(a_1\en a_n)\of\dots \of(a_{n-1}\en a_n)}.
\end{equation}
Subtracting \eqref{eq:sn} from \eqref{eq:s2}, the left hand side is $\Of\tuple a-S_n(\tuple a)$, and the right hand side is equal to
\begin{multline*}
\big((a_1\of\dots\of a_{n-1})-{S_{n-1}\rij a{n-1}}\big)-\\
\big({(a_1\en a_n)\of\dots ( a_{n-1}\of a_n)}-{S_{n-1}( a_1\en a_n,\dots, a_{n-1}\en a_n)}\big).
\end{multline*}
By induction, both terms are zero in $\grotlat\Lambda$, and the result follows.
\end{proof}

In particular,  $N$ is generated by all scissor relations, of any arity. We may generalize this to a semi-lattice $\Lambda$ (with multiplication still given by $\en$) as follows. Let $N$ be the subgroup of $\pol\zet\Lambda$ generated by all expressions $a-S_n(\tuple b)$, where $\tuple b=\rij bn$ (or rather its entries) ranges over all admissible coverings of $a$, that is to say, such that $a$ is the supremum of the $b_i$. It is not hard to check that $N$ is in fact an ideal, and the resulting residue ring is the scissor ring  $\grotlat\Lambda:=\pol\zet\Lambda/N$.

\subsection*{The \grgroup\ of a graph (semi-)lattice}
Let $\Lambda$ be a (semi-)lattice. 
By a \emph{(directed) graph} on  $\Lambda$ we simply mean  a binary relation $E$ on $\Lambda$ (we do not require it to be compatible with the join or the meet). We define the \emph{\grgroup} of $(\Lambda,E)$, denoted $\grotmon\Lambda E$, as the factor group of $\grotlat \Lambda$ modulo the subgroup   generated by all elements of the form $a-b$ such that there is an edge from $a$ to $b$. In other words, $\grotmon\Lambda E$ is the quotient of the free Abelian group $\pol\zet\Lambda$ on $\Lambda$ modulo the subgroup $N_{\text{sciss}}+N_E$, where $N_{\text{sciss}}=N$ is the group of   scissor relations and $N_E$ the group generated by all $a-b$ with $(a,b)\in E$. 

If $\tilde E$ is the equivalence relation generated by $E$ (meaning that $a$ is equivalent to $b$ if there is an undirected path from $a$ to $b$), then $N_E=N_{\tilde E}$, and hence both graphs have the same \grgroup, as do all intermediate graphs between $E$ and $\tilde E$. Therefore,  often, though not always, $E$ will already be an equivalence relation. Although the quotient $\pol\zet\Lambda/N_E$ is equal to the free Abelian group  on the quotient $\Lambda/\tilde E$, the latter is no longer a semi-lattice, and so a priori, does not admit a scissor subgroup.  We may paraphrase this situation as: \emph{cut first, then identify}.

\section{The \gr\ of a theory}\label{s:grth}
Inspired by the ground-breaking work of Denef and Loeser, model-theorists have recently been interested in the \gr\ of an arbitrary first-order theory, see for instance \cite{CH,CHGrot,HuKaz,KraSca}. The  new perspective offered here is that rather than looking at all formulae and all definable isomorphisms, much better behaved objects can be obtained when restricting these classes.

Fix a
language $\mathcal L$, by which we mean the collection of all well-formed
formulae in a certain signature, in a fixed countable collection of variables $v_1,v_2,\dots$.\footnote{We usually start numbering from $1$, as any non-logician would.} Note that some  
authors use the terms language and signature interchangeably. We denote a formula by Greek lower case letters $\phi,\psi,\dots$, and often we give names to their free variables as well, taken from the last letters of the Latin alphabet: $\var,\vary,z$. If $\phi\rij\var s$ is an $\mathcal L$-formula, and $M$ an $\mathcal L$-structure, then the \emph{set defined by $\phi$ in $M$}, or the \emph{interpretation of $\phi$ in $M$} is the following subset $\inter\phi M$. Suppose $x_i=v_{n_i}$, and let $n$ be the maximum of all $n_i$. Then $\inter\phi M$ is the subset of $M^n$ of all $\rij an$ such that $\phi(a_{n_1},\dots,a_{n_s})$ holds in $M$. Any set of the form $\inter\phi M$ will be called a \emph{definable subset}. Note that the Cartesian power $n$ is not determined by the number of free variables $s$, but by the highest index of a  $v_i$ occurring in the formula $\phi$; we   call $n$ the \emph{arity} of $\phi$ (which therefore is not to be confused with its number of free variables). For instance, the subset defined by $\phi:=(v_3=0)\en (v_7=1)$ in $\zet$ is the $7$-ary subset $\zet^2\times\{0\}\times\zet^3\times\{1\}$ of $\zet^7$. Also note that this leaves a certain amount of ambiguity:  the formula $v_3=0$ has, prima facie, arity $3$, but as a conjunct of $\phi$ it behaves as a formula of arity $7$. Notwithstanding all this, the tacit rule will be that if $\rij \var n$ denotes  the tuple of free variables of $\phi$, then $\rij\var n$ stands for the tuple $\rij vn$, or more generally, if $\var,\vary, z,\dots$ are tuples of free variables of $\phi$, which are listed in that order, and whose total number  equals  $n$,   then these variables represent the first $n$ variables $v_i$, that is to say, $(\var_1,\dots,\var_s,\vary_1,\dots,\vary_t,z,\dots)=\rij vn$. Put differently, unless mentioned explicitly, the arity of a formula is its number of free variables. In this respect, it is useful to introduce the the \emph{primary form} of $\phi$, defined as $\phi^\circ:=\phi(v_1,\dots,v_s)$, where $s$ is the number of   free variables of $\phi$. The implicit assumption is that a formula is in primary form, unless   the variables are stated explicitly. Furthermore, the implicit assumption is that the disjunction or conjunction of two formulae is equal to the maximum of their arities.

  This definition also applies to a sentence $\sigma$, that is to say, a formula without free variables: given   an $\mathcal L$-structure $M$, we let $\inter\sigma M$ be one of the two possible   subsets of $M^0:=\{\emptyset\}$, namely $\{\emptyset\}$ if $\sigma$ holds in $M$, and $\emptyset$ if it does not. If $\phi$ and $\psi$ are two $\mathcal L$-formulae, then $\phi\en\psi$ and $\phi\of\psi$ are again $\mathcal L$-formulae, defining in each $\mathcal L$-structure $M$ respectively $\inter\phi M\cap\inter\psi M$ and $\inter\phi M\cup\inter\psi M$. Moreover, $\niet\phi$ defines the complement of $\inter\phi M$ in $M^n$, where $n$ is the arity of $\phi$. A trivial yet important formula is the \emph{$n$-th Lefschetz formula} 
  $$
\lambda_n:=  (v_1=v_1)\en\dots\en (v_n=v_n),
  $$ 
defining the full Cartesian power $M^n$ in any model. We abbreviate the Lefschetz formula $\lambda_n(\var):=(\var=\var)$ with $\var$ an $n$-tuple of variables (which, according to our tacit assumptions, stand  for the $n$ first variables $v_i$), and write $\lambda(\var)$ for $\lambda_1(\var)$.

An $\mathcal L$-\emph{theory} $\theory T$ is any non-empty collection of consistent $\mathcal L$-sentences (it is convenient to assume that $\theory T$ contains at least one sentence, which we always could assume to be a tautology like $\forall \var\lambda(\var)$). A \emph{model} of $\theory T$ is an $\mathcal L$-structure for which all sentences in $\theory T$ are true. By the compactness theorem, every theory has at least one model. Given a (non-empty) collection $\mathfrak K$ of $\mathcal L$-structures, we define the \emph{$\mathcal L$-theory of $\mathfrak K$}, denoted $\theory T_{\mathcal L}(\mathfrak K)$, to be the collection of all $\mathcal L$-sentences that are true in any structure belonging to $\mathfrak K$. The collection $\mathfrak K$ is \emph{axiomatizable} (also called \emph{first-order}), if it consists precisely of the  models of $\theory T_{\mathcal L}(\mathfrak K)$.  
  
Let $\theory T$ be a theory in the language $\mathcal L$.
We say that   two formulae $\phi$ and
$\psi$ are \emph{$\theory
T$-equivalent}, denoted $\phi\sim_{\theory T}\psi$,
if they define the same subset in any model of $\theory T$, that is to say, if
$\inter\phi M=\inter\psi M$, for any model $M$ of $\theory T$. By the compactness
theorem, this is equivalent with $\theory T$ proving that $(\forall
\var)\phi(\var)\asa\psi(\var)$. In particular, equivalent formula must have the same free variables $v_i$.  Note that the logical connectives $\en$ and $\of$, as well as negation $\niet$,  respect this equivalence relation. If $\theory T$ consists entirely of tautologies, then two formulae are $\theory T$-equivalent \iff\ they are logically equivalent. If $\theory T=\theory T_{\mathcal L}(\mathfrak K)$ is the theory of a non-empty class $\mathfrak K$ of $\mathcal L$-language, then $\phi$ and $\psi$ are $\theory T$-equivalent \iff\ they define the same subset in each structure belonging to $\mathfrak K$. In other words, we do not need to check all models, but only those that ``generate'' the theory, and therefore, we will often make no distinction between theories and collections of structures. (Caveat: when dealing with infinitary formulae, as in \S\ref{s:infgr}, this is no longer true.)  For instance,  instead of calling two formulae $\theory T_{\mathcal L}(\mathfrak K)$-equivalent, we may just simply call them \emph{$\mathfrak K$-equivalent}. When the theory $\theory T$ is fixed, we will often identify $\theory T$-equivalent formulae. Formally, we therefore introduce:

\subsection*{The category of definable sets}
 The \emph{category of $\theory T$-definable sets}, $\defcat{\theory T}$, has
as objects the $\theory T$-equivalence classes of formulae and as   morphisms
the definable  maps. There, are in fact a few variant ways of defining the latter notion, and for our purposes,   the following will be most suitable. Let $\phi $ and $\psi $ be $\mathcal L$-formulae of arity $n$ and $m$ respectively, and let   $\var=\rij \var n$ and $\vary=\rij\vary m$ be variables. A formula $\theta(\var,\vary)$ (or, rather, its $\theory T$-equivalence class) is called \emph{morphic} (on $\phi$), or \emph{defines a $\theory T$-morphism $f_\theta\colon\phi\to \psi$}, if  $\theory T$ proves the following sentences
\begin{enumerate}
\item \label{i:map} $(\forall \var)(\exists\vary)\phi(\var)\dan\theta(\var,\vary)$
\item\label{i:fct} $(\forall \var,\vary,\vary')\phi(\var)\en\theta(\var,\vary)\en\theta(\var,\vary')\dan\vary=\vary'$
\item\label{i:im} $(\forall \var,\vary)\phi(\var)\en\theta(\var,\vary)\dan\psi(\var)$.
\end{enumerate}
In other words, \eqref{i:map} and \eqref{i:fct} express that in each model $M$ of $\theory T$, the definable subset $\inter\theta M\sub M^n\times M^m$, when restricted to $\inter\phi M\times M^m$, is the graph of a function $f_M\colon \inter\phi M\to M^m$, and, furthermore, \eqref{i:im} ensures that the image of $f_M$ lies in $\inter\psi M\sub M^m$. We will therefore denote this map by $f_M\colon \inter\phi M\to\inter\psi M$.  Although slightly inaccurate, we will express this situation also by simply saying that $\inter\theta M$ is the graph of a function $\inter\phi M\to\inter\psi M$. As part of the definition of morphic formula is an, often implicit, division of the free variables in two sets, the \emph{source} variables, $\var$, and the \emph{target} variables, $\vary$. The next result shows that these definitions constitute  a category.

\begin{lemma}\label{L:morcomp}
If $f_\theta\colon\phi(\var)\to\psi(\vary)$ and $g_\zeta\colon\psi(\vary)\to \gamma(z)$ are morphisms defined by $\theta(\var,\vary)$ and $\zeta(\vary,z)$ respectively, then the formula $(\zeta\after\theta)(\var,z):=(\exists\vary)\theta(\var,\vary)\en\zeta(\vary,z)$ is again a morphic formula, defining the composition $h\colon\phi(\var)\to\gamma(z)$. 
\end{lemma}
\begin{proof}
The proof is straightforward but technical; since we will rely on it heavily, we will provide it in some detail. Put $\xi:=\zeta\after\theta$. We first verify condition~\eqref{i:map}, and to this end, we may work in a fixed model $M$ of $\theory T$. So, let $\tuple a\in\inter\phi M$. By \eqref{i:map} and \eqref{i:im} applied to $\theta(\var,\vary)$, we get a tuple $\tuple b\in\inter\psi M$ such that $(\tuple a,\tuple b)\in\inter\theta M$. By the same argument, applied to $\zeta$, we then get $\tuple c\in\inter\gamma M$ such that $(\tuple b,\tuple c)\in\inter\zeta M$. In particular, $\tuple b$ witnesses that $(\tuple a,\tuple c)\in\inter\xi M$, proving \eqref{i:map} for $\xi$. The same argument essentially also shows that \eqref{i:im} holds. So remains to show \eqref{i:fct} for $\xi$. To this end, let $\tuple a\in\inter\phi M$ such that $(\tuple a,\tuple c)$ and $(\tuple a,\tuple c')$ both belong to $\inter\xi M$. This means that there are tuples $\tuple b,\tuple b'$ such that $\theta(\tuple a,\tuple b)\en\zeta(\tuple b,\tuple c)$ and $\theta(\tuple a,\tuple b')\en\zeta(\tuple b',\tuple c')$ hold in $M$. By  condition~\eqref{i:fct} for $\theta$, we get $\tuple b=\tuple b'$, and by \eqref{i:im}, this tuple then belongs to $\inter\psi M$. So we may repeat this argument   to the tuples $\tuple b,\tuple c,\tuple c'$ and $\zeta$, to get $\tuple c=\tuple c'$, as we wanted to show.
\end{proof}

We
 call a morphism $f_\theta\colon \phi\to \psi$, or the corresponding formula $\theta$,
\emph{injective}, \emph{surjective}, or \emph{bijective}, if all the
corresponding maps $f_M\colon \inter\phi M\to \inter\psi M$ are. The next result shows that any definable bijection is an  isomorphism in the category $\defcat{\theory T}$, that is to say, its inverse is also a morphism, which we therefore call a  \emph{$\theory T$-isomorphism}.

\begin{lemma}\label{L:morinv}
Let $\theta(\var,\vary)$ be a morphic formula defining a bijection $f_\theta\colon\phi(\var)\to\psi(\vary)$. If $\op{inv}(\theta)(\var,\vary):=\theta(\vary,\var)$, then $\op{inv}(\theta)$ defines a morphism $g\colon \psi\to\phi$, which gives the inverse of $f_M$ on each model $M$ of $\theory T$.
\end{lemma}
\begin{proof}
To show that $\zeta(\var,\vary):=\op{inv}(\theta)(\var,\vary)$ is morphic, yielding a morphism $\psi\to\phi$,   we may again check this in a model $M$ of $\theory T$. Suppose $\tuple a\in\inter\psi M$. Since $f_M$ is bijective, there is some $\tuple b\in\inter\phi M$ such that $f_M(\tuple b)=\tuple a$. This means that $(\tuple b,\tuple a)\in\inter\theta M$, whence $(\tuple a,\tuple b)\in\inter\zeta M$, proving conditions~\eqref{i:map} and \eqref{i:im}. Since $f_M$ is a bijection, the tuple $\tuple b$ is unique, and this proves \eqref{i:fct}. It is now easy to see that the map $g_M\colon\tuple a\mapsto\tuple b$ is the inverse of $f_M$.
\end{proof}

\subsubsection{$\mathcal I$-morphisms}
Given   a
family $\mathcal I\sub \mathcal L$ of formulae (closed under $\theory T$-equivalence), by an
\emph{$\mathcal I$-definable $\theory T$-map}, or simply, an \emph{$\mathcal I$-morphism} between $\phi$ and $\psi$, we mean
a morphic formula $\theta$  belonging to $\mathcal I$ which defines a morphism
$f_\theta\colon\phi\to\psi$ 
(without imposing any restriction on  $\phi$ and $\psi$). We call $f_\theta\colon\phi\to \psi$ an \emph{$\mathcal I$-isomorphism (modulo $\theory T$)}, if $f_\theta$ is bijective and its inverse is again an $\mathcal I$-morphism. In view of Lemma~\ref{L:morinv}, a bijective morphism $f_\theta$ is an $\mathcal I$-isomorphism,  if, for instance, both $\theta$ and $\op{inv}(\theta)$ belong to $\mathcal I$. However, in general, a bijective $\mathcal I$-morphism need not be an $\mathcal I$-isomorphism (see, for instance, Example~\ref{E:nonexpl}).

A note of caution: in general, $\mathcal I$-isomorphism, in spite of its name, is not an equivalence relation, since it is not clear that the composition of two $\mathcal I$-isomorphism is again an $\mathcal I$-morphism. In view of Lemma~\ref{L:morcomp},   the collection of $\mathcal I$-morphisms is closed under composition, if, for instance,    $\mathcal I$ is  closed under existential quantification, although, as we shall see, this is not the only instance in which this is true.

\subsubsection{Explicit formulae}
A morphic formula, in general, only defines a partial map, but there is an important
 type of morphism that is always global: by an \emph{explicit} morphism, we mean a formula  $\theta(\var,\vary)$ of the form $\En_{j=1}^m\vary_j =t_j(\var)$, with each $t_j$ an $\mathcal L$-term, and with source variables $\var=\rij\var n$ and target variables $\vary=\rij\vary m$; we will abbreviate this by $\vary =t(\var)$, for $t:=\rij tm$. Such a formula always defines a global map on each model $M$, given as $t_M\colon M^n\to M^m\colon\tuple a\mapsto t(\tuple a)$. In particular, if $f\colon\phi\to \psi$ is explicit, then $f_M$ is the restriction of $t_M$ to $\inter\phi M$. We denote the collection of all explicit formulae by $\expl$. Note that $\expl$ is compositionally closed, for if $\vary=t(\var)$ and $z=s(\vary)$ are two explicit formulae, then the explicit formula $z=s(t(\var))$ defines their composition. However, as Example~\ref{E:nonexpl} shows, a bijective $\expl$-morphism need not be an $\expl$-isomorphism.  At any rate, any formula $\phi$  is    $\expl$-isomorphic to its primary form $\phi^\circ$.

We already observed that   the logical connectives $\en$ and $\of$ are well-defined on $\defcat{\theory T}$. We introduce two further operations on $\defcat{\theory T}$.

\subsubsection{Multiplication on $\defcat {\theory T}$}\label{s:mult}
Let $\phi$ and $\psi$ be 
two $\mathcal L$-formulae, in $n$ and $m$ free variables respectively (or, more correctly, of arity $n$ and $m$ respectively).
We   define their \emph{product} as
$$
(\phi\times\psi)\rij v{n+m}:={\phi\rij vn\en\psi(v_{n+1},\dots,v_{n+m})}.
$$
A note of caution: it is not always true that $\inter{(\phi\times\psi)} M$ is equal to $\inter\phi M\times \inter\psi M$, in a model $M$ of $\theory T$, due to the numbering of the variables.  This only holds for primary forms since always $\phi\times\psi=\phi^\circ\times\psi^\circ$, and the interpretation of this formula in $M$ is   equal to $\inter{\phi^\circ}M\times\inter{\psi^\circ}M$.  In particular, although  this multiplication is not Abelian, it is up to  explicit isomorphism (given by permuting the variables appropriately). For Lefschetz formulae we obviously have $\lambda_n\times\lambda_m=\lambda_{m+n}$.

We leave it to the reader to verify that if $\phi_i$ and $\psi_i$ are $\theory T$-equivalent, for $i=1,2$, then so are the respective products $\phi_1\times\phi_2$ and $\psi_1\times\psi_2$. In other words, the multiplication is well-defined modulo $\theory T$-equivalence, and hence yields a multiplication on $\defcat{\theory T}$.    The multiplicative unit in $\defcat{\theory T}$ is the class of any sentence $\sigma$ which is a logical consequence of $\theory T$, and will be denoted $\one$ (for instance, one may take $\sigma$ to be the tautology $(\forall\var)\lambda(\var)$). We will also write $\nul$ for the class of $\niet\sigma$, and we have  $\nul\times\phi=\nul=\phi\times\nul$, for all
formulae $\phi$ (note that, per convention, the Cartesian product of any set with the empty set is the empty set).

Given two morphisms  $f_\theta\colon\phi\to\psi$ and
$f_{\theta'}\colon{\phi'}\to{\psi'}$ of $\mathcal L$-formulae, they induce a morphism $f\colon (\phi\times\phi')\to (\psi\times\psi')$ between the respective products as follows. If $\theta(\tuple\var,\tuple\vary)$ and $\theta'(\tuple\var',\tuple\vary')$ are the respective defining formulae, then the order of the variables in the product $\theta\times\theta'$ is by definition   $\tuple\var,\tuple\vary,\tuple\var',\tuple\vary'$. The formula obtained from this product by   changing this order to $\tuple\var,\tuple\var',\tuple\vary,\tuple\vary'$ then defines $f$. Note that the formula defining $f$ is therefore $\expl$-isomorphic with $\theta\times\theta'$: indeed, this isomorphism is given by 
\begin{equation}\label{eq:multexpliso}
\delta(\tuple \var, \tuple\vary, \tuple\var', \tuple\vary,\tuple z_1,\tuple z_2,\tuple z_3,\tuple z_4):=(\tuple z_1=\tuple x) \en (\tuple z_2=\tuple x')  \en (\tuple z_3=\tuple y) \en (\tuple z_4=\tuple y')
\end{equation}
with  $\tuple z_1,\tuple z_2,\tuple z_3,\tuple z_4$  tuples of variables (of the appropriate length).

\subsubsection{Disjoint sum}\label{s:dissum}
To define the second operation on $\defcat{\theory T}$, we need to assume that $\mathcal L$ contains at least two constant symbols which are interpreted in each model of $\theory T$ as different elements. Since in all our applications, $\theory T$ will always be a theory of rings, for which we have the distinct constants $0$ and $1$, we will for simplicity assume that these two constant symbols are denoted $0$ and $1$ (not to be confused with the $\theory T$-equivalence classes $\nul$ and $\one$).
We define the \emph{disjoint sum} of two formulae $\phi$ and $\psi$  as the formula 
$$
\phi\oplus\psi:=\big(\phi\en (v_{n+1}=0)\big)\of\big(\psi\en (v_{n+1}=1)\big),
$$
 where $n$ is the maximum of the arities of $\phi$ and $\psi$ (so that $v_{n+1}$ is the ``next'' free variable).  The disjoint sum is a commutative operation, but without identity element: we only have that  $\phi\oplus \nul$ is $\expl$-isomorphic with $\phi$ via the morphic formula $\vary=\var$ (note that its inverse $\phi\to\phi\oplus\nul$ is given by $(\vary=\var) \en(v_{n+1}=0)$). Unlike  $\of$, disjoint sum is not idempotent: $\phi\oplus\phi$ will in general be different from $\phi$. We will   assume that the (set-theoretic) \emph{disjoint union} $V\sqcup W$ of two sets $V$ and $W$ is defined as the union of $V\times\{0\}$ and $W\times\{1\}$, so that we proved the following characterization of disjoint sum:
 
 \begin{lemma}\label{L:disjsum}
Given two formulae $\phi$ and $\psi$,   their direct sum $\phi\oplus\psi$ is, up to $\theory T$-equivalence, the unique formula $\gamma$ such that 
$$
\inter\gamma M=\inter\phi M\sqcup\inter\psi M,
$$
for all models $M$ of $\theory T$.\qed
\end{lemma}

\subsection*{The \gr\ of a theory}
It is  useful to   construct  \gr{s}  over restricted classes of
formulas, like quantifier free, or pp-formulae. To do this in as general a
setup as possible, fix a language $\mathcal L$ with constant symbols for $0$ and $1$ (so that disjoint sums are defined). By a \emph{sub-semi-lattice}  $\mathcal G$ of $\defcat{\theory T}$, we mean a collection of formulae closed under conjunction. If $\mathcal G$ is also closed under disjunction, we call it a \emph{sublattice}, and if it is moreover  closed under negation, we call it  \emph{Boolean}.\footnote{Although only the restrictions $\mathcal G\cap\mathcal L_n$ are then   Boolean lattices,  this should not cause any confusion.}
 We say that a sub-semi-lattice $\mathcal G$ is \emph{primary} if it contains the Lefschetz formula $\lambda$ and all  formulae of the form $v_i=c$, with $c$ a constant,    and, moreover,   $\phi$ belongs to $\mathcal G$ \iff\ its primary form $\phi^\circ$ does. It follows that $\mathcal G$ is closed under multiplication, and if $\mathcal G$ is a lattice, then it is also closed under  disjoint sums.  

For the remainder of this section, we fix two primary sub-semi-lattices  $\mathcal F\sub \defcat{\theory T}$ (the ``formulae'') and $\mathcal I\sub \defcat{\theory T}$
(the ``isomorphisms'').  We assume, moreover, that $\expl\sub\mathcal I$; and, more often than not, $\mathcal F$ will actually be a lattice. In any case, by \S\ref{s:lat},   we can define the \sciss\ $\grotlat{\mathcal F}$. The $\mathcal I$-isomorphism relation induces a binary relation on $\mathcal F$, which we denote by $\iso_{\mathcal I}$ (strictly speaking, we should consider the equivalence relation generated by this relation, but this does not matter when working with \grgroup{s}).  Recall that the corresponding \grgroup\ $\grotmon{\mathcal F}{\iso_{\mathcal I}}$ is the quotient of $\pol{\zet}{\mathcal F}$ modulo the subgroup $\mathfrak N:=\mathfrak N_{\mathcal I}+\mathfrak N_{\text{sciss}}$, where $\mathfrak N_{\mathcal I}$ is generated by all expressions $\sym \phi-\sym{\phi'}$, for any pair of $\mathcal I$-isomorphic formulae $\phi,\phi'\in\mathcal F$, and $\mathfrak N_{\text{sciss}}$ is generated, in the lattice case, by all  second scissor relations $\sym{\phi\of\psi}-\sym\phi-\sym\psi+ \sym{\phi\en\psi}$, for any pair $\phi,\psi\in\mathcal F$, and where for clarity, we have written $\sym\phi$ for the $\theory T$-equivalence class of $\phi$ (although we will continue our practice of confusing $\theory T$-equivalence classes with their representatives). If $\mathcal F$ is only a semi-lattice, then $\mathfrak N_{\text{sciss}}$ is generated by all    scissor relations $\sym{\psi}-\sym{S_n(\phi_1,\dots,\phi_n)}$, for all   $\phi_i\in\mathcal F$ such that $\sym\psi=\sym{\phi_1\of\dots\of\phi_n}$.

\begin{lemma}\label{L:scissid}
The subgroup $\mathfrak N$ is a two-sided ideal in $\pol{\zet}{\mathcal F}$ with respect to the multiplication $\times$ on formulae, and the quotient ring is commutative. 
\end{lemma}
\begin{proof}
Note that in \S\ref{s:lat} we used the multiplication given by $\en$ to define scissor relations, whereas here multiplication is as defined in  \S\ref{s:mult}, and is not commutative. 
Let $\alpha,\phi,\psi$ be formulae in $\mathcal F$. If $\phi\iso_{\mathcal I}\psi$, 
then $\phi\times\alpha\iso_{\mathcal I}\psi\times\alpha$ by the discussion of \eqref{eq:multexpliso}, showing that $\sym\alpha\times(\sym\phi-\sym\psi)\in\mathfrak N_{\mathcal I}$, and a similar result for multiplication from the right. Hence $\mathfrak N_{\mathcal I}$   is a two-sided ideal.   Moreover, since $\alpha\times\phi$ is explicitly isomorphic, whence $\mathcal I$-isomorphic, to $\phi\times\alpha$, both formulae have the same image modulo $\mathfrak N_{\mathcal I}$, showing that multiplication is commutative in the quotient.

For simplicity, we only give the proof that $\mathfrak N$ is an ideal in the lattice case, and leave the semi-lattice case to the reader.  It suffices to show    that any multiple of 
$$
u:=\sym{\phi\of\psi}-\sym\phi-\sym\psi+ \sym{\phi\en\psi}
$$
 lies in $\mathfrak N$.   Choose $\alpha'$  to be $\mathcal I$-isomorphic to $\alpha$, but having free variables distinct from those of $\phi$ and $\psi$ (this can always be accomplished with an explicit change of variables). It follows from what we just proved that $\sym\alpha\times u-\sym{\alpha'}\times u$ belongs to $\mathfrak N$, and so we may assume from the start  that $\alpha$ has no free variables in common with $\phi$ and $\psi$. In particular,   $\sym\alpha \times\sym\phi=\sym{\phi\en\alpha}$, and similarly for any other term in $u$. It follows that 
 \begin{equation}\label{eq:multscissor}
\sym\alpha\times u= \sym{(\phi\en\alpha)\of(\psi\en\alpha)}-\sym{\phi \en\alpha}-\sym{\psi \en\alpha }+ \sym{(\phi \en\alpha)\en(\psi \en\alpha)}
\end{equation}
which is none other than   the second scissor relation  on the two formulae $\phi\en\alpha$ and $\psi\en\alpha$, and therefore, by definition, lies in $\mathfrak N$.
\end{proof}
 
Note that $\mathfrak N_{\text{sciss}}$ is in general not an ideal.
In any case,  $\grotmon{\mathcal F}{\iso_{\mathcal I}}$ has the structure of a commutative ring, which we will call the  \emph{$\mathcal I$-\gr\ of $\mathcal F$-formulae modulo $\theory
T$} and which we denote, for simplicity,  by $\grotmod{\theory T}{\mathcal F}{\mathcal I}$, or just $\grotmod{}{\mathcal F}{{\mathcal I}}$, if the underlying theory $\theory T$ is understood. 
 We denote the class of a
formula $\phi$ in $\grotmod{\theory T}{\mathcal F}{\mathcal
I}$ by $\class\phi$, or in case we want to emphasize the isomorphism type, by 
$\class\phi_{\mathcal I}$. We denote the class of $\nul$ and $\one$   by $0 $ and $1$ respectively; they are the   neutral elements for   addition and multiplication in $\grotmod{\theory T}{\mathcal F}{\mathcal I}$ respectively. Note that $1$ is equal to the class of $v_1=0$ defining a singleton. 

The \emph{full \gr} of a theory $\theory T$ is obtained by taking for $\mathcal F$ and $\mathcal I$ simply all formulae. It will be denoted $\grot{\theory T}$.
Immediately from the definitions, we have:

\begin{corollary}\label{C:addmap}
For each   pair $\mathcal F,\mathcal I\sub\defcat{\theory T}$ as above, there exists a canonical additive epimorphism of groups $\grotlat{{\mathcal F}}\onto \grotmod{\theory T}{\mathcal F}{\mathcal I}$. 

Moreover,
if $\theory T'$ is a subtheory of $\theory T$, and $\expl\sub\mathcal I'\sub\mathcal I$ and $\mathcal F'\sub\mathcal F$      primary sub-semi-lattices , then we have a natural \homo\ of \gr{s} $\grotmod{\theory T'}{\mathcal F'}{\mathcal I'}\to\grotmod{\theory T}{\mathcal F}{\mathcal I}  $.\qed
\end{corollary}

Note that even if $\theory T= \theory T'$ and $\mathcal I=\mathcal I'$, the latter \homo\ need not be injective, as there are potentially more relations when the class of formulae is larger.

\subsubsection*{The Lefschetz class}
We denote the class of the first Lefschetz formula $\lambda:=(v_1=v_1)$  by $\lef$ (recall that by assumption $\lambda$ belongs to $\mathcal F$), and call it the \emph{Lefschetz class} of $\grotmod{\theory T}{\mathcal F}{\mathcal I}$.  By definition of product, we immediately get:

\begin{lemma}\label{L:Lef}
For every $n$, we have  $\class{\lambda_n}=\lef^n$ in $\grotmod{\theory T}{\mathcal F}{\mathcal I}$.\qed
\end{lemma}

\begin{lemma}\label{L:negation}
If $\mathcal F$ is a lattice, then $\class{\phi\oplus\psi}=\class\phi+\class\psi$ in $\grotmod{\theory T}{\mathcal F}{\mathcal I}$, for all $\phi,\psi\in{\mathcal F}$. If $\mathcal F$ is moreover   Boolean,  then$\class{\niet\phi}=\lef^n-\class\phi$, where $n$ is the arity of $\phi$.
\end{lemma}
\begin{proof}
To prove the first assertion, let $\phi'(\var,z):=\phi(\var)\en (z=0)$ and $\psi'(\var,z):=\psi(\var)\en(z=1)$. By definition, and after possibly taking primary forms, we have $\class{\phi'\of\psi'}=\class{\phi\oplus\psi}$. The claim now follows since $\phi'\en\psi'$ defines the empty set in any model $M$ of $\theory T$, and hence its class is zero.
To prove the second,  observe that  $\sym{\phi\en\niet\phi}=0$, whereas $\phi\of\niet\phi$ is $\theory T$-equivalent with $\lambda_n$. Hence the result follows by Lemma~\ref{L:Lef}.\end{proof}

\begin{lemma}\label{L:incompscis}
If $\mathcal F$ is Boolean, then the  image of the  ideal $\mathfrak N_{\text{sciss}}$ in $Z:=\pol\zet{\mathcal F}/\mathfrak N_{\mathcal I}$
is generated as a group by all formal sums  of the form
$\sym\phi+\sym\psi-\sym{\phi\oplus\psi}$, with $\phi$ and $\psi$
formulae in ${\mathcal F}$.
\end{lemma}
\begin{proof}
Let $\mathfrak N'$ be the set of all formal sums in $Z$ of scissor relations
$\sym\alpha+\sym\beta-\sym{\alpha\oplus\beta}$, with $\alpha$ and
$\beta$   in ${\mathcal F}$. Let $\phi$ and $\psi$ be two formulae
in ${\mathcal F}$. Then the scissor relation
$\sym\phi+\sym\psi-\sym{\phi\en\psi}-\sym{\phi\of\psi}$ is equal to the
difference
$$
\Big(\sym\phi+\sym{\niet\phi\en\psi}-\sym{\phi\of\psi}
\Big)-\Big(\sym{\niet\phi\en\psi}+\sym{\phi\en\psi}-\sym\psi\Big)
$$
whence belongs to $\mathfrak N' $, since $\phi\of \psi$ is $\mathcal I$-isomorphic with $\phi\oplus(\niet\phi\en\psi)$. Using \eqref{eq:multscissor}, it is clear that   $\mathfrak N'$ is closed under multiples, whence is an ideal of $Z$.
\end{proof}

In particular, $\mathfrak N$ is generated as a group in $\pol\zet{\mathcal F}$ by $\mathfrak N_{\mathcal I}$ and all formal sums $\sym\phi+\sym\psi-\sym{\phi\oplus\psi}$, for $\phi,\psi\in\mathcal F$.

\begin{corollary}\label{C:grgen}
Suppose $\mathcal F$ is a  Boolean lattice. If $\mathcal G\sub\mathcal F$ is a sub-semi-lattice whose Boolean closure   is equal to $\mathcal F$, then the natural map   $\grotmod{\theory T}{\mathcal G}{\mathcal I}\to \grotmod{\theory T}{\mathcal F}{\mathcal I}$ is surjective, and hence $\grotmod{\theory T}{\mathcal F}{\mathcal I}$  is generated as a group by classes of formulae in $\mathcal G$. If $\mathcal G$ is moreover closed under disjoint sums, then any element in $\grotmod{\theory T}{\mathcal F}{\mathcal I}$ is of the form $\class\psi-\class{\psi'}$, with $\psi,\psi'\in\mathcal G$.
\end{corollary}
\begin{proof}
Let $G$ be the image of $\grotmod{\theory T}{\mathcal G}{\mathcal I}\to \grotmod{\theory T}{\mathcal F}{\mathcal I}$ (see Corollary~\ref{C:addmap}), that is to say, the subgroup generated by all $\class\phi$ with $\phi\in\mathcal G$. 
By Lemma~\ref{L:negation}, the class of the negation of a formula in $\mathcal G$  lies in $G$. Since every term in $\class{S_n\rij\phi n}$, for $\phi_i\in\mathcal G$, lies by assumption in $G$, so does the class of any disjunction $\phi_1\of\dots\of\phi_n$ by Proposition~\ref{P:scissor} and Corollary~\ref{C:addmap}. This proves the first assertion.

To prove the last, we can write, by what we just proved,  any element $u$ of $\grotmod{\theory T}{\mathcal F}{\mathcal I}$  as a sum $\class{\phi_1}+\dots+\class{\phi_s}-\class{\phi_{s+1}}-\dots-\class{\phi_t}$, with $\phi_i\in\mathcal G$. Let $\psi$ be the disjoint sum $\phi_1\oplus\dots\oplus\phi_s$ and let $\psi'$ be the disjoint sum $\phi_{s+1}\oplus\dots\oplus\phi_t$. By assumption, both $\psi$ and $\psi'$ lie in $\mathcal G$, and $u=\class\psi-\class{\psi'}$ by Lemma~\ref{L:incompscis}.
\end{proof}

  We
call two ${\mathcal F}$-formulae
$\phi_1$ and $\phi_2$    \emph{stably $\mathcal I$-isomorphic in
${\mathcal F}$ (modulo $\theory T$)},
if there exists
  an ${\mathcal F}$-formula $\psi$   such that  $\phi_1 \oplus\psi\iso_{\mathcal
I}\phi_2 \oplus\psi$. A priori, this is weaker than being $\mathcal I$-isomorphic, but in many cases, as we shall see, it is equivalent to it. In any case, we have:

\begin{lemma}\label{L:stabiso}
Suppose $\mathcal F$ is Boolean. If $\phi_1 \oplus\psi\iso_{\mathcal
I}\phi_2 \oplus\psi$ and $\psi\dan\psi'$ modulo $\theory T$, then $\phi_1 \oplus\psi'\iso_{\mathcal
I}\phi_2 \oplus\psi'$.
\end{lemma}
\begin{proof}
By adding a disjoint copy  to either side, we get 
$$
\phi_1 \oplus\psi\oplus(\psi'\en\niet\psi)\iso_{\mathcal I}\phi_2 \oplus\psi\oplus(\psi'\en\niet\psi).
$$
Since $\psi\dan\psi'$, the formulae $\psi'$ and  $\psi\oplus(\psi'\en\niet\psi)$ are $\mathcal I$-isomorphic.
\end{proof}

\begin{theorem}\label{T:stabiso}
Suppose $\mathcal F$ is Boolean, and $\iso_{\mathcal I}$ is an equivalence relation. Two formulae $\phi,\psi\in{\mathcal F}$ have the same class in $\grotmod{\theory T}{\mathcal F}{\mathcal I}$ \iff\
they are stably $\mathcal I$-isomorphic in ${\mathcal F}$.
\end{theorem}
\begin{proof}
One direction is easy, so that we only need to verify the direct implication. Let $\bar{\mathcal F}$ be the quotient of $\mathcal F$ modulo the equivalence $\iso_{\mathcal I}$, that is to say, the collection of $\mathcal I$-isomorphism classes of formulae in $\mathcal F$. Let 
$Z:=\pol\zet{\mathcal F}/\mathfrak N_{\mathcal I}$, so that $\grotmod{\theory T}{\mathcal F}{\mathcal I}\iso Z/\mathfrak N_{\text{sciss}}Z$.  Moreover, the quotient map $\mathcal F\to\bar{\mathcal F}$ induces an isomorphism $Z\iso \pol\zet{\bar{ \mathcal F}}$, so that as an Abelian group, $Z$ is freely generated.
By Lemma~\ref{L:incompscis}, we can find  ${\mathcal F}$-formulae $\phi_i$, $\phi_i'$,   $\psi_i$, and $\psi_i'$   (without loss of generality we may assume that their number is the same), 
such that  
\begin{equation}\label{eq:sumincomp}
\sym\phi+\sum_i\sym{\phi_i}+\sym{\phi_i'}+\sym{\psi_i \oplus\psi_i'}= \sym\psi
+\sum_i\sym{\psi_i}+\sym{\psi_i'}+\sym{\phi_i \oplus\phi_i'}
\end{equation}
in $Z$.
Let $\sigma$ be the disjoint sum of all the     formulae
$\phi_i,\phi_i',\psi_i$ and $\psi_i'$. Since $Z$ is freely generated every formula in the left hand side of equation~\eqref{eq:sumincomp} also appears   in its right hand side, and vice versa. Hence both formulae 
$\phi \oplus\sigma$ and $\psi \oplus\sigma$ yield the same class in $Z\iso\pol\zet{\bar{\mathcal F}}$, whence must be $\mathcal I$-isomorphic.
\end{proof}

\section{Affine schemes of finite type}
All schemes are assumed to be Noetherian, even if we do not always mention this. Let $X=\op{Spec}A$ be an affine Noetherian scheme. By an \emph{$X$-scheme}, we mean a separated scheme $Y$ together with a morphism of finite type $Y\to X$. Hence, affine $X$-schemes   are in one-one correspondence  with finitely generated $A$-algebras. 
We call a scheme $Y$ a \emph{variety} if it is reduced (but not necessarily irreducible).

Fix an affine Noetherian scheme $X=\op{Spec}A$. Let $\mathcal L_A$ be the language of $A$-algebras
in the signature consisting of two binary operations $+$ and $\cdot$, plus
constant symbols for each element in $A$. A formula in this language will simply be  called an
\emph{$A$-algebra formula}. By a \emph{\zariski}  formula in $\mathcal L_A$, we
mean a
finite conjunction of   formulae $f(\var)=0$, where $f\in \pol A\var$ and
$\var$ is a finite tuple of indeterminates. We denote the collection of all
\zariski\ formulae by $\zar$.   
Let $\theory T_A$ be the theory of $A$-algebras, that is to say, the theory
whose models are the $\mathcal L_A$-structures that carry the structure of an
$A$-algebra. We also consider some   extensions of this theory. 
As we shall see, and as would not be surprising to an algebraic geometer, it suffices to work with  local rings. Being local is a first-order property as it is equivalence with the statement that the sum of any two non-units is  again a non-unit, and hence $\loctheory A$ is the theory $\theory T_A$ to which we adjoin the    first-order sentence
$$
(\forall x,y) \Big((\forall a) (ax\neq 1) \en (\forall b)(by\neq 1)\dan (\forall c)((x+y)c\neq 1)\Big).
$$ 
But not only can we restrict to local rings, we may restrict our theory to zero-dimensional algebras.   More precisely, let $\arttheory A$ be the class of all  local $A$-algebras of finite length as an $A$-module, for short, the  \emph{local $A$-Artinian algebras}. Unfortunately, $\arttheory A$  is not elementary, and hence its theory will have models that are not Artinian local rings. Nonetheless, as observed earlier, whenever we have to verify an equivalence or an isomorphism modulo this theory, it suffices to check this on the rings in $\arttheory A$.   Finally, the ``classical'' theory is recovered from looking at the theory $\ACF A$, consisting of all \acf{s}  that are $A$-algebras. Instead of writing $\mathcal L_A$ and  $\arttheory A$,   we also may write $\mathcal L_X$ or $\arttheory X$, when we take a more  geometrical point of view. Similarly, given an $A$-algebra $B$ and a \zariski\ formula $\phi$, we call $\inter\phi B$ sometimes the definable subset of $Y:=\op{Spec}B$ given by $\phi$, an denote it  $\phi(Y)$.

\begin{theorem}\label{T:fosch}
Let $A$ be a Noetherian ring, $\var$ an $n$-tuple of indeterminates, and
$\affine An$ the affine scheme $\op{Spec}(\pol Ax)$. 
There is a one-one correspondence between the following three sets:
\begin{enumerate}
\item\label{i:gf} 
  the  set  of $\theory T_A$-equivalence classes of 
\zariski\  formulae of arity  $n$;
\item\label{i:Art} 
  the  set  of $\arttheory A$-equivalence classes of 
\zariski\  formulae of arity  $n$;
\item\label{i:id}  the set  of all ideals in
$\pol A\var$;
\item\label{i:cs} the set  of all closed
subschemes of $\affine An$.
\end{enumerate}
\end{theorem}
\begin{proof}
The one-one correspondence between \eqref{i:id} and \eqref{i:cs} is of course
well-known: to an ideal $I\sub \pol A\var$ one associates the closed subscheme
$\op{Spec}(\pol A\var/I)$. Let $\phi$ be the \zariski\ formula
$f_(\var)=\dots=f_s(\var)=0$, with $f_i\in\pol A\var$, and let $I(\phi):=\rij
fs\pol A\var$. Suppose $\psi$ is another \zariski\ formula in the free
variables $\var$  which is
$\theory T_A$-equivalent to $\phi$.
We need to show that $I(\phi)=I(\psi)$. Let $C_\phi:=\pol A\var/I(\psi)$. Since
$\var$ satisfies $\psi$ in the $A$-algebra $C_\psi$, it must also satisfy
$\phi$ in $C_\psi$, by   definition of equivalence. This means that
all $f_i(\var)$ are zero in $C_\psi$, that is to say, $f_i\in I(\psi)$. Hence
$I(\phi)\sub I(\psi)$. Reversing the argument then shows that both ideals are
equal. Conversely, if both ideals are the same, then writing the $f_i$ in terms
of the generators of $I(\psi)$ shows that any solution to $\psi$, in any
$A$-algebra, is also a solution to $\phi$, and vice versa. Hence the
two formulas are equivalent.

So remains to show the equivalence of   \eqref{i:Art} with the remaining conditions. One direction is trivial, so assume $\phi$ and $\psi$ are $\arttheory A$-equivalent, that is to say, $\inter\phi R= \inter\psi R$ for all   local $A$-Artinian algebras $R$. Towards a contradiction, assume $I:=I(\phi)$ and $J:=I(\psi)$ are   different ideals in $B:=\pol A\var$. Hence, there exists a maximal ideal $\maxim\sub B$ such that $IB_\maxim\neq JB_\maxim$.  Moreover,  by Krull's Intersection Theorem,  there exists an $n$ such that $IR\neq JR$,  where $R:=B_\maxim/\maxim^nB_\maxim=B/\maxim^n$.  Let $a$ be the image of the $n$-tuple   $\var$ in $R$. We have $a\in \inter\phi{R/IR}$, since each $f_i(a)=0$ in $R/IR$. Since $R/IR$  has finite length as an $A$-module,    $\inter\phi {R/IR}= \inter\psi {R/IR}$. Hence, $g(a)=0$ in $R/IR$, for any $g\in J$, showing that $g\in IR$ whence $JR\sub IR$. Switching the role of $I$ and $J$, the latter inclusion is in fact an equality, contradicting the choice of $R$. 
\end{proof}

From the proof of Theorem~\ref{T:fosch}, we see that the ideal $I(\phi)$ associated to a \zariski\ formula $\phi$ only depends on the $\arttheory A$-equivalence class of $\phi$. We denote the affine scheme corresponding to $\phi$ by  $Y_\phi$, that is to say, $Y_\phi:=\op{Spec}(\pol A\var/I(\phi))$.

\subsubsection{Base change}\label{s:bc}
Let $A'$ be an $A$-algebra, that is to say, a morphism $X':=\op{Spec}(A')\to X:=\op{Spec}(A)$. We may assign to each $A$-algebra $B$ its scalar extension $A'\tensor_AB$, or in terms of affine schemes, $Y=\op{Spec}(B)$ yields by base change the affine $X'$-scheme $X'\times_XY$. In terms of  formulae, if $\phi$ is the \zariski\ $\mathcal L_A$-formula defining $Y$, then we may view $\phi$ also as an $\mathcal L_{A'}$-formula. As such, it defines the base change $X'\times_XY$.

Under the one-one correspondence of Theorem~\ref{T:fosch}, \zariski\ sentences correspond to ideals of $A$. More precisely, if $\sigma$ is the \zariski\ sentence $a_1=\dots=a_s=0$ with $a_i\in A$, and $\id:=\rij asA$ the corresponding ideal, then in a model $C$ of $\theory T_A$, the interpretation of $\sigma$ is either the empty set, in case $\id C\neq 0$, or the singleton $\{\emptyset\}$, in case $\id C=0$.

Note that the previous result is false for non-\zariski\ formulae. For instance, if $\phi(\var):=(\var^2=\var)$ and $\psi(\var):=(\var=0)\of(\var=1)$, then $\phi\dan\psi$ in  the theory $\arttheory A$  (since local rings only have trivial idempotents), but not in $\theory T_A$ (take, for instance, as model $C:=\pol A\var/(\var^2-\var)\pol A\var$). In fact, the \zariski\ formula $\vary=\var$ defines a morphism $\phi\to\psi$ in $\arttheory A$, but not in $\theory T_A$ (take again $C$ as the model; see also Example~\ref{E:disjointsum} below). 

Unless stated explicitly otherwise, we will from now on assume that the underlying theory  is $\arttheory A$. 

\subsubsection*{Disjoint sums and unions}
Although we have the general construction of a disjoint sum $\oplus$ of two \zariski\ formula,   the result is no longer a \zariski\ formula. To this end, we define the \emph{disjoint union} of two \zariski\ formulae $\phi$ and $\psi$ as follows. Let $n$ be the maximum of the arities of $\phi$ and $\psi$, and put $z:=v_{n+1}$. Let $I(\phi)$ and $I(\psi)$ be the respective ideals of $\phi$ and $\psi$ in $\pol\zet\var$, with $\var=\rij\var n$, and put 
$$
\id:=((1-z)I(\phi), zI(\psi), z(z-1))\pol\zet{\var,z}.
$$
Then the \emph{disjoint union} of $\phi$ and $\psi$ is the \zariski\ formula $\phi\sqcup\psi$ given by the ideal $\id$, that is to say, the conjunction of all equations $(z-1)f=0$ and $zg=0$, with $f\in I(\phi)$ and $g\in I(\psi)$, together with $z(z-1)=0$.

\begin{lemma}\label{L:schdisjun}
Given     \zariski\ formulae $\phi$ and $\psi$, their disjoint union $\phi\sqcup\psi$ is   $\arttheory A$-equivalent to their disjoint sum $\phi\oplus\psi$.
\end{lemma}
\begin{proof}
It suffices to show  that both formulae define the same subset in any   local $A$-Artinian algebra $R$. Let $\phi'$ and $\psi'$ be the formulae $\phi(\var)\en (z=0)$ and $\psi(\var)\en(z=1)$ respectively, where $(\var,z)=\rij v{n+1}$ as above. In particular, $\phi\oplus\psi=\phi'\of\psi'$. 
Suppose $( a,c)\in\inter{(\phi\sqcup\psi)} R$. Since  $R$ is local, its only idempotents are $0$ and $1$. Hence, since $c^2=c$, we may assume, upon reversing the role of $\phi$ and $\psi$ if necessary, that $c=0$. In particular, $(1-c)f( a)=f( a)=0$ in $R$, for all $f\in I(\phi)$,  showing that $ a\in\inter\phi R$, whence $( a,c)\in\inter{(\phi'\of\psi')}R$. Conversely, suppose $(\phi'\of\psi')( a,c)$ holds in $R$, so that one of the disjuncts is true in $R$, say, $\phi'( a,c)$. In particular, $0=c=c(c-1)$ and $f( a)=0$ for all $f\in I(\phi)$, showing that $( a,c)\in\inter{(\phi\sqcup\psi)} R$.
\end{proof}

\begin{example}\label{E:disjointsum}
All we needed from the models of $\arttheory A$   was that they were local. However, the result is false in $\theory T_A$: for instance, let $\lambda$ be the \zariski\ (Lefschetz) formula corresponding to the zero ideal in $\pol A\var$, with $\var$ a single variable. Then $\lambda\oplus \lambda $ is the formula $(z=0)\of(z=1)$, whereas $\lambda\sqcup \lambda $ is the formula $z^2-z=0$ (with the usual primary form assumption that $x=v_1$ and $z=v_2$). However, in the $A$-algebra $C:=\pol At/(t^2-t)\pol At$ these formulae define different subsets: $\inter{(\lambda\oplus \lambda)}C$ is the subset $C\times\{0,1\}$, whereas $\inter{(\lambda\sqcup \lambda)}C$ contains   $(0,t)$.
\end{example}

 Immediately from    Lemma~\ref{L:schdisjun}, we get:

\begin{lemma}\label{L:disunion}
If $\phi$ and $\psi$ are \zariski\ formulae with corresponding affine schemes $Y_\phi$ and $Y_\psi$, then $\phi\sqcup\psi$ corresponds to the disjoint union $Y_\phi\sqcup Y_\psi$.  \qed
\end{lemma}

\begin{lemma}\label{L:ratpt}
For $X:=\op{Spec}A$   an affine Noetherian scheme, $\phi$ a \zariski\  formula defining an affine $X$-scheme $Y:=Y_\phi$, and    $B$    an $A$-algebra, the definable subset    $\inter\phi B$ is in one-one correspondence with $\mor X{\op{Spec}B}{Y}$,  the set of \emph{$B$-rational points on $Y$ over $X$}.
\end{lemma}
\begin{proof}
Put $Z:=\op{Spec}B$. Recall that $\mor XZ{Y}$ consists of all morphisms  of schemes $Z\to Y$ over $X$. Such a map  is uniquely determined by an $A$-algebra \homo\ $\pol A\var/I(\phi)\to B$, and this in turn, is uniquely determined by the image $ b$ of $\var$ in $B$. Since $ b$ is therefore a solution of all $f\in I(\phi)$, it satisfies $\phi$. Conversely, any tuple $ b\in \inter\phi B$ induces an $A$-algebra \homo\ $\pol A\var/I(\phi)\to B$.
\end{proof}

In fact, we may view $\phi$ as a functor on the category of $A$-algebras and as such it agrees with the functor represented by the scheme $Y_\phi$ over $X$. More precisely,  if $B\to C$ is an $A$-algebra \homo\, then the induced map $B^n\to C^n$ on the Cartesian products  maps $\inter\phi B$ into $\inter\phi C$. On the other hand, the associated map of schemes $\op{Spec}C\to \op{Spec}B$ induces, by composition, a map $\mor X{\op{Spec}B}{Y_\phi}\to \mor X{\op{Spec}C}{Y_\phi}$. Under the one-one correspondence in Lemma~\ref{L:ratpt}, we get a commutative diagram
\commdiagram {\inter\phi B}{} {\inter\phi C}{}{}{\mor X{\op{Spec}B}{Y_\phi}}{}{\mor X{\op{Spec}C}{Y_\phi}.}

\subsection*{\Zariski\ morphisms}
By   definition, a $\zar$-morphism $f\colon \phi\to \psi$  between   $\mathcal L_A$-formulae $\phi(\var)$ and $\psi(\vary)$   is given by a \zariski\ formula $\theta(\var,\vary)$, such that   $\inter\theta B$  is the graph of a function $f_B\colon \inter\phi B\to \inter\psi
B$,  for every $A$-algebra $B$. In particular, $\theta$ is explicit, that is to say, belongs to  $\expl$, if it is of the form $\En_{i=1}^m\vary_i=p_i(\var)$, with $p_i\in\pol A\var$. For each $A$-algebra $B$, we denote the global map defined by this explicit formula by $p_B\colon \affine Bn\to\affine Bm$, where $p=\rij pm$. In particular,  $f_B\colon \inter\phi B\to \inter\psi B$ is then the restriction of $p_B$.

\begin{proposition}\label{P:defmap}
Every \zariski\ morphism between \zariski\ formulae is explicit.  
\end{proposition}
\begin{proof}
Let us first prove this modulo $\theory T_A$. Let $f_\theta\colon\phi\to\psi$ be a \zariski\ morphism between the \zariski\ formulae $\phi(\var)$ and $\psi(\vary)$, with $\var=\rij\var n$ and $\vary=\rij\vary m$. Let $C:=\pol A\var/I(\phi)$, where  $I(\phi)$ is the ideal associated to $\phi$
under the equivalence given in Theorem~\ref{T:fosch}, that is to say, the
ideal generated by the equations that make out the \zariski\ formula $\phi$.
Let $a$ denote
the $n$-tuple in $C$ given by the image of the variables $\var$ in $C$. In particular,
$a\in\inter\phi C$, and hence $f_C(a)\in\inter\psi C$. Let $p$ 
be a tuple of polynomials in $\pol A\var$  whose image is the tuple $f_C(a)$ in $C$. Suppose
$\theta$ is the conjunction of formulae $h_i(\var,\vary)=0$, with $h_i\in\pol
A{\var,\vary}$, for $i=\range 1s$,  and $\psi$ is the conjunction of formulae $g_j(\vary)=0$ with
$g_j\in\pol A\vary$, for $j=\range 1t$. Put $\tilde h_i(\var):=h_i(\var, p(\var))$
and  $\tilde g_j(\var):=g_j( p(\var))$.   Since both $\theta(a, p(a))$ and $\psi( p(a))$ hold in $C$, all  $\tilde
h_i$ and $\tilde g_j$ belong to $I(\phi)$. 

Therefore, if $B$ is an arbitrary $A$-algebra, and $ b\in\inter\phi
B$, then all $\tilde h_i( b)$ and $\tilde g_j( b)$ are
zero.
In particular, $ p( b)$ lies in $\inter\psi B$ and $\theta(
b, p( b))$ holds in $B$. By condition~\eqref{i:fct} in the definition of a morphism, $f_B( b)$ must be equal to $ p( b)$. Hence we showed that $\theta(\var,\vary)$ is isomorphic to the explicit formula $\gamma(\var,\vary)$ given as the conjunction of all $\vary_i=p_i(\var)$. By Theorem~\ref{T:fosch}, the two \zariski\ formulae $\theta$ and $\gamma$ being $\theory T_A$-equivalent, are then also $\arttheory A$-equivalent, showing that $f$ is   explicit modulo the latter theory.
\end{proof}

\begin{example}\label{E:nonexpl}
Not every \zariski\ morphism is explicit. For instance, let $\phi(\var_1,\var_2)$ be the formula $\var_1\var_2=1$ and let $\psi(\var_1):=(\exists \var_2)\phi(\var_1,\var_2)$. Then the formula $(\var_1=\vary_1)\en (\var_1\vary_2=1)$ is morphic and yields a \zariski\ morphism $f\colon \psi(\var_1)\to \phi(\vary_1,\vary_2)$, but this is not explicit, since $1/\var_1$ is not a term (polynomial). Put differently, for any   $A$-algebra $B$, the map $f_B\colon\psi(B)\to\inter\phi B\colon b\mapsto (b,1/b)$ is not induced by any total map $\affine B1\to\affine B2$. Note that $f$  is even a bijection, and its inverse $\phi(\var_1,\var_2)\to\psi(\vary_1)$ is given by the explicit formula $\theta(\var_1,\var_2,\vary_1):=\sym{\vary_1=\var_1}$. In particular, the latter explicit bijection is not an $\expl$-isomorphism, since $\op{inv}(\theta)(\var_1,\vary_1,\vary_2)=\sym{y_1=x_1}$ is  not explicit, as it does not contain a conjunct of the form $y_2=g(\var_1)$ with $g$ a polynomial (as pointed out, $1/\var_1$ is not a term). By Lemma~\ref{L:morinv}, however, $f$ is a $\zar$-isomorphism.
\end{example}

\begin{corollary}\label{C:zarmor}
There is a one-one correspondence between \zariski\ morphisms $\phi\to \psi$ modulo $\arttheory  A$ and morphisms $Y_\phi\to Y_\psi$ of schemes over $X$.
\end{corollary}
\begin{proof}
If $\phi\to \psi$ is a \zariski\ morphism, then by Proposition~\ref{P:defmap}, there exists a tuple of polynomials $p$ such that $\inter\phi B\to \inter\psi B$ is given by the base change of the polynomial map $p\colon\affine An\to\affine Am$. By construction, the restriction of $p$ to the closed subscheme $Y_\phi$ of $\affine An$ has image inside the closed subscheme $Y_\psi\sub\affine Am$, and hence induces a morphism $Y_\phi\to Y_\psi$ over $X$. Conversely, any morphism $Y_\phi\to Y_\psi$ is easily seen to induce a morphism $\phi\to \psi$ under the one-one correspondence given by Lemma~\ref{L:ratpt}.
\end{proof}

Proposition~\ref{P:defmap}  shows that $\zar$ is compositionally closed relative to $\zar$.

\begin{corollary}\label{C:isosch}
For each Noetherian ring $A$, there is a one-one correspondence between $\zar$-isomorphism classes of \zariski\ formulae modulo $\arttheory  A$, and   isomorphism classes of
affine schemes of finite type over $A$.
\end{corollary}
\begin{proof}
By  Theorem~\ref{T:fosch},  a \zariski\ formula $\phi$ corresponds to a scheme $Y_\phi$ of finite
type over $X:=\op{Spec}A$, and by Corollary~\ref{C:zarmor}, a \zariski\ definable
map $\phi\to\psi$  induces a morphism $Y_\phi\to Y_\psi$ of schemes over $X$. Using this correspondence, one checks that an isomorphism of formulae corresponds to an isomorphism of schemes over $X$.
\end{proof}

Recall that a \emph{pp-formula} (\emph{positive primitive formula}) is the projection of a \zariski\ formula, that is to say,  a formula of the form $(\exists\vary)\psi(\var,\vary)$, where $\psi(\var,\vary)$ is a \zariski\ formula. To emphasize that pp-formulae represent  projections, we will also denote a pp-formula as 
$$
\fim\psi(\var):=(\exists\vary)\psi(\var,\vary).
$$
 It follows from Lemma~\ref{L:morcomp} that the collection of all pp-formulae is compositionally closed.

\begin{corollary}\label{C:ppexpl}
If  $\phi$ is a \zariski\ formula and $\theta$ a pp-morphism $f_\theta\colon \phi\to\psi$,   then $\theta$ is explicit.
\end{corollary}
\begin{proof}
By the same argument as in the proof of Proposition~\ref{P:defmap}, it suffices to show this modulo $\theory T_A$. 
Let us first show this result for $\psi$ a pp-formula. 
We can repeat the proof of Proposition~\ref{P:defmap}, up to the point that we introduced the polynomials $\tilde h_i$ and  $\tilde g_j$. Instead, $\theta(\var,\vary)$ is of the form $(\exists z)\zeta(\var,\vary,z)$, with $\zeta$ a \zariski\ formula in the variables $\var$, $\vary$, and $z$, and likewise, $\psi$ is of the form $\fim\gamma=(\exists z)\gamma(\var,z)$, with $\gamma$ a \zariski\ formula in the variables $\var$ and $z$. Let $I(\zeta)$ and $I(\gamma)$ be generated respectively by polynomials $h_i(\var,\vary,z)\in \pol A{\var,\vary,z}$ and $g_j(\var,z)\in\pol A{\var,z}$, for $i=\range 1s$ and $j=\range 1t$. Since $\theta(a,p(a)$ and $\psi(p(a))$ hold in $C$, we can find tuples of polynomials $g,q\in \pol A\var$ so that $\zeta(a,p(a),g(a))$ and $\gamma(p(a),q(a))$ hold  in $C$, implying that all $\tilde h_i(\var):=h_i(\var,p(\var),g(\var))$ and $\tilde g_j(\var):=g_j(p(\var),q(\var))$ lie in $I(\phi)$. Hence in an arbitrary $A$-algebra $B$, we have for every $ b\in \inter\phi B$, that $\tilde h_i( b)$ and $\tilde g_j( b)$ are all zero. This proves $p( b)\in \inter\psi B$ and $( b,p( b))\in\inter\theta B$, and we can now finish the proof as above for $\psi$ a pp-formula.

For the general case, let $\psi$ be arbitrary, and define the pp-formula
$$
\fim{\phi\en\theta}(\vary):=(\exists \var)\phi(\var)\en\theta(\var,\vary).
$$
In order to show that $\theta$ defines a morphism $\phi\to\fim{\phi\en\theta}$, we may verify this in an  arbitrary $A$-algebra $B$. Let $c\in\inter\phi B$ and put $ b:=f_B(c)$. Hence $\theta(c, b)$, whence $\fim{\phi\en\theta}( b)$, holds, showing that $f_B$ maps $\inter\phi B$ inside $\inter{\fim{\phi\en\theta}}B$. On the other hand, if $ b\in\inter{\fim{\phi\en\theta}}B$, then there exists $c\in\inter\phi B$ such that $\theta(c, b)$ holds, and hence $ b=f_B(c)\in\inter\psi B$. Hence, the implication $\fim{\phi\en\theta}\dan\psi$ is explicit. Moreover,  by our previous argument, the morphism $\phi\to\fim{\phi\en\theta}$ is explicit, and therefore, so is the composition $\phi\to\fim{\phi\en\theta}\to\psi$, as we wanted to show.
\end{proof}

\begin{corollary}
If $\phi$ is a pp-formula  and $f\colon\phi\to\psi$   a pp-morphism, then we can lift $f$ to an explicit morphism. More precisely, if  $\phi=\fim\gamma$ with $\gamma$ a \zariski\ formula, then there exist explicit morphisms $p\colon\gamma\to\phi$ and $\tilde f\colon\gamma\to \psi$ with $p$ surjective, yielding a commutative diagram
\commtrianglefront\gamma{p}\phi f{\psi.}{\tilde f}
\end{corollary}
\begin{proof}
Let $p\colon \gamma\to \phi$ be the projection map, that is to say, the explicit morphism given by $\vary =\var$. Hence the composition $\tilde f:=f\after p$ is a pp-morphism. By Corollary~\ref{C:ppexpl}, this is an explicit morphism, as we wanted to show.
\end{proof}

\subsection*{Zariski closure of a formula}
Given a formula $\psi$ in $\mathcal L_A$, we define its \emph{Zariski closure} $\bar\psi$ as follows. Suppose $\psi$ has arity $n$. Let $I$ be the sum of all $I(\phi)$, with $\phi\in\zar_n$ such that $\psi\dan \phi$ holds in $\theory T_A$, and let $\bar\psi$ be a \zariski\ formula corresponding to the ideal $I$. Hence, by Theorem~\ref{T:fosch}, the Zariski closure is defined up to $\arttheory A$-equivalence. Moreover, it satisfies the following universal property:   $\psi\dan \bar\psi$, and if $\phi$ is a \zariski\ formula such that $\psi\dan\phi$, then $\bar\psi\dan\phi$. Form the definition, it follows that $\bar\psi$ is equivalent to the infinite conjunction of all $\phi\in\zar_n$ such that $\psi\dan\phi$. 

\begin{lemma}\label{L:zarclmor}
Every explicit morphism $f\colon\psi\to\psi'$ of $\mathcal L_A$-formulae extends to an explicit morphism $\bar f\colon\bar\psi\to\bar\psi'$.  
\end{lemma}
\begin{proof}
Let $\theta$ be an explicit formula defining the morphism $f_\theta\colon \psi\to\psi'$, and let $p\colon \affine An\to\affine Am$ be the total map defined by $\theta$. Again, by Theorem~\ref{T:fosch}, we may work modulo $\theory T_A$. Hence for any $A$-algebra $B$,  the induced map $f_B$ is simply the restriction of the base change $p_B$ of $p$. So remains to show that $p_B$ maps $\inter{\bar\psi} B$ inside $\inter{\bar\psi'} B$. Let $\phi(\var)$ be the \zariski\ formula  $\bar\psi'(p(\var))$ obtained by substituting $p$ for $\vary$. It follows that $\inter\psi B\sub \inter\phi B$. Since this holds for all $B$, we get $\psi\dan\phi$ and hence by the universal property of Zariski closure, $\bar\psi\dan\phi$. It is now easy to see that this means that $p_B$ maps $\inter{\bar\psi}B$ inside $\inter{\bar\psi'}B$.  
\end{proof}

 The non-explicit, \zariski\ map from Example~\ref{E:nonexpl} does not extend to the Zariski closure of $\psi$, as $\bar\psi(\var_1)=(\var_1=\var_1)$.

 \begin{corollary}\label{C:zarclmor}
If two $\mathcal L_A$-formulae $\psi$ and $\psi'$ are $\expl$-isomorphic, then so are their Zariski closures $\bar\psi$ and $\bar\psi'$.
\end{corollary}
\begin{proof}
Let $\theta$ and $\op{inv}(\theta)$ be   explicit formulae defining respectively $f\colon\psi\to\psi'$ and its inverse $g\colon \psi'\to\psi$. Applying Lemma~\ref{L:zarclmor} to both morphisms   yields explicit morphisms $\bar f\colon\bar\psi\to\bar\psi'$ and $\bar g\colon\bar\psi'\to\bar\psi$.  Moreover, since the compositions $g\after f\colon \psi\to \psi$ and $f\after g\colon \psi'\to\psi'$  are both   identity morphisms, so must their Zariski closures be, and it is not hard to see that these  are $\bar g\after \bar f$ and $\bar f\after\bar g$ respectively. Hence, we showed that $\bar f$ and $\bar g$ are each others inverse.
\end{proof}

The Zariski closure of a formula can in general be hard to calculate. Here is a simple example: if $\fld$ is an \acf\ and   $f,g\in\pol \fld\var$ are relatively prime polynomials in a single variable, then the Zariski closure of $\psi:=(f=0)\of(g=0)$ is the formula $fg=0$. It is clear that $\psi\dan(fg=0)$. To show that it satisfies the universal property for Zariski closures, let $\phi$ be any \zariski\ formula implied by $\psi$. We need to show that $I(\phi)\sub fg\pol \fld\var$. We may reduce therefore to the case that $\phi$ is the formula $h=0$, and hence we have to show that any root of $f$ or $g$ in $\fld$ is also a root of $h$ of at least the same multiplicity. By the Nullstellensatz, and after a translation, it suffices to prove that if $0$ is a root of $f$ of multiplicity $e$, then it is also a root of $h$ of multiplicity $e$. Write $f(\var)=\var^e\tilde f(\var)$, for some $\tilde f\in\pol \fld\var$. Let $B:=\pol \fld\var/fg\pol \fld\var$, and put $b:=\var\tilde fg\in\pol \fld\var$. Since $f(b)=\var^e\tilde f^eg^e\tilde f(b)$, it is zero in $B$, that is to say, $b\in\psi(B)$. Hence, by assumption, $b\in\phi(B)$, that is to say, $h(b)=0$ in $B$. In particular, $h(b)$ is divisible by $\var^e$ in $\pol \fld\var$, and writing out $h$ as polynomial in $\var$ then easily implies that $h$ itself must be divisible by $\var^e$. Hence $\var=0$ is a root of $h$ of multiplicity $e$, as we wanted to show. We expect that this result holds true in far greater generality: is it the case that the Zariski closure of a disjunction $\phi_1\of\phi_2\of\dots\of\phi_s$ of \zariski\ formulae $\phi_i$ is the \zariski\ formula corresponding to the ideal $I(\phi_1)\cap I(\phi_2)\cap\dots\cap I(\phi_s)$?

\section{The \zariski\ \gr}
Let $X:=\op{Spec}A$ be  an affine, Noetherian scheme. 

\subsection*{The classical \gr}
Before we discuss our generalization  to schemes, let us first study the classical case. To this end, we must work in the theory $\ACF X$, the theory of \acf{s} having the structure of an $A$-algebra. We have the following analogue of Theorem~\ref{T:fosch}.

\begin{theorem}\label{T:fovar}
Let $A$ be a Noetherian ring, $\var$ an $n$-tuple of indeterminates, and
$\affine An$ the affine scheme $\op{Spec}(\pol Ax)$. 
There is a one-one correspondence between the following three sets:
\begin{enumerate}
\item\label{i:vargf} 
  the  set   of $\ACF A$-equivalence classes of 
\zariski\  formulae of arity  $n$;
\item\label{i:varid}   the set of  radical ideals in
$\pol A\var$;
\item\label{i:varcs} the set of reduced subschemes  of $\affine An$.
\end{enumerate}
\end{theorem}
\begin{proof}
The one-one correspondence between the last two sets is again classical. Let $\phi$ and $\psi$ be   \zariski\ formulae  in the free
variables $\var$. Assume first that $\phi$ and $\psi$ are
$\ACF A$-equivalent.
We need to show that $I(\phi)$ and $I(\psi)$ have the same radical. Suppose not, so that there exists a prime ideal $\pr\sub\pol A\var$ containing exactly one of these ideals, say, $I(\phi)$, but not the other. Let $K$ be the algebraic closure of $\pol A\var/\pr$, so that $K$ is a model of $\ACF A$, and let $a$ denote the image of $\var$ in $K$. By assumption, we can find $f\in I(\psi)$ such that $f\notin\pr$.   In particular,  since $f(a)\neq 0$ in $K$, the tuple  $a$ does not belong to $\inter\psi K=\inter\phi K$. However, for any $g\in I(\phi)$, we have $g\in\pr$ whence $g(a)=0$ in $K$, contradiction.

  Conversely, if both ideals have the  same radical, then  each $f\in I(\phi)$ has some power $f^N$  belonging to $I(\psi)$. In particular, if $K$ is a model of $\ACF A$ and $c\in \inter\psi K$, then $f^N(c)$ whence also $f(c)$ vanishes in $K$, for all $f\in I(\phi)$, showing that $c\in\inter\phi K$. This shows that $\inter\psi K\sub\inter\phi K$, and the reverse inclusion follows by the same argument, proving that $\phi$ and $\psi$ are $\ACF A$-equivalent.
\end{proof}

Let $X$ be a Noetherian affine scheme. We define its \emph{classical \gr}   to be the \gr\
$$
\grotclass X:=\grotmod{\ACF X}{\qf}{\expl},
$$
that is to say, the ring obtained by killing the ideal of all scissor relations and all $\expl$-isomorphism relations in the free Abelian group on all   classes of quantifier free formulae modulo $\ACF X$. The analogue of Corollary~\ref{C:isosch} holds, showing that the set of $\expl$-isomorphism classes of \zariski\ formulae modulo $ \ACF X$ is in one-one correspondence with the set of  isomorphism classes of reduced affine $ X$-schemes. Moreover, Corollary~\ref{C:grgen} applies with $\mathcal G=\zar$, so that the classes of \zariski\ formulae generate this \gr. Hence, we showed  the first assertion of:

\begin{corollary}\label{C:classgr}
If $ \fld$ is an \acf, then $\grotclass \fld$ is the \gr\ $K_0( \fld)$ obtained by taking the free Abelian group on isomorphism classes of varieties and killing all scissor relations. Moreover, if $\fld$ has \ch\ zero, then it is also equal to the full \gr\ $\grot{\ACF \fld}$ of the theory $\ACF \fld$.
\end{corollary} 
\begin{proof}
To prove the last assertion, observe that $\ACF \fld$ has quantifier elimination, and therefore the full \gr\ $\grot{\ACF \fld}$ is generated by the classes of quantifier free formulae. The only issue is the nature of isomorphism. 
In view of Corollary~\ref{C:grgen}, it suffices to show that if $f\colon Y\to X$ is an $\ACF \fld$-isomorphism of affine $\fld$-schemes, then $\class X=\class Y$ in $\grotclass \fld$. By \cite[?]{Mar}, we can find a constructible partition $Y=Y_1\cup\dots\cup Y_s$ of $Y$, such that each restriction $\restrict f{Y_i}$ is an explicit isomorphism (note that we need \ch\ zero to avoid having to take $p$-th roots). Hence $\class{Y_i}=\class{f(Y_i)}$ in $\grotclass \fld$, and the result now follows since $\class Y$ and $\class X$ are the respective sums of all $\class{Y_i}$ and all $\class{f(Y_i)}$.
\end{proof}

\subsection*{The \zariski\ \gr}

Let  $\qf$  be the Boolean closure of $\zar$, that is to say, the lattice of \emph{quantifier free formulae}. We define the
\emph{\zariski\  \gr} of $X$ as 
$$
\grotsch X:=\grotmod{\arttheory X}{\qf}{\zar}.
$$
As before, we denote the class of a formula by $\class \phi$, and in case $Y$ is an affine $X$-scheme, we also write $\class Y$  for the class of its defining \zariski\ formula given by Corollary~\ref{C:isosch}, and henceforth, identify both. In particular, the base scheme $X$ corresponds to the class of the sentence $\one$, which we will denote simply by $1$.  

\begin{lemma}\label{L:cycle}
Let $X$ be an affine, Noetherian scheme. Any element in $\grotsch X$ is of the
form $\sum_{i=1}^s n_i\class{Y_i}$, for some integers $n_i$, and some affine $X$-schemes $Y_i$. Alternatively, we may  write any element as $\class Z-\class
Z'$, for $Z$ and $Z'$ affine $X$-schemes.
\end{lemma}
\begin{proof}
Both statements follow  immediately from Theorem~\ref{T:fosch} and Corollary~\ref{C:grgen}, since $\zar$ is closed, modulo $\arttheory X$, under conjunctions, and, by Lemma~\ref{L:schdisjun},    disjoint sums.  
\end{proof}

 In particular, the natural morphism $\grotmod{\arttheory X}{\zar}\zar\to \grotsch X$   is surjective.    
Let us call two $X$-schemes $Y$ and $Y'$ \emph{stably isomorphic}, if there exists an affine $X$-scheme $Z$ such that $Y\sqcup Z$ and $ Y'\sqcup Z$ are isomorphic over $X$. A priori this is a weaker equivalence relation   than the isomorphism relation, but for affine Noetherian schemes, it is the same, as we will discuss in Appendix~\ref{s:app}.

\begin{theorem}\label{T:classinv}
Two affine $X$-schemes  $Y$ and $Y'$ are   isomorphic \iff\  their classes $\class Y$ and $\class {Y'}$ in $\grotsch X$ are the same. 
\end{theorem}
\begin{proof}
By Theorem~\ref{T:stabiso}, the defining \zariski\ formulae $\phi$ and $\phi'$ of respectively $Y$ and $Y'$ are   stably $\zar$-isomorphic,   that is to say, 
$$
\phi\oplus\psi\iso_{\zar}\phi'\oplus\psi,
$$
 for some quantifier free formula $\psi$. By Lemma~\ref{L:stabiso}, we may replace $\psi$ by any formula implied by it, whence, in particular, by its Zariski closure. In conclusion, we may assume $\psi$ is \zariski.  By Lemma~\ref{L:schdisjun},   disjoint sum  and union  are $\arttheory X$-equivalent. Hence, if $Z$ denotes the affine scheme defined by $ \psi$, then $Y\sqcup Z\iso Y'\sqcup Z$ by Corollary~\ref{C:zarmor}, that is to say,  $Y$ and $Y'$ are stably isomorphic, whence  isomorphic, by Theorem~\ref{T:stabisosch}.
\end{proof}

By definition of the Lefschetz class, the class of $\affine X1:=\op{Spec}(\pol A\var)$, with $\var$ a single indeterminate, is $\lef$. In particular, we get the following generalization of Lemma~\ref{L:Lef}:

\begin{lemma}\label{L:Lefsch}
If $Y$ is an affine $X$-scheme, then $\class{\affine Yn}=\lef^n\cdot\class Y$ in $\grotsch X$.\qed
\end{lemma}

\section{The   pp-\gr{} of $X$}\label{s:ppgr}
Theorem~\ref{T:classinv} essentially says that $\grotsch X$ creates no further relations among   $X$-schemes. To obtain non-trivial relations among   classes of $X$-schemes, we will now work in  a larger class of formulae. It turns out that   pp-formulae are sufficiently general to accomplish this. To obtain the greatest amount of versatility, rather than working within the semi-lattice of pp-formulae, we will  work in the Boolean  lattice  $\pp$, given as the Boolean closure of all pp-formulae, that is to say, all finite disjunctions of pp-formulae and their negations.  Recall that we write $\fim\phi(\var)$ to denote the pp-formula $(\exists\vary)\phi(\var,\vary)$, for $\phi$ a \zariski\ formula. We define the \emph{pp-\gr} of $X=\op{Spec}A$ to be the \gr\
$$
\grotart X{}:=\grotmod{\arttheory X}\pp\zar.
$$
By Corollary~\ref{C:addmap}, we have a natural ring \homo\ $\grotsch X\to \grotart X{}$.
The analogue of Lemma~\ref{L:schdisjun} holds, and in particular, we can describe the elements of $\grotart X{}$ as in Lemma~\ref{L:cycle}:

\begin{proposition}\label{P:ppcycle}
The disjoint sum of two pp-formulae is $\arttheory X$-equivalent to a pp-formulae. In particular, every element of $\grotart X{}$ can be written as a difference $\class\psi-\class{\psi'}$ with $\psi$ and $\psi'$ pp-formulae.
\end{proposition}
\begin{proof}
Let   $\phi(\var):=\fim{\phi_0}$ and $\psi(\var):=\fim{\psi_0}$ be two pp-formulae, with $\phi_0(\var,\vary)$ and $\psi_0(\var,\vary)$ \zariski\ formulae. As in the proof of Lemma~\ref{L:cycle}, one easily shows that the disjoint union   $\phi\sqcup\psi$ is $\arttheory X$-equivalent with $(\exists y)(\phi_0\sqcup\psi_0)$. The second assertion now follows from this and Corollary~\ref{C:grgen}.
\end{proof}

Assume $Y=\op{Spec}B$ is an affine, Noetherian scheme. Let  $\Theta^{\text{aff}}_Y$  be the collection of all  affine opens of $Y$. We view $\Theta^{\text{aff}}_Y$ as a semi-lattice with $\en$ given by intersection. In general, the union of affine opens need not be affine, so that we cannot define $\of$ on $\Theta^{\text{aff}}_Y$, and therefore, it is only a sub-semi-lattice of the lattice $\Theta_Y$ of all opens of $Y$. Recall that for a finite, open affine covering  $\mathcal U=\{U_1,\dots,U_n\}$ of $Y$, we have a  scissor relation  
$$
Y=S_n\rij Un 
$$
in $\grotlat{\Theta_Y}$. 
The map $\Theta^{\text{aff}}_Y\to \grotart X{}$, sending an affine open $U\sub Y$ to its class $\class U$ in $\grotart X{}$, extends  to an additive map  $\pol\zet{\Theta^{\text{aff}}_Y}\to \grotart X{}$ (note that this map is not multiplicative since multiplication on the former is different from that on the latter). We will    show in Corollary~\ref{C:classsch} below that it in fact induces a ring \homo\ $\grotmon{\Theta^{\text{aff}}_Y}{\iso_X}\to \grotart X{}$. Among the members of $\Theta^{\text{aff}}_Y$ are the \emph{basic} open subsets   $\op D(f)=\op{Spec}(B_f)$, with $f$ a non-nilpotent element of $B$ (so that in particular, a basic open subset is never empty). Note that   if $\op D(f)$ is a basic open, and $U\sub Y$ and affine open, then $\op D(f)\cap U$ is the basic open $\op D(\restrict fU)$ in $U$.

\begin{proposition}\label{P:basop}
Let $X$ be a Noetherian scheme,  and $Y$ an affine $X$-scheme. For every finite covering $\{D_1,\dots, D_n\}$ of $Y$ by basic open subsets, we have an identity $\class Y=\class{S_n\rij Dn}$ in $\grotart X{}$.
\end{proposition}
\begin{proof}
Let $A:=\loc_X$, and     let $\phi(\var)$   be a \zariski\ formula defining $Y$, that is to say,  $Y=\op{Spec}B$ with  $B:=\pol A\var/I(\phi)$. By definition of basic subset, there exist $f_i\in\pol A\var$ so that  $D_i=\op D(f_i)=\op{Spec}(B_{f_i})$. The fact that the $D_i$ cover $Y$ is equivalent with $\rij fnB$ being the unit ideal. Hence, we can find  $u_i\in \pol A\var$ and $m\in I(\phi)$, such that
\begin{equation}\label{eq:bascov}
\sum_{i=1}^nu_if_i+m=1.
\end{equation}

Consider the following formulae in the variables $\var$ and $z$:
$$
\begin{aligned}
\psi_i(\var,z)&:=\phi(\var)\en(f_i(\var)z=1)\\
\tilde\psi_i(\var)&:=\fim{\psi_i}(\var)=(\exists z)\psi_i(\var,z).
\end{aligned}
$$
The \zariski\ formula $\var=\vary$ defines a   \zariski\ isomorphism $\tilde\psi_i(\var)\to\psi_i(\vary,z)$, for each $i$. 
Let $\tilde\psi$ be the disjunction of all $\tilde\psi_i$.  By Proposition~\ref{P:scissor}, we have a scissor identity
$\class {\tilde\psi}=\class{S_n\rij{\tilde\psi}n}$ in $\grotart X{}$. Since $\psi_i$ and $\tilde \psi_i$ are $\zar$-isomorphic, so are any of their conjunctions, and hence $\class{S_n\rij{\tilde\psi}n}=\class{S_n\rij\psi n}$ in $\grotart X{}$.  By definition of basic subset, $\psi_i$ is the defining \zariski\ formula of $D_i$, and hence $\class{S_n\rij\psi n}$ is equal to the class of $S_n\rij Dn$ in $\grotart X{}$. So remains to show that $Y=\class \phi=\class{\tilde\psi}$.

 To this end, we have to show that $\inter{\tilde\psi}R=\inter\phi R$ in any Artinian local $A$-algebra $R$. The direct  inclusion is immediate, so assume $ a\in\inter\phi R$. From \eqref{eq:bascov} and the fact that $m( a)=0$, we get 
 $$
 \sum_{i=1}^nu_i( a)f_i( a)=1.
 $$
  Since   $R$ is local, one of these terms must be a unit, say, the first one. Therefore, there exists $b\in R$ such that $bf_1( a)=1$,  and hence $ a\in\tilde\psi_1(R)\sub\inter{\tilde\psi}R$, as we needed to show.
\end{proof}

\begin{remark}\label{R:caut}
As already observed, this shows that we have an additive map $\grot{\Theta^{\text{aff}}_Y}\to \grotart X{}$.  
A cautionary note: although one informally states that a basic open subset $D:=\op D(f)$ is given by the equation $f\neq 0$ in $\op{Spec}(B)$ (as it is the complement of the closed subset given by the equation $f=0$), this is not correct from the point of view of formulae. As we saw in the above proof, $D$ is defined by the \zariski\ formula $\psi(\var,z):=\sym{f(\var)z=1}$, or alternatively, by its projection, the pp-formula $\fim\psi(\var):=\sym{(\exists z)f(\var)z=1}$. That this  is different from the formula $\sym{f\neq 0}$ is easily checked on an example: let $f(\var):=\var\in\pol \fld\var$, with $\var$  a single variable, and compare both formulae in the Artinian local ring $\pol \fld T/T^2\pol \fld T$ (we will see below that in the latter model, nonetheless, the basic open $\op D(f)$ is   given by  the quantifier free (non-\zariski) formula $f^2\neq0$). 

A second issue requiring some care is the difference between conjunctions and  intersections. Let $D':=\op D(f')$ be another basic open in $Y$, and let $\psi':=\sym{f'z=1}$ and $\fim{\psi'}:=\sym{(\exists z)f'z=1)}$ be its respective \zariski\ and pp defining formula. The intersection $D\cap D'$  is again a basic open subset, whence an affine scheme. However, $D\cap D'$ is not defined by the \zariski\ formula $\psi\en\psi'$, but by the \zariski\ formula $\sym{ff'z=1}$. Nonetheless,  $D\cap D'$ is defined by the pp-formula $\fim\psi\en\fim{\psi'}$. 
\end{remark}

\begin{theorem}\label{T:classsch}
There exists a well-defined map which assigns to any (isomorphism class of an) $X$-scheme $Y$ an element $\class Y$ in $\grotart X{}$ which agrees on affine schemes with the class map. Moreover, if $\{U_1,\dots,U_n\}$ is any open covering of an $X$-scheme $Y$, then $\class Y=\class{S_n\rij Un}$ in $\grotart X{}$.
\end{theorem}
\begin{proof}
We start with proving that the second assertion holds in case $Y$ is affine. So let $\mathcal U:=\{U_1,\dots,U_s\}$ be an open affine covering of $Y$, and we need to show that $\class Y=\class{S\rij Un}$ in $\grotart X{}$. We induct on the number $e$ of non-basic opens among the $U_i$. If $e=0$, then the result holds by Proposition~\ref{P:basop}. So assume $e>1$. Let $U_n$ be a non-basic open subset, and let $\{D_1,\dots,D_m\}$ be an open covering of $U_n$ by basic opens. 
Using Lemma~\ref{L:scisind}\eqref{eq:scisind}, as in the proof of  Proposition~\ref{P:scissor},    we have an identity
\begin{multline}\label{eq:UVn}
S(U_1,\dots,U_{n-1},D_1,\dots,D_m,U_n) =\\
S(U_1,\dots,U_{n-1},D_1,\dots,D_m)-S(U_1\cap U_n,\dots,U_{n-1}\cap U_n,D_1,\dots,D_m)+U_n
\end{multline}
Since $\{U_1,\dots,U_{n-1},D_1,\dots,D_m\}$ and $\{U_1\cap U_n,\dots,U_{n-1}\cap U_n,D_1,\dots,D_m\}$ are open affine coverings of $Y$ and $U_n$ respectively, both containing less than $e$ non-basic open subsets, our induction hypothesis yields
$$
\begin{aligned}
\class Y&=\class{S(U_1,\dots,U_{n-1},D_1,\dots,D_m)} \quad\text{and}\\\class{U_n}&=\class{S(U_1\cap U_n,\dots,U_{n-1}\cap U_n,D_1,\dots,D_m)},
\end{aligned}
$$
in $\grotart X{}$. Taking classes of both sides of \eqref{eq:UVn} together with the latter identities, shows that $\class Y=\class{S(U_1,\dots,U_n,D_1,\dots,D_m)}$ (note that the order in scissor relations is irrelevant). We will now prove by induction on $m$, that $S(U_1,\dots,U_n,D_1,\dots,D_m)$ and $S\rij Un$ have the same class in $\grotart X{}$. 
 By Lemma~\ref{L:scisind}\eqref{eq:scisind},
we have an identity
\begin{multline}\label{eq:UVmn}
S(U_1,\dots,U_n,D_1,\dots,D_m) =S(U_1,\dots,U_n,D_1,\dots,D_{m-1})\\
-S(U_1\cap D_m,\dots,U_n\cap D_m,  D_1\cap D_m,\dots,D_{m-1}\cap D_m)+D_m
\end{multline}
in $\pol\zet{\Theta^{\text{aff}}_Y}$. By induction on $m$, we have 
$$
\class{S\rij Un}=\class{S(U_1,\dots,U_n,D_1,\dots,D_{m-1})}
$$
 in $\grotart X{}$. On the other hand, since $\{U_i\cap D_m,   D_j\cap D_m\}$, for $i=\range 1n$ and $j=\range 1{m-1}$, is an open affine covering of $D_m$ containing less than $e$ non-basic open subsets (note that $U_n\cap D_m=D_m$), our first induction hypothesis (on $e$) yields 
$$
\class{D_m}=\class{S(U_1\cap D_m,\dots,U_n\cap D_m,D_1\cap D_m,\dots,D_{m-1}\cap D_m)}
$$
in $\grotart X{}$. Hence,   taking classes of both sides of \eqref{eq:UVmn} together with the previous two identities, yields the desired conclusion
$$
\class Y=\class{S(U_1,\dots,U_n,D_1,\dots,D_m)}=\class{S\rij Un}.
$$

We now prove the first assertion for $Y$   an arbitrary $X$-scheme, that is to say, a (not necessarily affine) separated scheme $Y$ of finite type over $X$. Let 
$\mathcal U:=\{U_1,\dots,U_s\}$ be an open affine covering of $Y$, and define 
\begin{equation}\label{eq:classsch}
\class Y_{\mathcal U}:=\class{S\rij Un}.
\end{equation}
Since $Y$  is separated, each intersection $U:=U_{i_1}\cap\dots\cap U_{i_k}$ is again affine, so that $\class{S\rij Un}$ is indeed an element of $\grotart X{}$. We want to show that $\class Y_{\mathcal U}$, as an element of $\grotart X{}$, does not depend on the open affine covering $\mathcal U$. To this end, let $\mathcal V$ be a second open affine covering of $Y$, and we seek to show that $\class Y_{\mathcal U}=\class Y_{\mathcal V}$ in $\grotart X{}$.  Replacing $\mathcal V$ by $\mathcal U\cup\mathcal V$  if necessary, we may assume that $\mathcal U\sub\mathcal V$. By induction on the number of members of $\mathcal V$, we may then reduce to the case that $\mathcal V=\mathcal U\cup\{V\}$, for some affine open $V\sub Y$.  By Lemma~\ref{L:scisind}\eqref{eq:scisind},   we have, in $\pol\zet{\Theta^{\text{aff}}_Y}$, an identity
\begin{equation}\label{eq:Un}
S(U_1,\dots,U_n,V) =S\rij Un-S(U_1\cap V,\dots,U_n\cap V)+V.
\end{equation}
Since $\{U_1\cap V,\dots,U_n\cap V\}$ is an   open affine covering of the affine open $V$, we have
$$
\class{V}=\class{S(U_1\cap V,\dots,U_n\cap V)}
$$
in $\grotart X{}$ by the first part of the proof. Hence, taking classes of both sides of \eqref{eq:Un} shows that 
$$
\class Y_{\mathcal V}=\class{S(U_1,\dots,U_n,V)}= \class{S_n\rij Un}=\class Y_{\mathcal U}
$$
 in $\grotart X{}$.  

So, for $Y$ an arbitrary $X$-scheme, we define $\class Y:=\class Y_{\mathcal U}$, where $\mathcal U$ is any finite open affine covering of $Y$. In particular, if $Y$ is affine, we can take for open cover the singleton $\{Y\}$, showing that this new notation coincides with our former. To show that this assignment only depends on the isomorphism class of $Y$, let $\sigma\colon Y\to Y'$ be an isomorphism of $X$-schemes. Let $\mathcal U'$ consist of all $\sigma(U)$ with $U\in\mathcal U$. Hence $\mathcal U'$ is an open covering of $Y'$. Moreover,   any intersection of members of $\mathcal U$ is  
isomorphic to the intersection of the corresponding images under $\sigma$, and hence both have the same class in $\grotart X{}$. Therefore, 
 $$
\class Y=\class Y_{\mathcal U}=\class{Y'}_{\mathcal U'}=\class {Y'},
$$
showing that the class of $Y$ only depends on its isomorphism type.

Finally, to prove the last assertion, we first show that the additive map $\class\cdot\colon\pol\zet{\Theta_Y}\to \grotart X{}$ factors through an additive map $\grotlat{\Theta_Y}\to \grotart X{}\colon U\mapsto \class U$, for every $X$-scheme $Y$ (recall that $\Theta_Y$ is the lattice of all opens of $Y$). It suffices to show that any second scissor relation $ U\cup U'-U-U'+U\cap U'$, with $U,U'\in\Theta_Y$, lies in the kernel of $\pol\zet{\Theta_Y}\to \grotart X{}$. Let $\tuple U:=\rij Un$ and $\tuple U':=\rij{U'}{n'}$ be affine open coverings of $U$ and $U'$ respectively. Hence the union of these two coverings is a covering of $U\cup U'$, whereas the collection $\tuple V$ of all $U_i\cap U_j'$, for $i=\range 1n$, and $j=\range 1{n'}$,  is an affine covering of $U\cap U'$. Therefore, by \eqref{eq:classsch}, we have 
\begin{equation}\label{eq:2ndscis}
 \class{U\cup U'}-\class U-\class {U'}+ \class {U\cap U'}=\class{S(\tuple U,\tuple U')}-\class{S(\tuple U)}-\class{S(\tuple U')}+\class{S(\tuple V)}
\end{equation}
However, by Lemma~\ref{L:scisind}\eqref{eq:scisprod}, we have an identity
$$
 {S(\tuple U,\tuple U')}- {S(\tuple U)}- {S(\tuple U')}+ {S(\tuple V)}=0
 $$
  in $\pol\zet{\Theta_Y}$,  showing that the right hand side of \eqref{eq:2ndscis}, whence also the left hand side, is zero. We can now prove the last assertion:  let $\{U_1,\dots,U_n\}$ be an arbitrary finite open covering of $Y$. By Proposition~\ref{P:scissor}, we have an identity
 $$
 Y=\bigcup_{i=1}^n U_i=S\rij Un
 $$
in $\grotlat{\Theta_Y}$. Applying the additive map $\grotlat{\Theta_Y}\to \grotart X{}$ then yields the desired identity.
\end{proof}

In the course of the proof, we obtained:

\begin{corollary}\label{C:classsch}
For each $X$-scheme $Y$, we have   a   \homo\ $\grotmon{\Theta_Y}{\iso_X}\to \grotart X{}$ of \gr{s},  where $\Theta_Y$ is the lattice of opens of $Y$, and $\iso_X$ denotes isomorphism as $X$-schemes.\qed
\end{corollary}

\begin{corollary}\label{C:openform}
If $U$ is an open in an affine $X$-scheme $Y$, then there exists a disjunction $\psi$ of \zariski\ formulae such that $\class U=\class\psi$ in $\grotart X{}$.
\end{corollary}
\begin{proof}
Let $\{D_1,\dots,D_n\}$ be a covering of $U$ by basic open subsets, and, as in the proof of Proposition~\ref{P:basop}, let $\psi_i$ be the \zariski\ formula defining the basic open $D_i$, that is to say, $\psi_i(\var,z):=\sym{\phi(\var)\en(f_i(\var)z=1)}$, where $\phi$ is the defining formula of $Y$, and $D_i$ is the basic open $\op{Spec}(\loc_{Y,f_i})$. By  Proposition~\ref{P:basop}, the basic open $D_i$ is   also defined by the pp-formula $\fim{\psi_i}$.
Let $\tilde\psi:=\fim{\psi_1}\of\dots\of\fim{\psi_n}$. By Proposition~\ref{P:scissor}, we have identities $U=S\rij Dn$
 in $\grotlat{\Theta_Y}$, and $\tilde\psi=S(\fim{\psi_1},\dots,\fim{\psi_n})$ in $\grotlat{\pp}$. By   Corollaries~\ref{C:classsch} and \ref{C:addmap} respectively, we get $\class U=\class{S\rij Dn}$ and $\class {\tilde\psi}=\class{S(\fim{\psi_1},\dots,\fim{\psi_n})}$ in $\grotart X{}$. One easily verifies that any conjunction $\En_{i\in I}\fim{\psi_i}$ is the defining pp-formula for the corresponding intersection of the $D_i$, for $I\sub\{1,\dots,n\}$   (but see Remark~\ref{R:caut} for why we cannot work with the \zariski\ formulae $\psi_i$ instead).
 Hence 
 $$
\class U= \class {S\rij Dn}=\class{S(\fim{\psi_1},\dots,\fim{\psi_n})}=\class{\tilde\psi}
 $$
  in $\grotart X{}$.
  
To obtain a disjunction of \zariski\ formulae, let 
$$
\psi(\var,z_1,\dots,z_n):=\Of_{i=1}^n\psi_i(\var,z_i).
$$
As in the proof of Proposition~\ref{P:basop}, one can show that the projection onto the $\var$-coordinates yields a (\zariski) isomorphism between $\psi$ and $\tilde \psi$ modulo $\arttheory X$, and hence $\class\psi=\class{\tilde\psi}$  in $ \grotart X{}$, completing the proof of the assertion. 
 \end{proof}

\begin{remark}\label{R:classproj}
Theorem~\ref{T:classsch} allows us to calculate the class of a non-affine scheme in terms of classes of affine schemes. For instance, the class of the projective line $\mathbb P_X^1$ is equal to $2\lef-\lef^*$, where, as before, $\lef$ is the Lefschetz class, that is to say, the class of the affine line $\affine X1$, and where $\lef^*$ denotes the class of the affine line without the origin. One would be tempted to think that $\lef^*=\lef-1$, but this is false, for the reason given  at end of Remark~\ref{R:caut}. We will give a correct version of this formula in \eqref{eq:lefpt} below.
\end{remark}

\section{Arc integrals}\label{s:motint}
Let $X=\op{Spec}A$ be an affine Noetherian scheme, and let $\phi$ be a \zariski\ formula in $\mathcal L_A$ with corresponding ideal
$I(\phi)$, and associated affine scheme $Y_\phi:=\op{Spec}(\pol A\var/I(\phi))$.
We call  $\phi$
\emph{Artinian} if the corresponding affine scheme $Y_\phi$ is Artinian, that is to say, has (Krull) dimension zero. We denote the Boolean closure of the collection of Artinian formulae by $\mathcal Art$ (not to be confused with the theory $\arttheory A$). 
We say that $\phi$ is a \emph{closed point formula}, if $I(\phi)$ is a maximal
ideal; we say that $\phi$ is a \emph{point formula}, if the radical of
$I(\phi)$ is a maximal ideal.   Closed point formulae and point formulae are
Artinian. Let $\phi$ be a point formula. There is a unique closed point formula $\bar\phi$ implying $\phi$, namely, the one corresponding to the radical of $I(\phi)$. To a point formula corresponds an Artinian local $A$-scheme $Y_\phi$, and the closed point of $Y_\phi$ then corresponds to $\bar\phi$. If
$A=\fld$ is an
\acf, then any two point formulae
are $\theory T_\fld$-equivalent by the Nullstellensatz, but this might fail in
general. For instance, over $A=\mathbb Q$, the formulae $\var^2+1=0$ and
$\var=0$ are not isomorphic. Another example, with
$A=\pow \fld t$, are the point formulae formulae $t\var-1=0$
and $t=\var=0$, which cannot be isomorphic, not even after a base change. Let
us denote the \gr\   of Artinian
formulae   modulo $\arttheory X$ by 
$$
\grotschzero X:=\grotmod{\arttheory X}{\mathcal Art}\zar.
$$
 Given an Artinian $X$-scheme $Y$, we will call its \emph{length} $\ell(Y)$ the length of the   coordinate ring $\loc_Y$ viewed as an $A$-module.

\begin{corollary}\label{C:0cycle}
Any element of $\grotschzero  X $ can be written as a difference $\class Y-\class{Y'}$, with $Y$  and $Y'$ Artinian $X$-schemes.
\end{corollary}
\begin{proof}
By Lemma~\ref{L:schdisjun},  Artinian formulae are closed under  disjoint sums, and the claim follows from Corollary~\ref{C:grgen}.
\end{proof}

\begin{lemma}\label{L:Artdu}
Any Artinian formula is a disjoint union of finitely many point formulae.
\end{lemma}
\begin{proof}
Immediate from the fact that an Artinian ring is a direct sum of Artinian
local rings.
\end{proof}

For simplicity, we will for the remainder of this section work over an \acf\
$\fld$. 

\begin{lemma}\label{L:length}
Let $\fld$ be an \acf. There exists a ring \homo\ $\ell\colon \grotschzero \fld\to
\zet$ such that $\ell(\class Y)=\ell(Y)$ for any Artinian $\fld$-scheme $Y$.
\end{lemma}
\begin{proof}
 Since $\fld$ is algebraically closed, $\ell(Y)$ is equal to the $\fld$-vector
space dimension of $\loc_Y$. 
Let $\phi$ be a formula  in $ {\mathcal
Art}$.   By Corollary~\ref{C:0cycle},  
we can find Artinian   
$\fld$-schemes $Y$ and $Y'$ 
such that  $\class\phi=\class Y-\class{Y'}$. Put $\ell(\class\phi):=\ell(Y)-\ell({Y'})$. To prove that this is independent from the choice of
affine schemes, suppose that also $\class\phi=\class Z-\class{Z'}$ for some
Artinian
$\fld$-schemes $Z$ and $Z'$. It follows that $\class{Y\sqcup
Z'}=\class{Y'\sqcup Z}$, and hence by Theorem~\ref{T:stabiso}, that the
schemes $Y\sqcup Z'$ and $Y'\sqcup Z$ are
stably $\zar$-isomorphic in $\mathcal Art$. This means that $Y\sqcup Z'\sqcup T$
and $Y'\sqcup Z\sqcup T$ are isomorphic, for some Artinian $\fld$-scheme $T$.
Since
length is additive on disjoint unions,   $\ell( Y)+\ell( Z')+\ell( T)= \ell(
Y')+\ell( Z)+\ell( T)$, showing that $\ell(\class\phi)$  is
well-defined. Linearity follows immediately from this. Furthermore, given two
Artinian $\fld$-schemes, we have
$\ell(\class Y\cdot \class Z)=\ell(\class{Y\times_\fld Z})$. Let $l$ and $m$ be
the length of $ Y$ and $ Z$ respectively, so that as $\fld$-vector spaces
$\loc_Y\iso \fld^l$ and $\loc_Z\iso \fld^m$. Since the coordinate ring of
$Y\times_\fld Z $ is equal to $\loc_Y\tensor_\fld\loc_Z$, its length is equal to $mn$,
showing that $\ell$ is also multiplicative. 
\end{proof}

\subsection*{Jets}
Given  a closed subscheme $Z$ of an affine scheme $Y:=\op{Spec}B$,   we define the
\emph{$n$-th jet of $Y$ along $Z$} to be the closed subscheme 
$$
\jet ZnY:=\op{Spec}(B/I^n),
$$
 where $I$ is the ideal of the closed subscheme $Z$. Note that $Z$ and any of the jets $\jet ZnY$ have the same underlying topological space, and we have an ascending chain of closed subschemes
\begin{equation}\label{eq:jetchain}
\emptyset =\jet Z0Y\sub Z=\jet Z1Y\sub\jet Z2Y\dots\sub\jet ZnY\sub\dots
\end{equation}
In most cases, this will be a proper chain by Nakayama's Lemma (for instance, if $Z$ is a proper closed subscheme of a variety $Y$, or more generally, if $I$ contains a non-zero divisor).
 
We can generalize the notion of a jet to formulae: let  $\phi$ be an arbitrary formula and $\zeta$ a \zariski\ formula. We define   the
\emph{$n$-th jet of $\phi$ along $\zeta$} to be the formula
$$
\jet\zeta n\phi:=\phi\en \En_{f_1,\dots,f_n\in I(\zeta)} (f_1\cdot f_2\cdots f_n=0).
$$
In other words, if $\zeta^{(n)}$ is the formula with defining ideal $I(\zeta)^n$, then 
$\jet\zeta n\phi=\phi\en\zeta^{(n)}$. In particular, if $\phi$ is also \zariski, defining an affine variety $Y:=Y_\phi$, and if $Z$ is the closed subscheme defined by $\phi\en\zeta$, then $\jet\zeta n\phi$ is the defining \zariski\ formula of $\jet Z nY$.

\subsection*{Formal Hilbert series}
Let $\phi$ be a \zariski\ formula, and $\tau$ a closed point
formula implying $\phi$. Each jet $\jet\tau i\phi$ is   an Artinian formula, and hence its class belongs to $\grotschzero  \fld$. In particular, if $X$ is the $\fld$-scheme defined by $\phi$, and $P$ the closed point defined by $\tau$, then $\jet\tau i\phi=\jet PiX$.
For $T$ a single variable we can therefore  define the
\emph{formal Hilbert series} of a $\fld$-scheme $X$ at a closed point $P$ as the series
$$
\op{Hilb}_P(X):=\sum_{i=0}^\infty \class{\jet PiX}T^i
$$
in $\pow{\grotschzero \fld}T$. If we
extend the \homo\ $\ell$ to $\pow{\grotschzero \fld}T$ by letting it act
on the coefficients of a power series,
then $\ell(\op{Hilb}_P(X))$ is a rational function in $\pow\zet T$
by the Hilbert-Samuel theory (it is the first difference of the classical Hilbert series
of $X$ at   $P$).

\subsection*{Arcs}
Let $R$ be  an Artinian algebra of dimension $l$ over $\fld$, and fix   some basis  $\Delta$ of $R$ over $\fld$. For each $\alpha\in\Delta$, we define the $\alpha$-th coordinate map $\pi_\alpha\colon R\to \fld$ by the rule
$$
r=\sum_{\alpha\in\Delta}\pi_\alpha (r)\cdot\alpha.
$$
We write $\pi_R(r)$, or just $\pi (r)$, for the tuple of all $\pi_\alpha (r)$, where we fix once and for all an order of $\Delta$. In particular, $\pi$ gives a ($\fld$-linear) bijection between $R$ and $\fld^l$. We also extend this notation to arbitrary tuples. More generally, if $A$ is a $\fld$-algebra, then $R\tensor_\fld A$ is a free $A$-module generated by $\Delta$ and  the base change of $\pi$ yields an $A$-linear isomorphism $R\tensor_\fld A\iso A^l$, which we continue to denote by $\pi$. In this section, we also will fix the following notation. Given an $n$-tuple of variables $\var$, we let $\tilde\var_\alpha$, for each  $\alpha\in \Delta$, be another $n$-tuple of variables, and we denote the $ln$-tuple consisting of all $\tilde\var_\alpha$ by $\tilde\var$,   referring to them as   \emph{arc variables}. We also associate to each     $n$-tuple of variables $\var$, an  $n$-tuple of \emph{generic arcs}
$$
\dot\var:=\sum_{\alpha\in\Delta}\tilde\var_\alpha\alpha
$$
viewed as a tuple in $\pol R{\tilde\var}$.  In particular, $\pi_\alpha(\dot x)=\tilde\var_\alpha$.

\begin{proposition}\label{P:arc}
For each $\mathcal L_\fld$-formula $\phi$ of arity $n$, and for each finite $\fld$-algebra $R$ of dimension $l$, there exists an $\mathcal L_\fld$-formula $\arc R\phi$ of arity $ln$ with the following property: if $A$ is an  $\fld$-algebra   and $  a$ an  $n$-tuple in $R\tensor_\fld A$, then $ a\in\inter\phi {R\tensor_\fld A}$ \iff\ $\pi( a)\in\inter{\arc R\phi}A$. Moreover, if $\phi$ is \zariski\ or pp, then so is $\arc R\phi$.
\end{proposition}
\begin{proof}
Let $\Delta$ be a basis of $R$ as a vector space over $\fld$.  For each polynomial $f\in\pol \fld\var$,   define polynomials $\coor\alpha f \in\pol \fld {\tilde\var}$ by the rule
\begin{equation}\label{eq:ethcomp}
f(\dot\var)=f(\sum_\alpha\tilde\var_\alpha\alpha)=\sum_{\alpha\in\Delta}\coor\alpha f (\tilde\var)\alpha.
\end{equation}
In particular, for $A$ an $\fld$-algebra, we have $\pi_\alpha(f(a))=(\coor \alpha f)(\pi(a))$, for all $a$ in $R\tensor_\fld A$ and all $\alpha\in\Delta$.
To define $\arc R\phi({\tilde\var})$, we induct on the complexity of the formula $\phi$. If $\phi$ is the   \zariski\ formula $f(\var)=0$, then $\arc R\phi$ is the \zariski\ formula 
$$
\arc R\phi({\tilde\var}):=\sym{\En_{\alpha\in\Delta}(\coor \alpha f ({\tilde\var})=0)}.
$$
If $\phi$ and $\psi$ are formulae for which we already defined $\arc R\phi$ and $\arc R\psi$, then $\arc R(\phi\of\psi):=\arc R\phi\of\arc R\psi$, and $\arc R(\niet\phi):=\niet\arc R\phi$. Finally, if $\phi(\var)$ is the formula $(\exists\vary)\psi(\var,\vary)$, then we define $\arc R\phi$ as the formula
$$
\arc R\phi(\var):=\sym{(\exists\tilde \vary)\arc R\psi(\tilde\var,\tilde\vary)}
$$
where, similarly,  $\tilde\vary$ is a tuple of $l$ copies of $\vary$. This concludes the proof of the existence of $\arc R\phi$. That it satisfies the desired one-one correspondence between definable sets is clear from \eqref{eq:ethcomp}, and the last assertion is   immediate   as well.
  \end{proof}
  
  We will refer to  $\arc R\phi$ as the \emph{arc formula of $\phi$ along $R$}.  Instead of using $R$ as a subscript, we may also use its defining \zariski\ formula, or the  Artinian   scheme $Z:=\op{Spec}R$ it determines,  or even leave out reference to it altogether, whenever it is clear from the context. If $\phi$ is a \zariski\ formula defining an affine scheme $Y$, then we will   write  $\arc ZY$ for the affine scheme determined by $\arc R\phi$, and   call it the \emph{arc scheme of $Y$ along $Z$}.   The following shows that   arcs along an Artinian scheme are generalizations of truncated arcs.
  
  \begin{proposition}\label{P:arcsch}
Let   $Z$ be an Artinian  $\fld$-scheme, and let $X$ and $Y$ be affine $\fld$-schemes. There is a one-one correspondence between $Z\times_\fld X$-rational points on  $Y\times_\fld X$ over $X$, and   $X$-rational points on the corresponding arc scheme $\arc ZY$ over $\fld$, that is to say, we have a one-one correspondence
$$
\mor X{Z\times_\fld X}{Y\times_\fld X}\iso \mor \fld X{\arc ZY}.
$$
\end{proposition}
\begin{proof}
Let $\phi$ be the \zariski\ formula defining $Y$, and let $R$ and $X$ be the respective coordinate rings of $Z$ and $X$. Viewing $\phi$ in the language $\mathcal L_X$, it is the \zariski\ formula defining the affine $X$-scheme $Y\times X$, by \S\ref{s:bc}. By Lemma~\ref{L:ratpt}, we may identify $\inter\phi {R\tensor_\fld A}$ with $\mor \fld {Z\times X}{Y\times X}$, and similarly, $\inter{\arc R\phi}A$ with $\mor \fld X {\arc ZY}$. The result  then follows from Proposition~\ref{P:arc}.
\end{proof}

In particular, $\mor \fld ZY=\arc ZY(\fld)$. For instance, if $Z_n:=\pol \fld \xi/\xi^n\pol \fld \xi$, then $(\arc {Z_n}Y)^{\text{red}}$ is the truncated arc space  $\mathcal L_n(Y)$ as  defined in \cite[p.~276]{BLR} or \cite{DLArcs}. 

\begin{remark}
Given a morphism of schemes $Z\to Y$ over an arbitrary base scheme $S$, we may view $\mor SZY$ as a contravariant functor on the category of $S$-schemes through base change, that is to say,
$$
\mor SZY(X):=\mor X{Z\times_SX}{Y\times_SX},
$$
for any $S$-scheme $X$. The content of Proposition~\ref{P:arcsch} is then that for any Artinian scheme $Z$ over an \acf\ $\fld$, and any affine $\fld$-scheme $Y$, the functor $\mor \fld ZY$ is representable. Indeed, in the definition of representability, it suffices to consider only affine $\fld$-schemes $X$, since $\mor \fld ZY$ is compatible with limits, yielding that the affine $\fld$-scheme $\arc ZY$ represents $\mor \fld ZY$.
\end{remark}
      
  \begin{example}\label{E:arc}
Before we proceed, some simple examples are in order. It is clear from the definitions that $\arc R\lambda=\lambda_l\iso \lambda^l$, for $\lambda$ the Lefschetz formula $\var=\var$.

Let us next calculate the arc scheme of the curve given by the formula $\phi:=\sym{\var^2=\vary^3}$ along the four dimensional algebra  $R:=\pol \fld{\xi,\zeta}/(\xi^2,\zeta^2)\pol \fld{\xi,\zeta}$, using the basis $\Delta:=\{1,\xi,\zeta,\xi\zeta\}$ (in the order listed), and corresponding arc variables  the  quadruples $\tilde\var=(\tilde\var_{(0,0)},\tilde\var_{(1,0)},\tilde\var_{(0,1)},\tilde\var_{(1,1)})$ and $\tilde\vary=(\tilde\vary_{(0,0)},\tilde\vary_{(1,0)},\tilde\vary_{(0,1)},\tilde\vary_{(1,1)})$. One easily calculates that $\arc R\phi$ is the \zariski\ formula
$$
\begin{aligned}
    \tilde\var_{(0,0)}^2&=\tilde\vary_{(0,0)}^3   \\
    2\tilde\var_{(0,0)}\tilde\var_{(1,0)}&=3\tilde\vary_{(0,0)}^2\tilde\vary_{(1,0)} \\  
     2\tilde\var_{(0,0)}\var_{(0,1)}&=3\tilde\vary_{(0,0)}^2\vary_{(0,1)} \\  
     2\tilde\var_{(0,0)}\var_{(1,1)}+2\tilde\var_{(1,0)}\var_{(0,1)}&=3\tilde\vary_{(0,0)}^2\vary_{(1,1)}+6\vary_{(0,0)}\tilde\vary_{(1,0)}\vary_{(0,1)}.
\end{aligned}
$$
Note that the first equation is $\phi(\tilde\var_{(0,0)}, \tilde\vary_{(0,0)})$, and that above the singular point $\tilde\var_{(0,0)}=0= \tilde\vary_{(0,0)}$, the fiber consist of two $4$-dimensional planes.
\end{example}

\begin{example}\label{E:dual}
Another example  is classical: let $R=\pol \fld \xi/\xi^2\pol \fld \xi$ be the ring of dual numbers. Then one verifies that a $\fld$-rational point on $\arc RY$ is given by a $\fld$-rational point $P$ on $Y$, and a tangent vector  $\tuple v$ to $Y$ at $P$, that is to say, an element in the kernel of the Jacobian matrix $\jac\phi (P)$. 
\end{example}

\begin{example} 
As a last example, we calculate $\arc {Z_n}{Z_m}$, where $Z_n:=\op{Spec}(\pol \fld \xi/\xi^n\pol \fld \xi)$. With $\dot \var=\tilde\var_0+\tilde\var_1\xi+\dots+\tilde\var_{n-1}\xi^{n-1}$, we will expand $\dot \var^m$ in the basis $\{1,\xi,\dots,\xi^{n-1}\}$ of $\pol \fld \xi /\xi^n\pol \fld \xi$ (see Lemma~\ref{L:uniquearc} below for why the choice of  basis is not important);  the coefficients of this expansion then generate the ideal of definition   of $\arc{Z_n}{Z_m}$. A quick calculation shows that these generators are the polynomials
$$
g_s(\tilde\var_0,\dots,\tilde\var_{n-1}):=\sum_{i_1+\dots+i_m=s}\tilde\var_{i_1}\tilde\var_{i_2}\cdots \tilde\var_{i_m}
$$
for $s=0,\dots,n-1$, where the $i_j$ run over $\{0,\dots,n-1\}$. Note that $g_0=\tilde\var_0^m$. One shows by induction that $(\tilde\var_0,\dots,\tilde\var_s)\pol \fld {\tilde\var}$ is the unique minimal prime  ideal of $\arc{Z_n}{Z_m}$, where $s=\round nm$ is the \emph{round-up}  of $n/m$, that is to say, the  least integer greater than or equal to $n/m$. In particular, $\arc{Z_n}{Z_m}$ is irreducible of dimension $n-\round nm$.
\end{example}

Although the  arc scheme depends on the choice of basis, we have:

  \begin{lemma}\label{L:uniquearc}
For each finite dimensional $\fld$-algebra $R$, and each $\mathcal L_\fld$-formula $\phi$, the   arc formula $\arc R\phi$ along $R$ is unique up to an explicit isomorphism modulo $\arttheory  \fld$.
\end{lemma}
\begin{proof}
Suppose $R$ has dimension $l$ over $\fld$, and let $\Delta$ and $\Delta^*$ be two bases of $R$, with corresponding isomorphisms $\pi$ and $\pi_*$ between $R$ and $\fld^l$, and corresponding arc maps $\arc {}{}$ and $\arc *{}$ on $\mathcal L_\fld$. There exists an $\fld$-linear automorphism $\sigma$ of $R$ sending     $\Delta$ to $\Delta^*$. Applying $\sigma$ to $r=\sum\pi_\alpha(r)\alpha$ yields $\sigma(r)=\sum\pi_\alpha(r)\sigma(\alpha)$, showing that 
\begin{equation}\label{eq:auto}
\pi_*(\sigma(r))=\pi(r),
\end{equation}
 for any $r\in R$. 
Define $\tau$ as the automorphism $\pi\after\inv\sigma\after\inv\pi$ of $\fld^l$. I claim that the explicit formula $\tilde\vary=\tau(\tilde\var)$ induces an isomorphism between $\arc{}\phi(\tilde\var)$ and $\arc *\phi(\tilde\vary)$ modulo $\theory T_\fld$ (whence also modulo $\arttheory\fld$), for any formula $\phi$. Indeed, let $A$ be a finitely generated $\fld$-algebra, and let 
  $u\in\inter{\arc{}\phi}A$. Put $a:=\inv\pi{u}$, where we continue to write $\pi$ for the base change $A':=R\tensor_\fld A\to A^l$. Applying Proposition~\ref{P:arc} twice, we get $a\in\inter\phi {A'}$ whence $\pi_*(a)\in\inter{\arc*\phi} A$. Since  $\tau(u)=\pi(\inv\sigma(a))=\pi_*(a)$ by a component-wise application of \eqref{eq:auto}, we showed that $\tau$ induces the desired isomorphism between $\inter{\arc{}\phi}A$ and $\inter{\arc*\phi}A$.
\end{proof}

\begin{remark}\label{R:goodbasis}
So, from now on, we may choose a basis $\Delta=\{\alpha_0,\dots,\alpha_{l-1}\}$  of $(R,\maxim)$ with some additional properties. In particular, unless noted explicitly, we will always assume that the first base element is $1$ and that the remaining ones belong to $\maxim$. Moreover, once the basis is fixed, we let $\tilde\var$   be  the $l$-tuple of arc variables $(\tilde\var_0,\dots,\tilde\var_{l-1})$, so that  $\dot\var=\tilde\var_0+\tilde\var_1\alpha_1+\dots+\tilde\var_{l-1}\alpha_{\l-1}$ is   the corresponding generic arc. It follows from \eqref{eq:ethcomp} that $\coor 0f=f(\tilde\var_0)$, for any $f\in \pol \fld\var$, where henceforth we simply write $\coor jf$ for $\coor {\alpha_j}f$.  By \cite[\S2.1]{SchEC}, we may choose $\Delta$  so that, with  $\id_i:=(\alpha_i,\dots,\alpha_{l-1})R$, we have a  Jordan-Holder composition series\footnote{Writing $R$ as a homomorphic image of $\pol \fld\vary$ so that      $\vary:=\rij\vary e$ generates $\maxim$, let $\id(\alpha)$, for $\alpha\in\zet^n_{\geq 0}$,  be the ideal in $R$ generated by all $\vary^\beta$ with $\beta$ lexicographically larger than $\alpha$. Then we may take $\Delta$ to be all monomials $\vary^\alpha$ such that $\vary^\alpha\notin\id(\alpha)$, ordered lexicographically.}  
$$
\id_l=0\varsubsetneq\id_{l-1}\varsubsetneq \id_{l-2}\varsubsetneq\dots \varsubsetneq\id_1=\maxim\varsubsetneq\id_0=R.
$$

I claim that $\pi_j$ vanishes on each element in $\id_j$ for $j<i$. Indeed, if not, let $j<i$ be minimal so that there exists a counterexample with $r_j:=\pi_j(r)\neq0$ for some $r\in\id_i$. By minimality, $r=r_j\alpha_j+r_{j+1}\alpha_{j+1}+\dots\in \id_i$ showing that $\alpha_j\in\id_{j+1}$, since $r_j$ is invertible. However, this implies that $\id_j=\id_{j+1}$, contradiction. 
 
 From this, it is now easy to see that the first $m$ basis elements of $\Delta$ form a basis of $R_m:=R/\id_{m+1}$. Put differently, if $r\in R$, then the $m$-tuple $\pi_{R_m}(r)$ is the initial part of the $l$-tuple $\pi_R(r)$. Therefore, calculating $\coor mf$ for $f\in\pol \fld\var$ does not depend on whether we work with $\pi_R$ or with $\pi_{R_m}$, and hence, in  particular, $\coor mf\in\pol \fld{\tilde\var_0,\dots,\tilde\var_m}$ for every $m< l$.
 \end{remark}

The next result together with Corollary~\ref{C:fib} below shows that arc schemes are functorial  fibrations:

\begin{theorem}\label{T:arc}
Let $Z$ be a local Artinian $\fld$-scheme of length $l$. For each affine $\fld$-scheme $X\sub\affine \fld n$, the projection $\affine \fld{ln}\to \affine \fld n$ onto the first $n$ coordinates induces a   split surjective map $\arc ZX\to X$, which  is smooth above the regular locus of $X$. If $h\colon Y\to X$ is a morphism of affine $\fld$-schemes, then we have an induced morphism $\arc Zh\colon \arc ZY \to \arc ZX$ making the diagram 
\commdiagram[arc]{\arc ZY}{\arc Z h}{\arc ZX}{}{}{Y}{h}X
commute. Moreover, if $h$ is   a closed immersion, then so is $\arc Zh$. If $Y\sub X$ is  an open immersion, then we even we have an isomorphism
\begin{equation}\label{eq:pbopen}
 \arc Z Y\iso \arc Z X\times_XY.
\end{equation}
\end{theorem}
\begin{proof}
Let $(R,\maxim)$ be the Artinian local ring $\loc_Z$ corresponding to $Z$, and calculate $\pi:=\pi_R$ with the basis   given as in Remark~\ref{R:goodbasis}.  
 The projection $\affine \fld{ln}\to \affine \fld n$ is given by the embedding $\pol \fld\var\to\pol \fld{\tilde\var}\colon \var\mapsto\tilde\var_0$.  Let $\phi$ be the \zariski\ formula defining $X$, let  $I:=I(\phi)$ be the corresponding ideal, and let $A:=\pol \fld\var/I$ be its coordinate ring. Furthermore, let $\tilde I:=I(\arc{} X)$ be the ideal defining $\arc {}X:=\arc  ZX$, and let $\tilde A:=\pol \fld{\tilde\var}/\tilde I$ be its coordinate ring. 
The existence of the map  $\arc{} X\to X$ follows from our observation  in Remark~\ref{R:goodbasis} that $\coor 0f=f(\tilde\var_0)$. Namely, applying this to every equation in $\phi$, we see that   the explicit formula $\tilde\var_0=\var$ defining the projection $\affine \fld{ln}\to \affine \fld n$  induces a morphism $\phi\to \arc {}\phi$, whence a \homo\ $A\to \tilde A$, that is to say, a morphism $\arc ZX\to X$. Let $\mathfrak b$ be the ideal in $\pol \fld{\tilde\var}$ generated by all $\tilde\var_i$ with $0<i<l$. Since $\dot\var\equiv\tilde\var_0\mod \mathfrak b\pol R{\tilde\var}$, equation~\eqref{eq:ethcomp} yields that all $\coor uf$ belong to $\mathfrak b$, for $u>0$. Hence $\tilde A/\mathfrak b \tilde A\iso A$, showing that $\arc {} X\to X$ has  a section, whence is split surjective. 

If $Y\to X$ is a morphism of affine $\fld$-schemes, then this corresponds by Corollary~\ref{C:zarmor} to an explicit morphism $\phi\to\psi$, where $\psi$ is the \zariski\ formula defining $Y$. We leave it to the reader to verify that this induces an explicit morphism $\arc {}\phi\to\arc {}\psi$, leading to a commutative diagram~\eqref{arc}. Suppose $Y\to X$ is a closed immersion. We may assume that $Y$ is a closed subscheme of $X$, and hence the \zariski\ formula of $Y$ can be taken to be a conjunction of the form $\phi\en\psi$, with $\psi$ some \zariski\ formula. Therefore, $\arc {}(\phi\en\psi)=\arc {}\phi\en\arc {}\psi$ is the \zariski\ formula for $\arc {} Y$, showing that it is a closed subscheme of $\arc {} X$. Next suppose $Y\to X$ is an open immersion. We may reduce to the case that $Y$ is a basic open subset of $X$, since $\arc {}{}$ is compatible with disjuncts/unions. By the case of a closed immersion just proved and the fact that $\arc {}{}$ also preserves intersections,  we may furthermore reduce to the case that $Y=\op D(f)\sub X=\affine \fld n$ for some non-zero $f\in\pol \fld\var$. Hence $Y$ is defined, as a closed subscheme of $\affine \fld{n+1}$, by $g(\var,\vary):=\vary f(\var)-1$. Let  $\tilde\vary=(\tilde\vary_0,\dots,\tilde\vary_{l-1})$ be arc variables  with corresponding generic arc  $\dot\vary:=\tilde\vary_0+\dots+\tilde\vary_{l-1}\alpha_{l-1}$. Let $\tilde J\sub\pol \fld{\tilde\var, \tilde\vary}$ be the ideal defining  $\arc {} {Y}$, that is to say, the ideal generated by all $\coor ug$, and let  $\tilde B:=\pol \fld{\tilde\var, \tilde\vary}/\tilde J$ be the coordinate ring of $\arc {} Y$. Our aim is to show that $\arc {} Y$ is isomorphic to $\op D(\coor 0f)\sub\affine \fld{ln}$. Using  \eqref{eq:ethcomp}, we get
$$
\sum_{u=0}^{l-1}\coor ug(\tilde\var, \tilde\vary)\alpha_u=g(\dot\var,\dot\vary)=  \dot\vary\cdot \Big(\sum_{u=0}^{l-1}\coor uf(\tilde\var)\alpha_u \Big)-1.
$$
Expansion yields $\coor 0f=f(\tilde\var_0)$ and $\coor 0g=\tilde\vary_0\coor 0f-1$. Hence under the canonical \homo\ $\pol \fld\var\to \tilde B$ given by $\var\mapsto \tilde\var_0$, we get 
$\tilde\vary_0f=1$ in $\tilde B$.   By  Remark~\ref{R:goodbasis},  in order to calculate $\coor ug$ for $u>0$, we may ignore all terms containing some $\alpha_i$ with $i>u$, that is to say, 
$$
\begin{aligned}
   \coor ug&=\coor u{\Big((\tilde\vary_0+\dots+\tilde\vary_u\alpha_u)(\coor 0f+\dots+(\coor uf)\alpha_u )\Big)} &   \\
    & =\tilde\vary_u\coor0 f+\coor u {\Big((\tilde\vary_0+\dots+\tilde\vary_{u-1}\alpha_{u-1})(\coor 0f+\dots+(\coor uf)\alpha_u) \Big)}.
\end{aligned}
$$
Note that the second term lies in $\pol \fld{\tilde\var,\tilde\vary_0,\dots,\tilde\vary_{u-1}}$. Since $\coor 0f=f$ and  since $y_0f=1$ and $\coor ug=0$ in $\tilde B$, we obtain, after multiplying this sum by $\tilde\vary_0$,   that  $\tilde\vary_u$ lies in the  $\fld$-subalgebra of $\tilde B$ generated by all $\tilde\var$ and all $\tilde\vary_j$ with $j<u$. Hence, by downward induction  on $u$, we get
$$
\tilde B\iso \pol \fld{\tilde\var,\tilde\vary_0}/(\tilde\vary_0f-1)\pol \fld{\tilde\var,\tilde\vary_0},
$$
 showing that $\arc {} Y=\op D(f)$, as claimed. In particular,  $\arc {} Y$ is the pull-back of $Y$ under the map $\arc {} X\to X$, that is to say,  \eqref{eq:pbopen} holds.
 
So remains to show that if $P$ is a closed point in the regular locus of $X$, then the fiber of $\arc {} X\to X$ at $P$ is non-singular.  By the Nullstellensatz, we may, after a change of variables (translation),  assume that $P$ corresponds to the maximal ideal $\mathfrak n:=\rij\var n\pol \fld\var$. Since $\loc_{X,P}=A_{\mathfrak nA}$ is a regular local ring, $I\pol \fld\var_{\mathfrak n}$ is generated by a regular system of parameters (\cite[Theorem 14.2]{Mats}), say, of length $h$. Hence, by Nakayama's Lemma, we can find an open $U\sub \affine \fld n$ containing $P$, such that $I\loc_U$ is generated by $h$ elements whose image in $\loc_{U,P}$ are part of a generating system of $\mathfrak n$. Since $\arc {} {(U\cap X)}$ is just the pull-back of $U\cap X$ by \eqref{eq:pbopen}, and since the present question is not affected by such a pull-back, we may take $X=U$, and assume that $I=\rij fh\pol \fld\var$, with   the $f_u$   part of a minimal system of generators of $\mathfrak n$. In particular, the linear parts of the $f_u$ must be linearly independent over $\fld$. Hence after a linear change of variables (rotation), we may assume that $f_u=\var_u+g_u$, with each $g_u\in\mathfrak n^2$, for $j=\range 1h$.   Using \eqref{eq:ethcomp}, we get
$$
f_v(\dot\var)= \tilde\var_{0,v}+\tilde\var_{1,v}\alpha_1+\dots+\tilde\var_{l-1,v}\alpha_{l-1}+\sum_{u=0}^{l-1}\coor u{g_v}(\tilde\var)\alpha_u
$$
for $v=\range 1h$, showing that $\arc {}\phi$ is the conjunction of the $lh$ equations $\coor u{f_v}=\tilde\var_{u,v}+\coor u{g_v}=0$, with $u=\range 0{l-1}$ and $v=\range 1h$.

Let $J(\tilde\var)$ be the $lh\times lh$-submatrix
$$
\left(
\begin{matrix}
\frac{\partial (\tilde\var_{0,1}+\coor 0{g_1})}{\partial\tilde\var_{0,1}} &
\dots& 
\frac{\partial (\tilde\var_{0,1}+\coor 0{g_1})}{\partial\tilde\var_{l-1,1}} &
\frac{\partial (\tilde\var_{0,1}+\coor 0{g_1})}{\partial\tilde\var_{0,2}} &
\dots&
\frac{\partial (\tilde\var_{0,1}+\coor 0{g_1})}{\partial\tilde\var_{l-1,h}} 
\\
\phantom{
\frac{\partial (\tilde\var_{1,1}+\coor 1{g_1})}{\partial\tilde\var_{1,1}} }
&&
\\
\frac{\partial (\tilde\var_{1,1}+\coor 1{g_1})}{\partial\tilde\var_{0,1}} & 
\dots& 
\frac{\partial (\tilde\var_{1,1}+\coor 1{g_1})}{\partial\tilde\var_{l-1,1}} &
\frac{\partial (\tilde\var_{1,1}+\coor 1{g_1})}{\partial\tilde\var_{0,2}} &
\dots&
\frac{\partial (\tilde\var_{1,1}+\coor 1{g_1})}{\partial\tilde\var_{l-1,h}} 
\\
\vdots
&\ddots&\vdots&\vdots&\ddots&\vdots\\
\frac{\partial (\tilde\var_{l-1,1}+\coor {l-1}{g_1})}{\partial\tilde\var_{0,1}} & 
\dots& 
\frac{\partial (\tilde\var_{l-1,1}+\coor {l-1}{g_1})}{\partial\tilde\var_{l-1,1}} &
\frac{\partial (\tilde\var_{l-1,1}+\coor {l-1}{g_1})}{\partial\tilde\var_{0,2}} &
\dots&
\frac{\partial (\tilde\var_{l-1,1}+\coor {l-1}l{g_1})}{\partial\tilde\var_{l-1,h}}
\\
\phantom{
\frac{\partial (\tilde\var_{l,1}+\coor 1{g_1})}{\partial\tilde\var_{1,1}} }
&&
\\
\frac{\partial (\tilde\var_{0,2}+\coor 0{g_2})}{\partial\tilde\var_{0,1}} & 
\dots& 
\frac{\partial (\tilde\var_{0,2}+\coor 0{g_2})}{\partial\tilde\var_{l-1,1}} &
\frac{\partial (\tilde\var_{0,2}+\coor 0{g_2})}{\partial\tilde\var_{0,2}} &
\dots&
\frac{\partial (\tilde\var_{0,2}+\coor 0{g_2})}{\partial\tilde\var_{l-1,h}} 
\\
\vdots
&\ddots&\vdots&\vdots&\ddots&\vdots\\
\frac{\partial (\tilde\var_{l-1,h}+\coor {l-1}{g_h})}{\partial\tilde\var_{0,1}}  & 
\dots& 
\frac{\partial (\tilde\var_{l-1,h}+\coor {l-1}{g_h})}{\partial\tilde\var_{l-1,1}} &
\frac{\partial (\tilde\var_{l-1,h}+\coor {l-1}{g_h})}{\partial\tilde\var_{0,2}} &
\dots&
\frac{\partial (\tilde\var_{l-1,h}+\coor {l-1}{g_h})}{\partial\tilde\var_{l-1,h}} 
\end{matrix}\right)
$$
  of the Jacobian $\jac{\arc {}\phi}(\tilde\var)$, and, for each $1\leq v,j\leq h$, let $J_{v,j}$ be the $l\times l$ submatrix $(\partial(\tilde\var_{u,v}+\coor u{g_v})/\partial \tilde\var_{i,j})_{u,i=\range 0{l-1}}$ of $J$. I claim that each $J_{v,j}$ is an upper-triangular matrix with diagonal entries all equal to $\partial f_v/\partial  \var_j$ (under the canonical embedding $\pol \fld\var\to \pol \fld{\tilde\var}$).  Assuming the claim, let $\tilde P$ be an arbitrary closed point in the fiber above $P$. In particular, if $\tilde{\mathfrak n}\sub\pol \fld{\tilde\var}$  is the maximal ideal corresponding to $\tilde P$, then $\mathfrak n= \tilde{\mathfrak n}\cap \pol \fld\var$. Hence, the determinant of $J_{v,j}$ evaluated at $\tilde P$ is equal to $(\partial f_v/\partial  \var_j)^l(P)$, which is equal to the Kronecker delta function  $\delta_{v,j}$. This shows that $J(\tilde P)$ is an upper-triangular matrix with  determinant equal to one,    from which it follows that 
$\jac{\arc {}\phi}(\tilde P)$   has rank at least $lh$. Since $I(\arc {}\phi)$ has height at most $lh$ by Krull's Principal Ideal theorem, we showed, using the  Jacobian criterion for smoothness (\cite[Theorem 16.19]{Eis}), that $\arc {} X$ is smooth at $\tilde  P$, as we wanted to show. 

So remains to prove the claim. For $f\in\pol \fld\var$, we apply the $(i,j)$-th partial derivative to both sides of \eqref{eq:ethcomp}. On the left hand side, we get, using the chain rule
\begin{equation*} 
\begin{aligned}
\frac {\partial \big(f(\dot\var)\big)}{\partial \tilde\var_{i,j}}&=\frac {\partial f}{\partial \var_j}(\dot\var)\cdot \frac{\partial(\dot\var_j)}{\partial\tilde\var_{i,j}}\\
&=\frac {\partial f}{\partial \var_j}(\dot\var)\cdot\alpha_i.
\end{aligned}
\end{equation*}
Doing the same to the right hand side of \eqref{eq:ethcomp}, we get
\begin{equation}\label{eq:rightpartial}
\frac {\partial f}{\partial \var_j}(\dot\var)\cdot\alpha_i=\sum_{u=0}^{l-1}\frac{\partial (\coor uf)}{\partial\tilde\var_{i,j}} \alpha_u.
\end{equation}
In view of  \eqref{eq:ethcomp} for $\partial f/\partial \var_j$, 
the left hand side of \eqref{eq:rightpartial} becomes 
$$
\Big(\frac{\partial f}{\partial \var_j} +\coor 1{(\frac{\partial f}{\partial \var_j})}\alpha_1+\dots+\coor {l-1}{(\frac{\partial f}{\partial \var_j})}\alpha_{l-1}\Big)\cdot\alpha_i.
$$
By the choice of  basis (which remains a basis for the base change $\pol R{\tilde\var}$ over $\pol \fld{\tilde\var}$), the coefficient of $\alpha_i$ in this product is $\partial f/\partial \var_j$. Hence comparing this with the right hand side of \eqref{eq:rightpartial}, we obtain
$$
\frac{\partial f}{\partial \var_j} = \frac{\partial (\coor if)}{\partial \tilde\var_{i,j}}.
$$
Furthermore,   
since the left hand side of    \eqref{eq:rightpartial}   belongs to $ \id_j\pol R{\tilde\var}$,   all $\partial (\coor uf)/\partial \tilde\var_{i,j}$ must be zero for $u<i$, by the choice of  basis (see Remark~\ref{R:goodbasis}). Applied with $f=f_v=\var_v+g_v$,    the claim follows from this.
\end{proof}

\begin{remark}\label{R:arcsect}
The section of $\arc ZX\to X$ given in the proof of the theorem is induced by $\coor 0{}$, and will be called the \emph{canonical section} of $\arc ZX\to X$. We will identify $X$ as a closed subscheme of $\arc ZX$ via this canonical section.
\end{remark}

\begin{corollary}\label{C:dimarc}
If $X$ is a $d$-dimensional affine variety and $Z$ an Artinian local scheme of length $l$, then $\arc ZX$ is an $ld$-dimensional variety.
\end{corollary}
\begin{proof}
Since $X$ is irreducible, it contains a dense open subset $U$ which is non-singular. By Theorem~\ref{T:arc}, the pull-back $\arc{} U=U\times_X\arc{} X$ is a dense open subset of $\arc{} X$. Moreover, Theorem~\ref{T:arc} also yields that $\arc{} U\to U$ is smooth. Since $U$ is non-singular, so is therefore $\arc{} U$. In particular, $\arc{} U$ is irreducible, whence so is $\arc{} X$. Moreover, by the proof of the theorem, $U$ is defined by an ideal of height $n-d$, and $\arc{} U$ by an ideal of height $l(n-d)$. Hence $\arc{} U$ has dimension $ln-l(n-d)=ld$, whence so does $\arc{} X$.
\end{proof}

\begin{corollary}\label{C:split}
If $\bar Z\sub Z$ is a closed immersion of Artinian local $\fld$-schemes, then for any affine $\fld$-scheme $X$, we have a split surjection $\arc ZX\to \arc {\bar Z}X$, making the diagram
\commtrianglefront {\arc ZX}{}{\arc {\bar Z}X}{}X{}
commute.
\end{corollary}
\begin{proof}
Let $\phi$ be the \zariski\ formula defining $X$, and let $A:=\pol \fld\var/I(\phi)$ be its coordinate ring, with $\var$ an $n$-tuple of variables. Let $R$ and $\bar R$ be the corresponding Artinian local rings, and let $l$ and $\bar l$ be their respective lengths. Let  $\id$ be the kernel of the epimorphism $R\to \bar R$. We may choose a basis $\Delta=\{\alpha_0,\dots,\alpha_{l-1}\}$ of $R$ such that $\alpha_j\in\id$ for $j\geq \bar l$. Hence the images of the first $\bar l$ elements of $\Delta$ form a basis $\bar\Delta$ of $\bar R$. In view of Lemma~\ref{L:uniquearc}, we may use these two bases to calculate $\arc R\phi$ and $\arc {\bar R}\phi$. Let   $\tilde\var$ be the $ln$-tuple of arc variables. Let us denote the $\bar ln$ first variables in $\tilde\var$ by $\bar\var$, and let $J$ be the ideal generated by the remaining variables. Hence $\tilde A:=\pol \fld{\tilde\var}/I(\arc R\phi)$ and $\bar A:=\pol \fld{\bar\var}/I(\arc {\bar R}\phi)$ are the respective coordinate rings of $\arc ZX$ and $\arc {\bar Z}X$. For $f\in\pol \fld\var$, equation~\eqref{eq:ethcomp} shows that $\coor if$ belongs to $\pol \fld{\bar\var}$ for $i< \bar l$, and belongs to $J$ for $i\geq \bar l$. Hence the embedding $\pol \fld{\bar\var}\to\pol \fld{\tilde\var}$ induces a \homo\ $\bar A\to \tilde A$. Moreover,  the isomorphism  $\tilde A/J\tilde A\iso \bar A$ shows that this \homo\ is split.
\end{proof}

Note that if we let $\bar Z=\op{Spec}\fld$ be the closed subscheme given by the residue field, then $\arc {\bar Z}X=X$, showing that   $\arc  ZX\to X$ is induced by the residue map.

\begin{corollary}\label{C:arcmap}
Each Artinian $\fld$-scheme $Z$ induces an  endomorphism  $\arc Z{}$, called the \emph{arc map along $Z$}, on   $ \grotsch \fld$ (respectively, on $\grotart  \fld{}$), by sending a class $\class\phi$ to the class $\class{\arc Z\phi}$.
\end{corollary}
\begin{proof}
Let us show in general that if $\theta$ is a morphic formula giving a morphism $\phi\to \psi$, then $\arc{}\theta$ is also morphic and induces a morphism $\arc{}\phi\to \arc{}\psi$. We verify this on an arbitrary  $\fld$-algebra $A$. Let $ a\in\inter{\arc{}\phi}A$ and put $\tilde{ a}:=\inverse\pi{ a}$, so that $\tilde{ a}\in\inter\phi{\tilde A}$, where $\tilde A:=R\tensor_\fld A$ and $R$ is the Artinian coordinate ring of $Z$. Since $\theta$ is morphic, it satisfies the morphic conditions~\eqref{i:map}--\eqref{i:im}, and hence we can find $\tilde{ b}\in\inter\psi {\tilde A}$, such that $(\tilde{ a},\tilde{ b})\in\inter\theta {\tilde A}$. Let $ b:=\pi(\tilde{ b})$ so that $ b\in\inter{\arc{}\psi}A$ and $( a, b)\in\inter{\arc{}\theta}A$. In particular, $ a$ satisfies the morphic conditions~\eqref{i:map} and \eqref{i:im}, and by a similar, easy argument we also verify \eqref{i:fct}, showing that $\arc{}\theta$ defines a morphism $\arc{}\phi\to\arc{}\psi$. 

Since $\arc{}{}$ preserves explicit and pp-formulae by Lemma~\ref{L:uniquearc} and its proof,  our previous argument then shows that it also preserves   $\zar$-isomorphisms. 
So remains   to verify that $\arc{}{}$ also preserves scissor relations. This is clear, however,   since  $\arc{}{(\phi\of\psi)}=\arc{}\phi\of\arc{}\psi$ and $\arc{}{(\phi\en\psi)}=\arc{}\phi\en\arc{}\psi$ by the proof of Proposition~\ref{P:arc}. This also shows that $\arc{}{}$ preserves multiplication, showing that it is a ring endomorphism on either \gr. 
\end{proof}

\subsection*{Integration}
 Integration is derived from the arc map as follows. Let $X$ and $Z$ be affine $\fld$-schemes, with $Z$ Artinian. We define their \emph{arc-integral} as the class 
$$
\integral ZX:=  \class{ \arc Z X }
$$
in $\grotsch \fld$ (or in $\grotart \fld{}$). Corollary~\ref{C:arcmap} shows that  this is   well-defined.

\begin{proposition}
The arc-integral $\integral ZX$ only depends on the class of the Artinian scheme $Z$ in $\grotschzero \fld$ and the class of the $\fld$-scheme $X$ in $\grotsch\fld$. In particular, for any $  p\in\grotsch \fld$ and $  q\in\grotschzero \fld$, the arc-integral $\integral {  q}{  p}$ is a well-defined element of $\grotsch \fld$. 

The same result holds upon replacing $\grotsch \fld$ by $\grotart \fld{}$ everywhere.
\end{proposition}
\begin{proof}
The dependence on the class of $X$ follows from Corollary~\ref{C:arcmap}. Using Corollary~\ref{C:isosch}, one easily shows  that   $\integral ZX$   only depends on the isomorphism class of $Z$.  To show that it even   depends on the class of $Z$ in $
\grotschzero \fld$ only, we need to show that it vanishes on any scissor relation. By Lemma~\ref{L:incompscis} (using Lemmas~\ref{L:schdisjun} and \ref{L:disunion}), it suffices to verify this on the scissor relation $Z+Z'-Z\sqcup Z'$, with $Z$ and $Z'$ Artinian schemes. We need to show that
\begin{equation}\label{eq:RR}
\integral ZX+\integral {Z'}X=\integral{Z\sqcup Z'}X.
\end{equation}
Let $R$ and $R'$ be the corresponding Artinian $\fld$-algebras, and let $\phi$ be the \zariski\ formula defining $X$. Hence $R\oplus R'$ is the Artinian algebra corresponding to $Z\sqcup Z'$. Therefore, \eqref{eq:RR} is equivalent with showing that $\class{\arc R\phi}+\class{\arc {R'}\phi}=\class{\arc {R\oplus R'}\phi}$, and this, in turn will follow if we can show that the disjoint sum $\arc R\phi\oplus\arc {R'}\phi$ is  isomorphic with $\arc {R\oplus R'}\phi$ modulo $\arttheory \fld$. We verify this on an arbitrary model $S$ of $\arttheory \fld$. Using Lemma~\ref{L:disjsum} and Proposition~\ref{P:arc} both twice, we have
$$
\begin{aligned}
\inter{\arc {R\oplus R'}\phi} S &= \inter\phi{(R\oplus R')\tensor_\fld  S}\\
&= \inter\phi{(R\tensor_\fld S)\oplus (R'\tensor_\fld S)}\\
&\iso \inter\phi{R\tensor_\fld S}\sqcup\inter\phi{R'\tensor_\fld S}\\
&=\inter{(\arc R\phi)}S\sqcup \inter{(\arc {R'}\phi)}S\\
&=\inter{(\arc R\phi\oplus\arc {R'}\phi)}S,
\end{aligned}
$$
where the middle bijection is induced by the canonical bijection between $(R\oplus R')^n$ and $R^n\oplus (R')^n$. Since this bijection is easily seen to be induced by an (explicit) isomorphism, we completed the proof of \eqref{eq:RR}.

The same argument shows that $\integral \cdot X$ is additive, and hence can be defined on any element $q\in\grotschzero \fld $. 
By Lemma~\ref{L:cycle}, we can write an arbitrary element $p\in\grotsch \fld $ as a difference $\class X-\class{X'}$. We let 
$$
\integral qp:=\integral qX-\integral q{X'}.
$$
This is well-defined, as each arc map $\arc Z{}$ is a ring \homo, whence in particular additive. 
\end{proof}

Since $\class{\op{Spec}\fld }=1$ and the arc map $\arc {\op{Spec}\fld }{}$ is the identity, we get the suggestive formula
$$
\integral {}X=\class X.
$$
Since arc maps are also multiplicative, we get the following Fubini-type   formula:

\begin{proposition}\label{P:prodform}
For $Z$ an Artinian $\fld $-scheme and $X$ and $Y$ affine $\fld $-schemes, we have
$$
\integral Z{(X\times_\fld  Y)}=\integral ZX\cdot \integral Z{Y}
$$
in $\grotsch \fld $.\qed
\end{proposition}

The same product formula also holds in $\grotart \fld {}$, where  we may even drop the affineness assumption in view of Theorem~\ref{T:classsch}. Since the arc formula of the Lefschetz formula in one variable is   the Lefschetz formula in $\ell(R)$ variables, we immediately obtain from Lemma~\ref{L:length} that:

\begin{proposition}\label{P:intlef}
For any element $q\in\grotschzero \fld $, we have $\integral q\lef=\lef^{\ell(q)}$.\qed
\end{proposition}

The next result generalizes this to an arbitrary smooth scheme, provided we work in the \gr\ $\grotart \fld {}$; this is needed since we need the covering properties proven in Theorem~\ref{T:classsch}. We call a morphism $Y\to X$ of $\fld $-schemes a \emph{locally trivial fibration with fiber $W$} if for each (closed) point $P\in X$, we can find an open $U\sub X$ containing $P$ such that the restriction of $Y\to X$ to $U$ is isomorphic with the projection $U\times_\fld  W\to U$. 

\begin{corollary}\label{C:fib}
If $\bar Z\sub Z$ is a closed immersion of Artinian local $\fld $-schemes, and $X$ is a smooth $d$-dimensional   affine $\fld $-scheme, then $\arc ZX\to \arc {\bar Z}X$ is a locally trivial fibration with fiber $\affine \fld {dm}$, where $m=\ell(Z)-\ell(\bar Z)$. In particular, 
$$
\integral ZX=\class X\cdot\lef^{d\ell(Z)-d}
$$
 in $\grotart \fld {}$.
\end{corollary}
\begin{proof}
Let $R$ and $\bar R$ be the Artinian local coordinate rings of $Z$ and $\bar Z$ respectively, and let $\phi$  be   the \zariski\ formula $f_1=\dots=f_s=0$ defining $X\sub\affine \fld m$. We will write $\arc {}{}$ and $\arcbar{}$ for the respective arc maps $\arc Z{}$ and $\arc {\bar Z}{}$, and similarly, $\pi$ and $\bar\pi$ for the isomorphisms $R\iso \fld^{\ell(R)m}$ and $\bar R\iso \fld^{\ell(\bar R)m}$. Since the composition of locally trivial fibrations is again a locally trivial fibration, with general fiber the product of the fibers, we may reduce to the case that $\bar R=R/\alpha R$ with $\alpha $ an element in the socle of $R$, that is to say, such that $\alpha \maxim=0$, where $\maxim$ is the maximal ideal of $R$.  Let $l$ be the length of $R$, and let $\Delta$ be a basis of $R$ as   in Remark~\ref{R:goodbasis}, with $\alpha_{l-1}=\alpha $ (since $\alpha $ is a socle element, such a basis always exists). In particular, $\Delta-\{\alpha \}$ is a basis of $\bar R$. We will use these bases to calculate both arc maps.

We start with calculating a general fiber of the map $\arc {} X\to \arcbar X$. By Corollary~\ref{C:isosch}, it suffices to do this in an arbitrary model of $\arttheory \fld $, that is to say, to calculate the fiber of an $S$-rational point  on $\arcbar X$, where $S$ is any Artinian local $\fld $-algebra.  Let $\bar a$ be an $m$-tuple in $\inter\phi{\bar R\tensor_\fld S}$. By Proposition~\ref{P:arcsch}, its image $\bar\pi(\bar a)$ is an $(l-1)m$-tuple in  $\inter{\arcbar\phi}S$, corresponding  by Lemma~\ref{L:ratpt}, therefore,  to an $S$-rational point $\bar Q$ of $\arcbar X$, and any $S$-rational point of $\arcbar X$ is obtained in this way.  Let $P$ be the $S$-rational point on $X$ given as the image of $\bar Q$ under  the canonical morphism $\arcbar X\to X$.\footnote{The image of an $S$-rational point $\op{Spec}S\to Y$ under a morphism $Y\to X$ is simply the composition of these two morphisms, yielding an $S$-rational point on $X$.} Hence $a:= \bar\pi_0(\bar a)$ is the $m$-tuple in $\inter\phi S$ corresponding to    $P$.

 The surjection $R\to \bar R$ induces a surjection $R\tensor_\fld S\to \bar R\tensor_\fld S$ and hence a map $\inter\phi{R\tensor_\fld S}\to \inter\phi{\bar R\tensor_\fld S}$. The fiber above  $\bar a$ is  therefore   defined by the equations $f_j(\bar a+
 \tilde\var_{l-1}\alpha )=0$, for $j=\range 1s$. By Taylor expansion, this becomes
\begin{equation}\label{eq:fib}
 0=f_j(\bar a+\tilde\var_{l-1}\alpha )=\Big(\sum_{i=1}^m\frac{\partial f_j}{\partial \var_i}(\bar a)\tilde\var_{l-1,i}\Big)\alpha 
\end{equation}
 since    $f_j(\bar a)=0$ and $\alpha^2=0$ in $R\tensor_\fld S$. In fact, since $\bar a\equiv a\mod\maxim(R\tensor_\fld S)$ and $\alpha \maxim=0$,  we may replace each $\partial f_j/\partial \var_i(\bar a)$  in \eqref{eq:fib}  by $\partial f_j/\partial \var_i(a)$. In terms of $S$-rational points, therefore, the fiber above $\bar Q$ is the linear subspace of $S^m$ defined as the kernel of the Jacobian $(s\times n)$-matrix $\jac\phi(P)$. 

Since $X$ is non-singular at $P$, this matrix has rank $m-d$, and hence its kernel is a $d$-dimensional  linear subspace. This proves that the fiber is equal to $\affine Sd$. More precisely, we may choose $\phi$ so that the first $(m-d)\times(m-d)$-minor in $\jac\phi(P)$ is invertible. Therefore, by Kramer's rule,  we may express each $\tilde\var_{l-1,i}$ with $i\leq m-d$ as a linear combination of the $\tilde\var_{l-1,j}$ with $j>m- d$ on an open neighborhood of $\bar Q$. Since this holds in any $S$-rational point $\bar Q$ of $\arcbar X$, we showed that $h$ is a locally trivial fibration with   fiber $\affine \fld d$. 

Applying this to $\arc{} X\to X$, (note that $X= \arc \fld X$) we get a locally trivial fibration with   fiber equal to $\affine \fld {d(l-1)}$. The last assertion then follows from Lemma~\ref{L:fib} below. 
\end{proof}

\begin{lemma}\label{L:fib}
If $f\colon Y\to X$ is a locally trivial fibration of $\fld $-schemes with   fiber $Z$, then $\class Y=\class X\cdot\class Z$ in $\grotart \fld {}$.
\end{lemma}
\begin{proof}
By definition and compactness,  there exists a finite open covering $\{U_1,\dots,U_m\}$ of $X$, so that   
$$
\inverse f{U_i}\iso U_i\times_\fld Z,
$$
 for $i=\range 1m$. Taking classes in $\grotart \fld {}$,   Lemma~\ref{L:Lefsch} yields $\class{\inverse f{ U_i}}=\class{U_i} \cdot\class Z$.  Since the $\inverse f{ U_i}$ form an open affine covering of $Y$, Theorem~\ref{T:classsch} yields, after taking the $m$-th scissor polynomial on both sides, 
$$
\class{Y}=\class{S(\inverse f{ U_1},\dots,\inverse f{ U_m})}=\class{S\rij Um}\cdot\class Z=\class X\cdot\class Z
$$
in $\grotart \fld {}$.
\end{proof}

\begin{remark}
Example~\ref{E:arc} shows that over a singular point, the dimension of the fiber may increase. 
\end{remark}

\subsection*{Igusa-zeta series}
By a \emph{germ}, we mean a pair $(X,P)$ with $X$ an $\fld $-scheme and $P$ a closed point on $X$; if $X$ is a closed subscheme of $X^*$, then we also say that $(X,P)$ is a germ in $X^*$. For any   $\fld $-scheme $Y$, we can define the \emph{(geometric) Igusa-zeta series of  $Y$ along the germ $(X,P)$} as the formal power series
$$
\igu YXP(t):=\integral{\op{Hilb}_P(X)}Y=\sum_n\left(\integral{\jet PnX}Y\right)t^n=\sum_n\class{\arc{\jet PnX}Y}\cdot t^n
$$
in $\pow{\grotsch \fld }t$. Note that this is well-defined since each jet is Artinian. This definition generalizes the one in \cite{DLIgu} or \cite[\S4]{DLDwork}:

\begin{proposition}\label{P:Iguclass}
The Igusa-zeta series $\igu Y{\affine \fld 1}O$ of $Y$ along the germ of the origin on the affine line  is sent under the canonical \homo\ 
$$\pow{\grotsch \fld }t\to\pow{\grotclass \fld }t$$
 to the geometric Igusa-zeta function $\igugeom Y$ of $Y$. If $\fld $ has \ch\ zero, then this  image  is a rational function.
\end{proposition}
\begin{proof}
Let $P$ be the origin on the affine line.  By our discussion preceding Example~\ref{E:arc}, the arc-integral   $\integral{\jet Pn\lef}Y$ is equal to the class of the $n$-th truncated arc space $\mathcal L_n(Y)$, and hence the first assertion follows from the definition of the geometric Igusa-zeta function in \cite[\S4]{DLDwork}. Rationality over the classical \gr\ $\grotclass \fld $ is proven in \cite[Theorem 4.2.1]{DLDwork}.
\end{proof}

For curves, we can give an explicit formula for the  Igusa-zeta series of the Lefschetz class:

\begin{proposition}\label{P:lineint}
If $(C,P)$ is a germ  of a point of multiplicity $e$ on  a curve $C$ over $\fld $,    then 
$$
\igu {\affine \fld 1}CP=\frac {p(t)}{1-\lef^et}
$$
for some polynomial $p\in\pol{\grotsch \fld }t$.
\end{proposition}
\begin{proof}
By definition, the \emph{multiplicity} of the germ $(C,P)$ is the multiplicity of the local ring $\loc_{C,P}$. By   Hilbert theory, there exist $b,N\in\zet$ such that the length of $ {\jet PnC}$ for $n\geq N$  is equal to $en+b$. Using Proposition~\ref{P:intlef}, we get
$$
\integral{ {\jet PnC}}\lef=\lef^{en+b}
$$
for $n\geq N$. Hence $\igu {\affine \fld 1}CP$ is the sum of some polynomial of degree $N$ and the power series 
$$
\sum_n \lef^{en+b}t^n=\frac {\lef^b} {1-\lef^et},
$$
from which the assertion easily follows.
\end{proof}

The  above proof  shows that
\begin{equation}\label{eq:Xlef}
\igu {\affine \fld 1}XP=\sum_n\lef^{j^n_P(X)}t^n
\end{equation}
for any germ   $(X,P)$, where $j^n_P(X):=\ell(\jet PnX)$. 
For a smooth scheme, we have the following rationality result:

\begin{proposition}\label{P:ratsmooth}
 Let $(C,P)$ be a germ of multiplicity $e$ on a curve. For any $d$-dimensional smooth affine scheme $Y$ over $\fld $, the Igusa-zeta series of  $Y$ along the germ $(C,P)$ is a rational function over $\grotart \fld {}$. More precisely,
\begin{equation}\label{eq:lefint}
\igu YCP=\frac {p(t)}{1-\lef^{de}t}
\end{equation}
for some polynomial $p\in\pol{\grotart \fld {}}t$.
\end{proposition}
\begin{proof}
By   Hilbert theory, there exist $b,N\in\zet$ such that $j_P^n(C)=en+b$ for $n\geq N$. By Corollary~\ref{C:fib}, the coefficient of the $n$-th  term in $\igu YCP$ is therefore equal to $\class Y\cdot\lef^{d(en+b-1)}$ for $n\geq N$, and the result follows as in the proof of Proposition~\ref{P:lineint}.
\end{proof}

\subsection*{Motivic integrals}
If we integrate over a higher dimensional scheme, then \eqref{eq:Xlef} suggests that we should calibrate the Igusa-zeta series to maintain rationality. To this end, we will define a normalized integration, which more closely resembles the motivic integration of Kontsevich, Denef and Loeser (\cite{CrawMot,DLArcs,LooMot}).  More precisely, let $\grotschinv \fld $ and $\grotartinv \fld $ be the respective localization  of $\grotsch \fld $ and $\grotart \fld {}$ at $\lef$, that is to say, $\grotschinv \fld =\pol{\grotsch \fld }{\inv\lef}$ and $\grotartinv \fld =\pol{\grotart \fld {}}{\inv\lef}$. Let $X$ and $Z$ be  affine $\fld $-schemes with $Z$   Artinian. We define the \emph{motivic integral of $X$ along $Z$} to be 
$$
\motint ZX:=\lef^{-dl}\cdot\integral ZX=\lef^{-dl}\cdot\class{\arc ZX}
$$
with $d$ the dimension of $X$ and $l$ the length of $Z$, viewed either as an element in $\grotschinv \fld $ or $\grotartinv \fld $.

We define the \emph{motivic Igusa-zeta series of $Y$ along a germ $(X,P)$} as the formal power series
$$
\igumot YXP(t):=\motint{\op{Hilb}_P(X)}Y=\sum_n\lef^{-d\cdot j^n_P(X)}   \left(\integral{ {\jet PnX}}Y\right) t^n
$$
in $\pow{\grotschinv \fld }t$ (respectively, in $\pow{\grotartinv \fld }t$), where $d$ is the dimension of $Y$. In particular, by the same argument as in the proof of Proposition~\ref{P:ratsmooth}, we get 
$$
\igumot YXP=\frac{\class Y\cdot\lef^{-d}}{1-t},
$$
over $\grotartinv \fld $, for any germ $(X,P)$,  and any smooth affine $\fld $-scheme $Y$. This raises the following question:

\begin{conjecture}\label{C:motigusa}
If $\fld $ is an \acf\ of \ch\ zero, then, for any   affine $\fld $-scheme $Y$, the    motivic  Igusa-zeta series $\igumot YXP$ of $Y$  along an arbitrary germ $(X,P)$  is rational over $\grotartinv \fld $.
\end{conjecture}

\section{Infinitary \gr{s}}\label{s:infgr}

In this section, we will extend the previous definitions to include infinitary formulae. This turns out to be necessary when dealing with the complement  of an open subscheme, as we mentioned  already  in the introduction.

\subsection*{Formularies}
Let $\mathcal L$ be an arbitrary first-order language, and let $\Phi$ be a collection   of  $\mathcal L$-formulae of some fixed arity $n$, which we then call the \emph{arity} of $\Phi$. For an $\mathcal L$-structure $M$, let $\inter\Phi M$ be the subset in $M^n$ given as the union of all $\inter\phi M$ with $\phi\in\Phi$. In other words,  we view $\Phi$ as an infinitary disjunction, defining in each structure  a subset   $\inter\Phi M$ which is in general only infinitary definable (and its complement is type-definable).  We call $\inter\Phi M$ the \emph{interpretation of $\Phi$ in $M$}.

Since elementary classes  do no longer behave the same as their theories on infinitary formulae, we must shift our attention from the latter to the former. So, let $\mathfrak K$ be a class of $\mathcal L$-structures, and let $\Phi$ and $\Psi$ be two collections of $\mathcal L$-formulae of the same arity $n$. We say that $\Phi$ 
and $\Psi$ are \emph{$\mathfrak K$-equivalent}, if $\inter \Phi M=\inter\Psi M$ for all $M\in \mathfrak K$. In particular, if we let $\Phi^\of$ be the collection of all finite disjuncts of formulae in $\Phi$, then $\Phi^\of$ is $\mathfrak  K$-equivalent with $\Phi$, and so without loss of generality, we may always assume, up to equivalence, that a collection is closed under finite disjunctions.
 
 We say that $\Phi$ 
is a   \emph{\form\ with respect to $\mathfrak K$}, if for each structure $M\in \mathfrak K$, there is some $\phi\in \Phi$ such that $\inter\phi M=\inter\Phi M$. Note that $\phi$ will in  general depend on the structure $M$. Put differently, although in each $\mathfrak K$-structure $M$, the subset $\Phi(M)$ is definable, its  definition depends on the given structure $M$. We say that $\Phi$ is \emph{first-order}, if it is $\mathfrak K$-equivalent with a first-order formulae $\phi$. Although we do not insist in this definition on $\phi$   being part of $\Phi$, there is no loss of generality including it, since adding it yields an $\mathfrak K$-equivalent \form. In particular, we may view formulae as first-order \forms\ and we will henceforth identify both (up to $\mathfrak K$-equivalence).

Most of the logical operations generalize to this infinitary setting. Namely, given $\Phi\sub \mathcal L_n$ and $\Psi\sub\mathcal L_m$,   let $\Phi\en\Psi$ be the collection of all $\phi\en\psi$ with $\phi\in \Phi$ and $\psi\in\Psi$, and   let $\Phi\of\Psi:=\Phi\cup\Psi$. Similarly, we define $\Phi\times\Psi\sub\mathcal L_{n+m}$ as the collection of all $\phi\times\psi$. In case $\mathcal L$ contains the two constant symbols $0$ and $1$, we also define $\Phi\oplus\Psi$ as $\Phi'\cup\Psi'$ where $\Psi'$ consists of all $\phi\en(v_{n+1}=0)$ for $\phi\in\Phi$, and $\Psi'$ of all $\psi\en(v_{n+1}=1)$ for $\psi\in\Psi$ (where we assume $m\leq n$). We leave it to the reader to verify that all these operations preserve \forms. If a \form\ is not first-order, then we can, however, not define  its negation (in model-theoretic terms, the negation of a \form\ is a type).

 From now on, $\mathcal L:=\mathcal L_\fld $, the language of $\fld $-algebras, for $\fld $   an \acf, and   $\mathfrak K =\arttheory \fld $   the collection of all Artinian  local $\fld $-algebras.  
We call a \form\ $\Phi$ respectively  \emph{\zariski} or  \emph{pp} if all formulae in $\Phi$ are of that kind.

\subsection*{Total jets and formal schemes}
Let  $\phi$ be an arbitrary formula and $\zeta$ a \zariski\ formula. We define the \emph{total jet of $\phi$ along $\zeta$} as the collection $\Jet \zeta\phi$ of all $n$-jets $\jet\zeta n\phi$. Let us show that $\Jet\zeta\phi$ is a \form. Let   $\id:=I(\zeta)$ be the ideal of $\zeta$, and let $(R,\maxim)$ be an Artinian local $\fld $-algebra of length $l$. I claim  that $\inter{(\Jet\zeta\phi)}R$ is equal to $\inter{(\jet\zeta l\phi)}R$. To prove this, it suffices to show that $\inter{(\jet\zeta n\phi)}R=\inter{(\jet\zeta {n+1}\phi)}R$, for all $n\geq l$. One direction is clear, so assume that $  a\in\inter{(\jet\zeta{n+1}\phi)}R$. Hence, for each $f\in \id$, we have $f^{n+1}(  a)=0$, whence $f(  a)\in\maxim$. Since this holds for all $f\in \id$, we see that $g(  a)=0$ in $R$  for every $g\in \id^n$, as $\maxim^n=0$. In conclusion, $  a\in\inter{(\jet\zeta n\phi)}R$.

If $\phi$ and $\zeta$ are \zariski\ with $\zeta\dan\phi$, therefore corresponding to a closed immersion $Z\sub X$, then we also will write $\Jet ZX$ for $\Jet\zeta\phi$.  We can give the following geometric interpretation of    total jets:

 \begin{proposition}\label{P:formsch}
For $Y\sub X$   a closed immersion of affine $\fld $-schemes, there is, for every Artinian local $\fld $-algebra $R$, a one-one correspondence between the $R$-rational points of the formal scheme $\complet X_Y$ and the interpretation of the \form\ $\Jet YX$ in $R$.
\end{proposition}
\begin{proof}
Let $A:=\loc_X(X)$ be the coordinate ring of $X$, and $I$ the ideal defining $Y$, that is to say, $\loc_Y(Y)=A/I$. Hence $\jet YnX=\op{Spec}(A/I^n)$. By \cite[II.\S9]{Hart}, the formal completion  $\complet X_Y$ is the ringed space with underlying set equal to the underlying set of  $Y$ and with sheaf of rings the inverse limit of the sheafs $\loc_{\jet YnX}$. In particular, the ring of global sections is equal to the $I$-adic completion $\complet A$ of $A$. 

Let $(R,\maxim)$ be an Artinian local $\fld $-algebra, and let $\op{Spec}R\to \complet X_Y$ be an $R$-rational point of $\complet X_Y$ over $\fld $.
Taking global sections, we get a $\fld $-algebra \homo\ $\complet A\to R$.   Since the closed point of $\op{Spec}(R)$ is sent to a point in the underlying set of $Y$, we have  $IR\sub\maxim$. If $R$ has length $l$, then $I^l$ lies in the kernel of $\complet A\to R$, and so we get a factorization $\complet A\to \complet A/I^l\iso A/I^l\to R$, that is to say, an $R$-rational point $\op{Spec}R\to \jet YlX$. By Lemma~\ref{L:ratpt}, this corresponds to a tuple  in $\inter{(\jet YlX)}R$ whence in $\inter{(\Jet YX)}R$, where, as before, we identify   jets with the \zariski\ formulae defining them. Conversely, a tuple   in $\inter{(\Jet YX)}R$ lies in some $\inter{(\jet YnX)}R$, and hence, by Lemma~\ref{L:ratpt} yields a $\fld $-algebra \homo\ $A/I^n\to R$. Composition with the canonical surjection $\complet A\to A/I^n$ then induces an $R$-rational point on $\complet X_Y$.
\end{proof}

This proposition allows us to identify the total jet $\Jet YX$ with the formal scheme $\complet X_Y$, which we henceforth will do.

\begin{lemma}\label{L:prodjets}
If $Y_i\sub X_i$ for $i=1,2$ are two closed immersions, then we have the following product formula
$$
(\Jet{Y_1}{X_1})\times_\fld (\Jet{Y_2}{X_2})\iso \Jet{Y_1\times_\fld  Y_2}{(X_1\times_\fld  X_2)}.
$$
\end{lemma}
\begin{proof}
Let $A_i$ be the coordinate ring of $X_i$, and $I_i$ the ideal defining $Y_i$. Hence $X_1\times X_2$ has coordinate ring $A:=A_1\tensor_\fld  A_2$ and $Y_1\times Y_2$ is defined by the ideal $I:=I_1A+I_2A$. The ideals corresponding to the \zariski\ formulae in the total jets $\Jet{Y_i}{X_i}$ and $\Jet{Y_1\times Y_2}{(X_1\times X_2)}$ are respectively the powers of $I_i$ and of $I$. Since $ I^{2n}\sub I_1^nA+I_2^nA\sub I^n$,   the formula follows readily.
\end{proof}

Although one could give a more general notion of morphism based on morphic \forms, we define them only using (first-order) formulae. Namely, let $\mathcal I$ be a family of (first-order) formulae, and let $\Phi\sub\mathcal L_n$ and $\Psi\sub\mathcal L_m$ be \forms. By   an $\mathcal I$-morphism $f\colon \Phi\to \Psi$, we mean a morphic formula $\theta\in\mathcal I$ such that for each Artinian local $\fld $-algebra $R$,  the definable subset $\inter\theta R\sub R^{n+m}$ restricted to $\inter\Phi R$ is the graph of a map $f_R\colon \inter\Phi R\to \inter\Psi R$. Here we have to replace the morphic conditions~\eqref{i:map}--\eqref{i:im} by their appropriate counterparts (more precisely,   replace $\phi$ and $\psi$ in these sentences respectively by the infinite disjunctions $\Of\Phi$ and $\Of\Psi$; then require that the resulting (non-first order) sentences to hold in any Artinian local $\fld $-algebra).   As before, an \emph{$\mathcal I$-isomorphism} is an $\mathcal I$-morphism which is a bijection on each model, and whose inverse is also an $\mathcal I$-morphism.

\subsection*{The infinitary pp-\gr}
Let $\pp^\infty$
be the lattice of all   \forms\ consisting of formulae in $\pp$, with $\en$ and $\of$ as defined above. On this lattice, we have an isomorphism relation $\iso_{\zar}$, given by \zariski\ isomorphisms modulo $\arttheory \fld $.  
We   define   the   \emph{infinitary pp-\gr} as  
$$ 
\grotartinf \fld 
:=\grotmon{\pp^\infty}{\iso_{\zar}},
$$
where the  multiplication is induced by the same argument as in Lemma~\ref{L:scissid} by the multiplication on \forms. Recall that $\grotartinf \fld $ is the quotient of the free Abelian group $\pol\zet{\pp^\infty}$ modulo the subgroup generated by all $\sym\Phi-\sym\Psi$ for $\zar$-isomorphic formularies $\Phi$ and $\Psi$, and by all $\sym{\Phi\of\Psi}+\sym{\Phi\en\Psi}-\sym\Phi-\sym\Psi$. 
Note that   $\pp^\infty$   is not Boolean, since the negation of a \form\ does not exist. In particular, we do no longer have the analogue of the second property in Lemma~\ref{L:negation}. I do not know whether the analogue of Corollary~\ref{C:grgen} holds (the proof of the corollary relies on the negation property, whence is not admissible here).

Since pp-formulae are just first-order pp-\forms, we get a canonical \homo\
$$
\grotart \fld {}\to \grotartinf \fld .
$$
This \homo, however, is not an embedding, as can be seen from the following relation.
 
\begin{theorem}\label{T:infcomp}
For $Y\sub X$   a closed immersion of affine $\fld $-schemes, we have a relation
$$
\class X=\class{X-Y}+\class{\Jet YX}.
$$
in $\grotartinf \fld $.
\end{theorem}
\begin{proof}
Let $\phi$ be the \zariski\ formula of $X$ and $A$ its coordinate ring.  Let $Y$ be defined by $\phi\en(f_1=\dots f_s=0)$, whence  $I:=\rij fsA$ its ideal of definition in $X$.  Let $\psi_i$ be the pp-formula defining the basic open $D_i:=\op{Spec}A_{f_i}$ of $U:=X-Y$, that is to say, $\psi_i:=\phi\en(\exists \vary)(\vary f_i=1)$. It follows that $\{D_1,\dots,D_n\}$ is an open covering of $U$, and hence by   Corollary~\ref{C:openform}, we have 
$$
\class{U}=\class{\psi_1\of\dots\of\psi_n}
$$
in $\grotart \fld {}$, whence also in $\grotartinf \fld $.  
So remains to show that $\phi\en\niet\psi_1\en\dots\en\niet\psi_n$ is $\arttheory \fld $-equivalent to the \form\ $\Jet YX$. 
One direction is obvious, and to verify the other, we check this in an arbitrary Artinian local $\fld $-algebra $(R,\maxim)$. Let $  a$ be an $n$-tuple in $R$ satisfying $\phi\en\niet\psi_1\en\dots\en\niet\psi_n$. In particular,  $f_i(  a)\in\maxim$, for all $i$. For $l$ at least the length of $R$, we therefore get $g(  a)=0$ for every $g\in I^l$, from which it follows that $  a\in\inter{(\Jet YX)}R$.
\end{proof}
 
Using Proposition~\ref{P:formsch},  we get the following  more suggestive version of Theorem~\ref{T:infcomp}: for any closed immersion $Y\sub X$ of affine $\fld $-schemes, we have
\begin{equation}\label{eq:infcomp}
\class X=\class{X-Y}+\class{\complet X_Y}
\end{equation}
in $\grotartinf \fld $. 

\subsubsection*{Formal Lefschetz class}
We define the \emph{formal Lefschetz class}, denoted $\complet\lef$, as the class of the formal completion of the affine line at the origin $O$,   that is to say, 
$$
\complet\lef:=\class{\Jet O{\affine \fld 1}}=\class{\complet{(\affine \fld 1)}_O}.
$$
 By \eqref{eq:infcomp}, we may now give the correct decomposition formula for the Lefschetz class discussed at the end of Remark~\ref{R:classproj}, namely, in $\grotartinf \fld $ we have
\begin{equation}\label{eq:lefpt}
\lef=\lef^*+\complet\lef.
\end{equation}
By Lemma~\ref{L:prodjets} and the Nullstellensatz, we have
\begin{equation}\label{eq:formjet}
\class{\Jet P{\affine \fld n}}=\class{\complet{(\affine \fld n)}_P}=\complet\lef^n,
\end{equation}
for any closed point $P$ in $\affine \fld n$. We next calculate the class of projective space. 
Using the standard affine covering by the basic opens $\op D(\var_i)\sub \mathbb P_\fld^n$, one easily verifies that the   morphism $(\affine \fld {n+1}-O)\to \mathbb P_\fld^n$, given by sending the affine coordinates $(\var_0,\dots,\var_n)$ to the projective ones $(\var_0:\cdots:\var_n)$, is a locally trivial fibration with fiber $\affine \fld 1-O$, where $O$ is the origin.   By Theorem~\ref{T:infcomp} and \eqref{eq:formjet}, the class of $\affine \fld i-O$ in $\grotartinf \fld $ is equal to $\lef^i-\complet\lef^i$, for every $i$. By Lemma~\ref{L:fib} applied to this locally trivial fibration $\affine \fld {n+1}-O\to \mathbb P_\fld^n$, we get
$$
\lef^{n+1}-\complet\lef^{n+1}=\class{\mathbb P_\fld^n}\cdot(\lef-\complet\lef).
$$
We would like to divide both sides by $\lef-\complet\lef$, but a priori, this is not a zero-divisor in $\grotartinf \fld $. The resulting formula does hold, as we now calculate by a different method:

\begin{proposition}\label{P:projclass}
For each $n$, the class of projective $n$-space in $\grotartinf \fld $ is given by the formula
$$
\class{\mathbb P_\fld^n}=\sum_{m=0}^n\lef^m\cdot\complet\lef^{n-m}.
$$
\end{proposition}
\begin{proof}
Let $(\var_0:\dots:\var_n)$ be the homogeneous coordinates of $\mathbb P_\fld^n$, and let $U_i:=D_+(\var_i)$ be the basic open given as the complement of the $\var_i$-hyperplane. Hence each $U_i$ is isomorphic with $\affine \fld n$ and their union is equal to $\mathbb P_\fld^n$. Therefore,
\begin{equation}\label{eq:projcov}
\class{\mathbb P_\fld^n}=\class {S(U_0,\dots,U_n)}
\end{equation}
in $\grotart \fld {}$ whence in $\grotartinf \fld $. 
So we need to calculate the class of each intersection  occurring in the right-hand side scissor relation. One easily verifies that, for $m\geq 0$, any intersection of $m$ different opens $U_i$ is isomorphic to the open $\affine \fld {n-m}\times(\affine \fld *)^m$, where $\affine \fld *$ is the affine line minus a point. Since $\class{\affine \fld *}=\lef^*=\lef-\complet\lef$ by \eqref{eq:lefpt}, the class of such an intersection is equal to the product $\lef^{n-m}(\lef-\complet\lef)^m$. Since   there are $\binomial{n+1}m$ terms of degree $m$ in the scissor polynomial $S_{n+1}$, the class of ${\mathbb P_k^n}$ is equal to $g(\lef,\complet\lef)$ by \eqref{eq:projcov} and the previous discussion, with
$$
g(t,u):=\sum_{m=0}^n(-1)^{m}\binomial {n+1}mt^{n-m}(t-u)^m.
$$
By the binomial theorem, $t^{n+1}-(t-u)g(t,u)=(t-(t-u))^{n+1}=u^{n+1}$, and hence 
$$
g(t,u)=\frac{t^{n+1}-u^{n+1}}{t-u}=\sum_{m=0}^n t^mu^{n-m},
$$
as we wanted to show.
\end{proof}

Although a priori an infinitary object, the infinitary pp-\gr\   still specializes to the classical \gr:

\begin{proposition}\label{P:infclass}
There exists a canonical \homo\  $\grotartinf \fld \to \grotclass \fld $.
\end{proposition}
\begin{proof}
We use the following observation: if $\mathfrak K $ is   a finite collection of $\fld $-algebras, then any   \form\ is first-order modulo $\mathfrak K $. Indeed, let $\Phi$ be a   \form. As observed above, we may assume that it is closed under finite disjunctions. Let $\mathfrak K=\{R_1,\dots,R_s\}$, and let $\phi_i\in \Phi$ be such that $\inter\Phi R=\inter{\phi_i}R$. Hence $\Phi$ is $\mathfrak K$-equivalent with the (first-order) disjunction $\phi_1\of\dots\of\phi_s$. We can apply this observation to the singleton $\{\fld \}$. Since $\ACF \fld $, the theory of $\fld $, admits  elimination of quantifiers, each class of a \form\ in $\grotartinf \fld $ is equal to a class in $\grotclass \fld $. 
\end{proof}

\begin{remark}\label{R:formlef}
From the proof it follows that image of the formal Lefschetz class $\complet\lef$ under the \homo\ $\grotartinf \fld \to \grotclass \fld $ is equal to $1$.
\end{remark}

\subsubsection*{Arc \forms}
Given a \form\ $\Phi\sub\mathcal L_n$, and an Artinian $\fld $-scheme $Z$, let us define $\arc Z\Phi$ as the \form\  of all $\arc Z\phi$ with $\phi\in \Phi$. As the next result shows, in the \form\ case, we are justified to call $\arc Z\Phi$ the \emph{arc \form\ of $\Phi$ along $Z$}.

\begin{proposition}
If $Z$ is a local Artinian $\fld $-scheme   of length $l$ with  coordinate ring $(R,\maxim)$, and $\Phi\sub\mathcal L_n$ an arbitrary   \form, then $\arc Z\Phi$ is also a \form, and for any Artinian local $\fld $-algebra $S$, there is a one-one correspondence between $\inter\Phi{R\tensor_\fld S}$ and $\inter{\arc Z\Phi}S$ induced by the canonical isomorphism $\pi\colon R\tensor_\fld S\to S^l$. 

Moreover, this induces an arc map on the  {infinitary pp-\gr} $\grotartinf \fld $, and as before, we will write $\integral Z{\Phi}$ for the class of $\arc Z\Phi$.
\end{proposition}
\begin{proof}
We will prove the first two assertions simultaneously. Let $(S,\mathfrak n)$ be an Artinian local $\fld $-algebra, put   $\bar S:=R\tensor_\fld S$, and let $\bar\maxim:=\maxim \bar S+\mathfrak n\bar S$.  Since $\bar S/\bar\maxim\iso R/\maxim\tensor_\fld  S/\mathfrak n\iso \fld \tensor_\fld \fld =\fld $, as $\fld $ is algebraically closed,   $\bar S$ is local with maximal ideal $\bar\maxim$.    In particular, $\bar S$ is   an Artinian local $\fld $-algebra, and hence there is some $\phi_0\in\Phi$ such that $\inter{\phi_0}{\bar S}=\inter\Phi {\bar S}$. Let $  a$ be an $ln$-tuple in $\inter{\arc{}\Phi} S$. Let $\bar a$ be the $n$-tuple over $\bar S$ such that $\pi(  \bar a)=  a$. By construction,  $  a\in\inter{\arc{}\phi} S$ for some $\phi\in \Phi$, and hence $  \bar a\in\inter\phi {\bar S}\sub\inter\Phi {\bar S}=\inter{\phi_0}{\bar S}$ by definition of arc formulae. This in turn implies that  $  a\in\inter{\arc{}{\phi_0}}S$, showing that $\inter{\arc{}\Phi}S=\inter{\phi_0}S$. The second assertion is then immediate   by Proposition~\ref{P:arc}. The last assertion follows by the exact same argument as for Corollary~\ref{C:arcmap}, and its proof is left to the reader.
\end{proof}

\begin{corollary}\label{C:arcjet}
For $(X,P)$   a germ in $\affine \fld n$, and $Z$ an Artinian local $\fld $-scheme of length $l$, we have  
$$
\arc Z{(\Jet PX)}=(\Jet PX\times\affine \fld {(l-1)n})\cap \arc ZX,
$$
where we view $\Jet PX$  as a closed subscheme of $\arc Z{(\Jet PX)}$ via the canonical section defined in Remark~\ref{R:arcsect}.
\end{corollary}
\begin{proof}
Suppose $X$ is a closed subscheme of $\affine \fld n$. An easy calculation shows that
\begin{equation}\label{eq:jetint}
\Jet PX =\Jet P{\affine \fld n}\cap X.
\end{equation}
By the Nullstellensatz, we may assume,  without loss of generality, that $P$ is the origin, and hence the ideals in $\pol \fld \var$ corresponding to the \zariski\ formulae in $\Jet P{\affine \fld n}$ are simply all the powers of the maximal ideal $\rij\var n\pol \fld \var$. Let $(R,\maxim)$ be the Artinian local coordinate ring of $Z$, and let  $(S,\mathfrak n)$ be an arbitrary Artinian local $\fld $-algebra. As already remarked previously,   $\bar S:=R\tensor_\fld S$  is   an Artinian local $\fld $-algebra with maximal ideal $\bar \maxim:=\maxim \bar S+\mathfrak n\bar S$. An $n$-tuple $ \bar a$ over $\bar S$ belongs to $\inter{\Jet P{\affine \fld n}}{\bar S}$ \iff\ all its entries are nilpotent, that is to say, \iff\ $ \bar a\in\bar \maxim\bar S^n$. Since $\bar a\equiv\pi_0({ \bar a})\mod\maxim \bar S$, the latter condition is equivalent with $\pi_0(\bar a)\in\mathfrak nS^n$, which in turn is equivalent with $\pi_0(\bar a)\in\inter{\Jet P{\affine \fld n}}S$, showing that 
$$
\arc Z{(\Jet P{\affine \fld n})}=\Jet P{\affine \fld n}\times \affine \fld {(l-1)n}
$$
 under the identification from Remark~\ref{R:arcsect}. Taken together with \eqref{eq:jetint} and the fact that the arc map preserves intersections, we get the desired equality.
\end{proof}

\begin{remark}
It follows from the proof  and \eqref{eq:infcomp} that we in fact have an equality
$$
\arc Z{(\complet X_P)}=(\complet {(\affine \fld n)}_P\times\affine \fld {(l-1)n})\cap \arc ZX,
$$
if $(X,P)$ is a germ in $\affine \fld n$.
\end{remark}

We have the following analogue of Proposition~\ref{P:intlef} for the formal Lefschetz class:

\begin{corollary}
For any element $q\in\grotschzero \fld $, we have $\integral q{\complet\lef}=\complet\lef\cdot\lef^{\ell(q)-1}$.
\end{corollary}
\begin{proof}
By additivity, it suffices to show this for $q$ equal to an Artinian local scheme $Z$ of length $l$. By Corollary~\ref{C:arcjet}, we have
$$
\arc Z{(\Jet O{\affine \fld 1})}=(\Jet O{\affine \fld 1}\times\affine \fld {l-1})\cap \arc Z{\affine \fld 1},
$$
where $O$ is the origin.  Since $\arc Z{\affine \fld 1}=\affine \fld l$, taking classes therefore yields the asserted formula.
\end{proof}

If instead we work in the \gr, we may generalize the previous result to higher dimensional fibers:

\begin{corollary}\label{C:jetarc}
Let $Y\sub X$ be a closed immersion of  $\fld $-schemes,   $Z$   an Artinian $\fld $-scheme, and   $\rho\colon\arc ZX\to X$   the canonical split projection. Then we have an equality
$$
\class{\arc Z{(\jet Y{}X})}=\class{\jet {\inverse\rho Y}{}{(\arc ZX)}}
$$
in $\grotartinf \fld $.
\end{corollary} 
\begin{proof}
By Theorem~\ref{T:infcomp}, we have an equality
$$
\class X=\class{X-Y}+\class{\jet Y{}X}.
$$
in $\grotartinf \fld $. Since $\arc Z\cdot$ is an endomorphism, we get
\begin{equation}\label{eq:arccomp}
\class{\arc  ZX}=\class{\arc Z{(X-Y)}}+\class{\arc Z{(\jet Y{}X})}.
\end{equation}
Since $X-Y\sub X$ is an open immersion, we have 
$$
\arc Z{(X-Y)}=\inverse\rho {X-Y}=\arc ZX-\inverse\rho Y
$$
 by \eqref{eq:pbopen} in Theorem~\ref{T:arc}. On the other hand, by another application of Theorem~\ref{T:infcomp}, we get
$$
\class{\arc ZX}=\class{\arc ZX-\inverse\rho Y}+\class{\jet{\inverse\rho Y}{}{\arc ZX}},
$$
from which the assertion now follows immediately in view of \eqref{eq:arccomp}.
\end{proof}

\section{Geometric Igusa-zeta series over linear arcs}\label{s:ratIgu}

For various schemes $X$, we will  calculate $\igu X{\mathbb A}O$, that is to say, we want to calculate $\sum_n\class{\arc{Z_n}X}t^n$, where, for the remainder of this section    $Z_n:=\op{Spec}(\pol \fld \xi/\xi^n\pol \fld \xi)$. To simplify notation, we simply write $\arc nX$ for    the \emph{$n$-th linear arc scheme} $\arc{Z_n}X$. We let  $\rho_n\colon\arc{n}X\to X$ be the canonical split projection with section $X\into \arc{n}X$, and we view closed subschemes of $X$ as closed subschemes of $\arc{n}X$ via the latter embedding. 
 
\subsection*{The formal \gr}
 
In $\grotartinf \fld $, we define the \emph{formal ideal} $\mathfrak N$ as the ideal generated by the relations $ \class {\jet Y{}X}-\class Y$, for all closed immersions $Y\sub X$ of affine schemes. It follows that if $Y,Y'\sub X$ are two closed subschemes of $X$ with the same underlying set, then $\class Y\equiv\class{Y'}\mod\mathfrak N$, since they have the same total jets.   Recall from Proposition~\ref{P:infclass} that we have a canonical \homo\  $\grotartinf \fld \to \grotclass \fld $. We prove below that $\mathfrak N$ belongs to its kernel, and so we introduce the \emph{formal \gr} $\grotform \fld $ as the quotient $\grotartinf \fld /\mathfrak N$.

\begin{proposition}\label{P:formid}
We have a sequence of natural \homo{s} of \gr{s} $\grotartinf \fld \to\grotform \fld \to \grotclass \fld $.
\end{proposition}
\begin{proof}
By definition of \form, there exists a \zariski\ formula $\phi$ in $\jet Y{}X$ such that its $\fld $-rational points are given by $\inter\phi \fld $, for a given closed immersion $Y\sub X$. In particular, $\inter\phi \fld $ is equal to $Y(\fld )$, showing that $\class\phi=\class Y$ in $\grotclass \fld $. By the argument in Proposition~\ref{P:infclass}, the image of $\class{\jet Y{}X}$ in $\grotclass \fld $ is equal to $\class\phi$. This shows that $\mathfrak N$ lies in the kernel of  $\grotartinf \fld \to \grotclass \fld $.
\end{proof}

In particular, to prove the rationality of the geometric Igusa-zeta series over $\grotclass \fld_\lef$, it will suffice to show that its image as a power series over $\grotform \fld $ is rational over $\grotform \fld_\lef$. We will simplify our notation and write $\igugeom X(t)$ for $\igu X{\mathbb A}O$, viewed as a power series over $\grotform \fld $.

\begin{lemma}\label{L:formpart}
If $\{Y_i\}_i$ is a constructible partition of a scheme $X$, then $\class X=\sum_i\class{Y_i}$ in $\grotform \fld $.
\end{lemma}
\begin{proof}
Note that this partition is finite, since $X$ is Noetherian. Moreover, at least one part must be open, say $Y_1$, and let $X_1:=X-Y_1$ be its complement. By Theorem~\ref{T:infcomp}, we have an equality $\class X=\class{Y_1}+\class{\jet {X_1}{}X}$ in $\grotartinf \fld $. By definition of formal ideal, $\class{\jet{X_1}{}X}\equiv\class {X_1}$ modulo $\mathfrak N$. Putting these two together, we get an identity $\class X=\class{Y_1}+\class{X_1}$ in $\grotform \fld $. Moreover, $\{Y_2,Y_3,\dots\}$ is a constructible partition of $X_1$, and so we are done by Noetherian induction. 
\end{proof}

\begin{theorem}\label{T:redfiber}
Let $X$ be a $d$-dimensional scheme, $Y$   a closed subscheme containing the singular locus of $X$, and   $Z$   an Artinian scheme of length $l$. If $\rho\colon\arc ZX\to X$ denotes the canonical projection, then   we have an equality
$$
\class{\arc ZX}=\class{X-Y}\cdot\lef^{d(l-1)}+\class {\inverse\rho Y}
$$
in $\grotform \fld $.
\end{theorem}
\begin{proof}
This follows immediately from  Corollaries~\ref{C:fib} and \ref{C:jetarc}, but here is the argument in more detail: let us put $W:=\arc ZX$ and $V:=\inverse\rho Y$. By Lemma~\ref{L:formpart}, we have an equality $\class W=\class{W-V}+\class{V}$ in $\grotform \fld $. Moreover, by Theorem~\ref{T:arc}, we have an isomorphism $W-V\iso \arc Z{(X-Y)}$. Since $X-Y$ is smooth by the choice of $Y$, $\class {\arc Z{(X-Y)}}= \class{X-Y}\cdot\lef^{d(l-1)}$ by Corollary~\ref{C:fib}, and the assertion follows. 
\end{proof}

In order to simplify our notation, we will henceforth write $X_s$ for the basic subset $\op D(s)$ in a scheme $X$, where $s$ is a global section on $X$. Likewise, we write $X\rij sn$ for the intersection $X\cap \op V\rij sn$, for given global sections $s_i$. In this notation, we have, by Lemma~\ref{L:formpart}, the following useful    equality
\begin{equation}\label{eq:basic}
\class X=\class{S_n(X_{s_1},\dots,X_{s_n})}+\class{X\rij sn}
\end{equation}
in $\grotform \fld $. Note that products in the  scissor polynomial are actually given by intersection, as in \S\ref{s:lat}. For instance, \eqref{eq:basic}  becomes for $n=2$, the identity
$$
\class X= \class{X_s}+\class{X_t}-\class{X_{st}}+\class{X(s,t)}.
$$

To derive the next identity, we introduce some further notation. Fix a scheme $X$ and an $n$-tuple of global sections $\rij sn$. Given a \emph{binary vector} $\delta$ of length $n$, that is to say, a $n$-tuple  in $\{0,1\}^n$, we will write $X_\delta$ for $X_{s^\delta}$ where $s^\delta$ is the product of all $s_i^{\delta_i}$, and we will write $\bar X_\delta$ for $X_\delta(s_\delta)$, where $s_\delta$ is the tuple of all $s_i$ for which $\delta_i=0$. In other words, $\bar X_\delta$ is defined by the formula expressing that each $s_i$ is either a unit or zero, depending on whether $\delta_i$ is one or zero. One easily verifies that
$
X=\bigsqcup_{\delta} \bar X_\delta,
$
where $\delta$ runs over all binary vectors of length $n$. Applying Lemma~\ref{L:formpart} to this constructible partition, we get 
\begin{equation}\label{eq:basicsum}
\class X=\sum_\delta\class{\bar X_\delta}
\end{equation}
in $\grotform \fld $, where the sum runs over all binary vectors. Let us again illustrate this  for $n=2$, yielding the identity
$$
\class X= \class{X_{st}}+\class{X_s(t)}+\class{X_t(s)}+\class{X(s,t)}.
$$
It is important to note that this equation  is false in $\grotartinf \fld $, whence, in particular,  we may not apply $\arc Z{}$ to it.

Before we turn to a proof of the rationality of the geometric Igusa-zeta series, we should mention that this method is different from  working with classical arcs in the classical \gr. Although we eventually take classes in $\grotform \fld $, thus \emph{collapsing nilpotents}, we will do this only after taking arcs. Put differently, although arcs will be reduced, the base schemes will not be, and to see that this makes a difference (even in regards to   dimension!), we list, for small lengths, some  defining equations of   arcs and their reductions for three different closed subschemes with the same underlying set, the union of two lines in the plane:

\newcommand\breathe{{\vphantom{\left(\complet\Sigma_\Sigma\right)}}} 
\begin{table}[h]
\caption{Equations and dimension $d$ of arcs   and their classes in $\grotform \fld $}
\label{tab:1}     
\begin{tabular}{|c|c|c|r|c|r|c|r|}
\hline
&$l \breathe $ &$xy=0$&$d$ & \phantom{xxx}$x^2y=0$\phantom{xxx}&$d$ & \phantom{xxxxx}$x^2y^3=0\phantom{xxxxxxxx}$&$d$   \\
\hline\hline
\multirow{4}{*}{$\arc l{}$} &$1 \breathe $ 	&\multicolumn{2}{|c|}{$\tilde x_0\tilde y_0,$} & \multicolumn{2}{|c|}{$\tilde x_0^2\tilde y_0,$} & \multicolumn{2}{|c|}{$\tilde x_0^2\tilde y_0^3,$}\\\cline{2-8}
		&$2 \breathe $	&\multicolumn{2}{|c|}{$\tilde x_0\tilde y_1+\tilde x_1\tilde y_0,$} & \multicolumn{2}{|c|}{$2\tilde x_0\tilde x_1\tilde y_0+\tilde x_0^2\tilde y_1,$} & \multicolumn{2}{|c|}{$2\tilde x_0\tilde x_1\tilde y_0^3+3\tilde x_0^2\tilde y_0^2\tilde y_1,$}\\\cline{2-8}
		&$3 \breathe $ &\multicolumn{2}{|c|}{ $\tilde x_0\tilde y_2+\tilde x_1\tilde y_1+$} & \multicolumn{2}{|c|}{$\tilde x_0^2\tilde y_2+2\tilde x_0\tilde x_1\tilde y_1+$} &
		\multicolumn{2}{|c|}{$3\tilde x_0^2(\tilde y_0\tilde y_1^2+\tilde y_0^2\tilde y_2)+$}\\
		&&\multicolumn{2}{|c|}{$\breathe \tilde x_2\tilde y_0$\phantom{xxxxxxxx}}& \multicolumn{2}{|c|}{$(2\tilde x_0\tilde x_2+\tilde x_1^2)\tilde y_0$\phantom{xx}}& \multicolumn{2}{|c|}{$6\tilde x_0\tilde x_1\tilde y_0^2\tilde y_1+(\tilde x_1^2+2\tilde x_0\tilde x_2)\tilde y_0^3$}\\
			\hline
\multirow{3}{*}{$\class{\arc l{}}$} &$1 \breathe $ & $\tilde x_0\tilde y_0,$&1& $\tilde x_0\tilde y_0,$&1& $\tilde x_0\tilde y_0,$&1\\\cline{2-8}
&$2 \breathe $ & $\tilde x_0\tilde y_1,\tilde x_1\tilde y_0,$ &2& $\tilde x_0\tilde y_1,$ &3& [no new equation]&3\\\cline{2-8}
&$3 \breathe $ & $\tilde x_0\tilde y_2,\tilde x_1\tilde y_1,\tilde x_2\tilde y_0$ &3& $\tilde x_0\tilde y_2,\tilde x_1\tilde y_0$ &4& $\tilde x_1\tilde y_0$&5 \\
\hline
\end{tabular}
\end{table}

\subsection*{Tagged and formal equations}
In the sequel, we will   only invert variables, and to this end, we simplify our notation even further.
To any natural number $a$, we associate its \emph{tagged} version $a^*$, and we call $v(a^*):=a$ the \emph{underlying value} (or \emph{untagged} version) of $a^*$. We can add tagged and/or untagged numbers by the rule that the underlying value of the sum is the sum of the underlying values of the terms, where the sum is tagged \iff\ at least one term is tagged (e.g,. $2+3^*=5^*$). Fix $m\geq 1$, and let $\Gamma_m$ be the collection of $m$-tuples with entries natural numbers or their tagged versions. We extend $v$ component-wise to get a map $v\colon \Gamma_m\to \nat^m$, sending a tuple $\theta\in\Gamma_m$ to its \emph{underlying value} $v\theta$. We define a partial order on $\Gamma_m$ by $\alpha\preceq\beta$ \iff\ $\alpha_j$ is untagged and $\alpha_j\leq\beta_j$, or $\alpha_j$ is tagged and $\alpha_j=\beta_j$, for all $j=\range 1m$. 

We will introduce two equational conventions in this section that are useful for discussing arc equations.
To each variable $x$, we associate its tagged version $x_*$, which we will treat as an invertible variable. Given a tagged number $a^*$, we let $x^{a^*}$  be the same as $x_*^{a^*}$, and simply write  $x_*^a$. Hence,  we may associate to a polynomial $f\in \pol \fld x$, the polynomial $f(x_*)$, which is just $f(x)$ but viewed in the Laurent polynomial ring $\pol \fld {x,\frac 1x}$. Therefore,  we interpret the equation $f(x_*)=0$ as the conjunction $f(x)=0$ and $x$ is invertible, that is to say, the pp-formula $(\exists x')f(x)=0\en xx'=1$. We may extend this practice to several variables, tagging some of them and leaving the others unchanged. For instance, the \emph{tagged} equation $x^2_*+x_*y^3+z^3_*=0$ should be considered as an element of the \emph{mixed Laurent polynomial ring} $\pol \fld {x,y,z,\frac 1x,\frac 1z}$, and is equivalent with the conditions $x^2+xy^3+z^3=0$ together with $x$ and $z$ are invertible. 

Our second convention is the use of a formal variable $\xi$, fixed once and for all. Given a power series $f(x, \xi)\in\pow{\pol \fld x} \xi $ (or, at times, a Laurent series) with coefficients in a polynomial ring $\pol \fld x$ (or a mixed Laurent polynomial ring), we interpret the (formal) equation $f=0$ as the condition on the $x$-variables that $f$ be identical zero as a power series (Laurent series) in $\xi $. In other words, if $f(x, \xi)=f_0(x)+f_1(x) \xi +f_2(x) \xi ^2+\dots$, then $f=0$ stands for the (infinite) conjunction $f_0=f_1=f_2=\dots=0$ (as $f=0$ and $\xi^if=0$ yield   equivalent systems of equations, we may reduce the Laurent series case to the power series case). Similarly, for each $n$, the equivalence $f(x, \xi)\equiv0\mod \xi ^n$ stands for  the conjunction $f_0=f_1=\dots=f_{n-1}=0$. An example of a combination of both conventions is
$$
0=(x+y_* \xi)^2+(z_*+w \xi)^3
$$
which is equivalent to the pp-formula 
\begin{align*}
x^2+z^3&=0 \\
2xy+3z^2w&=0 \\
y^2+3zw^2&=0 \\
w^3&=0 \\
(\exists y',z')\quad yy'&=1 \en zz'=1
\end{align*}

To any $m$-tuple of variables $\var=\rij \var m$, we associate the corresponding (countably many) \emph{arc variables}  $\tilde \var=(\tilde\var_0,\tilde\var_1,\dots) $, where each $\tilde \var_i$ is an  $m$-tuple $(\tilde \var_{i,1},\dots,\tilde \var_{i,m})$. For each $i$, we let
$$
\dot x_i=\tilde x_{0,i}+\tilde x_{1,i} \xi+ \tilde x_{2,i}\xi ^2+\dots
$$
be the \emph{generic arc series}   in $\xi $, and we write $\dot x$ for the tuple $(\dot x_1,\dots,\dot x_m)$. Given $\theta\in\Gamma_m$, we define $\org x\theta$ to be the $m$-tuple of twisted power series with $i$-th entry equal to
$$
\dot x_i(\theta):=\tilde x_{\theta_i,i}+\tilde x_{\theta_i+1,i} \xi + \tilde x_{\theta_i+2,i}\xi ^2+\dots
$$
if $\theta_i$ is untagged, and
$$
(\dot x_i)_*(\theta):=(\tilde x_{v\theta_i,i})_*+   \tilde x_{v\theta_i+1,i}\xi +   \tilde x_{v\theta_i+2,i}\xi ^2+\dots
$$
if $\theta_i$ is tagged (note that according to this convention, only the constant term is actually tagged, which accords with the fact that a power series is invertible \iff\ its constant term is). For each $\theta\in\Gamma_m$, define a change of variables $\tau_\theta$   sending, for each $j=\range 1m$, the variable $\tilde x_{i,j}$ to $\tilde x_{i-v\theta_j,j}$ and $(\tilde x_{i,j})_*$ to $(\tilde x_{i-v\theta_j,j})_*$  for $i\geq v\theta_j$,  and leaving the remaining variables and their tagged versions unchanged. In particular,  $\tau_\theta$ only depends on the underlying value of $\theta$, and $\tau_\theta(\org {x_i}\theta)$ is equal to   $\dot x_i$ if $\theta_i$ is untagged and to $(\dot x_i)_*$ if $\theta_i$ is tagged.

\subsection*{Directed arcs}
With these conventions, we can now write down the equations of an arc scheme more succinctly. If $X\sub\affine \fld  m$ is the closed subscheme defined by the \zariski\ formula $g_1=\dots=g_s=0$, then $\arc nX$ is defined by the conditions
$$
g_1(\dot x)\equiv g_2(\dot x)\equiv\dots\equiv g_s(\dot x)\equiv 0\mod \xi^n
$$
and $\org x{\tuple n}=0$ (recall that $\tuple n$ is the tuple all of whose entries are equal to $n$). Note that the latter condition simply means that $\tilde x_{i,j}=0$ for all $i\geq n$ and all $j=\range 1m$.

We extend the notion of arc scheme, by considering certain (initial) linear subspaces of arc schemes. Given $\theta\in\Gamma_m$, we define the \emph{$n$-th directed arc scheme along $\theta$}, denoted $\parc n\theta X$, as the locally closed subscheme of $\arc nX$ defined by the conditions $\tilde x_{i,j}=0$ for $i<\theta_j$, and $\tilde x_{v\theta_j,j}$ is invertible if $\theta_j$ is tagged,  for $j=\range 1m$. We may also refer to $\parc n\theta X$  as the subscheme of all \emph{arcs along}, or \emph{with initial direction} $\theta$. Writing out these conditions in more detail, the defining equations of $\parc n\theta X$ become
$$
g_1(\org x\theta)=\dots=g_s(\org x\theta)\equiv 0\mod \xi^n
$$
and $\tilde x_{i,j}=0$ for $i<\theta_j$ or $i\geq n$, and for $j=\range 1m$. By the change of variables  $\tau_\theta$, we can rewrite these equations as
\begin{equation}\label{eq:parceq}
\tau_\theta(g_1(\org x\theta))=\dots=\tau_\theta(g_s(\org x\theta))\equiv 0\mod \xi^n \qquad\en\qquad \org x{\tuple n-\theta}=0.
\end{equation}
This form will be easier to work with, as we can now compare arcs along different directions; we call the first set of equations in \eqref{eq:parceq} the \emph{arc equations}, and the second set the \emph{initial conditions}. We will use the arc equations as follows:  given $\theta\in\Gamma_m$, let $z$ be a new $m$-tuple of variables with corresponding arc variables $\tilde z$, called \emph{the $\theta$-tagging of $x$}, where $z_j$  is equal to  $(x_j)_*$ or $x_j$   depending on whether  $\theta_j$ is tagged or not, and similarly, $\tilde z_{i,j}=\tilde x_{i,j}$ unless $i=0$ and $\theta_j$ is tagged, in which case $\tilde z_{0,j}=(\tilde x_{0,j})_*$). If $f  =\sum_\mu c_{\mu}x^\mu $, then
\begin{equation}\label{eq:twist}
\tau_\theta(f(\org x\theta))=\sum_{\mu}c_{\mu}\xi^{\mu\theta}\dot z^\mu.
\end{equation}

\begin{example}[Fibers]\label{E:fiber}
Recall that $\rho_n\colon\arc nX\to X$ is the canonical projection of the arc scheme onto the base scheme. Let us calculate the fiber $\inverse{\rho_n}O$ of the origin. If $f_1=\dots=f_s=0$ is the \zariski\ formula defining $X$, then $\arc nX$ is given by the equations $f_i(\dot x)\equiv0\mod\xi^n$, and $\inverse{\rho_n}O$ is the closed subscheme given by $\tilde x_0=0$, that is to say,   
\begin{equation}\label{eq:fiber}
\inverse{\rho_n}O=\parc n{\tuple 1}X. 
\end{equation}
\end{example}

\begin{definition}[Twisted geometric Igusa-zeta series]\label{D:twistigu}
Given $\theta\in \Gamma_m$, define the \emph{$\theta$-twisted geometric Igusa-zeta series} of $X$ to be
$$
\igutwist X\theta(t):=\sum_{n=0}^\infty \class{\parc n\theta X}t^n.
$$
Hence, $\igugeom X$ is just the case in which the twist is zero.
\end{definition}

At times, it is convenient, especially in inductive arguments, to prove that all twisted geometric Igusa-zeta series are rational.

\subsection*{Twisted initial forms}
For the remainder of the section, we restrict to the case of a hypersurface    $X$   defined by a single equation $f:=\sum_\nu c_\nu x^\nu$. If $f$ is not homogeneous, we can no longer expect such a simple relation between the arc scheme and the fiber above the singular locus. As we will shortly see, the following hypersurfaces derived from $X$ will play an important role: for every   $\theta\in\Gamma_m$, let $\tin X\theta$ be defined as follows. View $\pol \fld x$ as a graded ring giving the variable $x_i$ weight $v\theta_i$. Let $\ord\theta f$, or $\ord\theta X$, be the  order of $f$ in this grading, that is to say, the minimum of all $v\theta\cdot\nu$ with $c_\nu\neq0$, and let $\tin X\theta$ be the hypersurface with defining equation 
$$
\tin f\theta:=\sum_{v\theta\cdot\nu=\ord\theta X} c_\nu x^\nu.
$$
 In particular, $X=\tin X{\tuple 0}$. We call $\tin X\theta$, or rather, $\tin f\theta$, the \emph{$\theta$-twisted initial form} of $X$.

Here is an example to view the previous conventions and definitions at work:

 \begin{example}\label{E:235}
 Let $f=x^9+x^2y^4+z^4$ and $\theta=(2,3^*,5)$. Hence   $\parc n{(2,3^*,5)} X$ is the locally closed subscheme of $\arc{(2,3,5)}X$ given by the conditions $\tilde x_0=\tilde x_1=\tilde y_0=\tilde y_1=\tilde y_2=\tilde z_0=\tilde z_1=\tilde z_2=\tilde z_3=\tilde z_4=0$ and $\tilde y_3$ is invertible.  
 Using \eqref{eq:parceq} and \eqref{eq:twist}, its equations are
  $$
 \tau_{(2,3^*,5)}(f(\org x{2,3^*,5},\org y{2,3^*,5},\org z{2,3^*,5}))= \xi^{18}\dot x^9+\xi^{16}\dot x^2\dot y^4_*+\xi^{20}\dot z^4\equiv 0\mod \xi^n
 $$
 and $\org x{n-2}=\org y{n-3}=\org z{n-5}=0$. Hence, $\ord {(2,3^*,5)}X=16$ and the twisted initial form   $\tin X{(2,3^*,5)}$ is given by $\tin f{(2,3^*,5)}=x^2y^4_*$, that is to say, by the two conditions $x^2y^4=0$ and $y$ is a unit. 
\end{example}

\subsection*{Regular base}
We will deduce rationality by   splitting off regular pieces of various twisted initial forms, until we arrive at a recursive relation involving the arc scheme of the original hypersurface. To this end, we introduce the following definition: we say that $\theta\in\Gamma_m$ is \emph{$X$-regular} if $\tin \theta X$ is smooth. 
As with arcs, directed arcs above a regular base have a locally trivial fibration,  a fact which will allow us to determine their contribution to the Igusa-zeta series:
 
 \begin{proposition}\label{P:parcreg}
 Let $X\sub\mathbb A^m$ be a hypersurface. For  each  $X$-regular tuple $\theta\in \Gamma_m$,  we have locally (i.e., on an open affine covering), an isomorphism
$$
\parc n\theta X\iso \arc{n-\ord X\theta}{(\tin X\theta)}\times\mathbb A^{m\cdot \ord X\theta-\norm\theta},
$$
for each $n> \ord X\theta$.
\end{proposition}
 \begin{proof}
Let $f=\sum_\nu c_\nu x^\nu$ be the defining equation of $X$. Let us put $a:=\ord X\theta$; recall that it is the minimum of all $v\theta\cdot\nu$ with $c_\nu\neq 0$. Instead of $x$, we will use the $m$-tuple of variables $z$ whose $i$-th variable is tagged precisely if $\theta_i$ is.  For each $k$, let $f^\theta_k :=\sum_{v\theta\cdot\nu=k}c_\nu z^\nu$, so that $\tin f\theta=f_{a}^\theta$ is the defining equation of $\tin X\theta$. By \eqref{eq:parceq} and \eqref{eq:twist}, the arc equation of $\parc n\theta X$ is  
$$
\tau_\theta(f(\org z\theta))=\sum_{k=  a}^{n-1} \xi^kf^\theta_k(\dot z)\equiv 0\mod \xi^n
$$
whereas the initial condition is $\dot z(\tuple n-\theta)=0$. Factoring out $\xi^{a}$, yields the arc equation
\begin{equation}\label{eq:parcntheta}
 \sum_{k=0}^{n-a-1} \xi^kf^\theta_{a+k}(\dot z)\equiv 0\mod \xi^{n-a}.
\end{equation}
On the other hand, the arc equation of $\arc{n-a}{\tin X\theta}$ is
\begin{equation}\label{eq:arcntheta}
\tin f\theta(\dot z)\equiv 0\mod \xi^{n-a}.
\end{equation}
 Note that expansions \eqref{eq:parcntheta} and \eqref{eq:arcntheta} have the same constant term $\tin f\theta=f^\theta_{a}$. Expand 
 $$
 f^\theta_k(\dot z)=\sum_l \xi^lf^\theta_{k,l}(\tilde z), 
 $$
 with each $f_{k,l}^\theta$ only depending on $\tilde z_0,\dots,\tilde z_{l-1}$. If $k= a$, we will  write $\tilde f^\theta_l(\tilde z)$ for $f^\theta_{a,l}(\tilde z)$, so that $\tin f\theta(\dot z)=\sum_l\xi^l \tilde f^\theta_l(\tilde z)$. Substituting in \eqref{eq:parcntheta}, we get an expansion
  $$
  \sum_{k,l=0}^{n-a-1} \xi^{k+l}f^\theta_{a+k,l}(\tilde z) 
    $$
showing that the defining equations of $\parc n\theta X$ are $g_0=\dots=g_{n-a-1}=0$ together with $\tilde z_{i,j}=0$ if $i\geq n-\theta_j$, where
\begin{equation}\label{eq:parcexpan}
g_l(\tilde z):=\sum_{k=0}^l f^\theta_{a+k,l-k} = \tilde f^\theta_l+ \sum_{k =1}^l f^\theta_{a+k,l-k}.
\end{equation}
 Since $\tin X\theta$ is smooth, the proof of Corollary~\ref{C:fib} shows that locally $\tilde f^\theta_l$, for $l>0$,  is linear in the $\tilde z_l$-variables, and smoothness allows us to solve for one of the $\tilde z_l$-variables  in terms of the others. Restricting   to a basic open, we may assume that we can do this globally (we leave the details to the reader; but see also  the proof of  Corollary~\ref{C:fib}). However, the same is then true for $g_k$, since the difference $g_l-\tilde f^\theta_l$ only depends on variables $\tilde z_0,\dots, \tilde z_{l-1}$ by \eqref{eq:parcexpan}. This shows that the closed subscheme defined by $g_0,\dots,g_{n-a-1}$ when viewed in the variables $\tilde z_0,\dots, \tilde z_{n-a-1}$ is the same as $\arc n{\tin X\theta}$. As for the remaining   $m a$ variables $\tilde z_{n-a},\dots, \tilde z_{n-1}$, among these, $\norm\theta$ many of them are put equal to zero, whereas the rest remains free, proving the assertion.
  \end{proof}
 
 \begin{corollary}\label{C:parcreg}
If $\theta$ is $X$-regular, then
$$
\class{\parc n\theta X} = \class{\tin X\theta}\cdot \lef^{(m-1)(n-1)+\ord X\theta-\norm\theta}
$$
in $\grotart \fld{} $.
\end{corollary}
\begin{proof}
This follows,  by the same argument as for Lemma~\ref{L:fib}, immediately from Corollary~\ref{C:fib} and Proposition~\ref{P:parcreg}, noting that $\tin X\theta$  has dimension $m-1$.
\end{proof}

 \subsection*{Recursion}
 Given $\alpha,\beta\in\Gamma_m$, we will write    $\recur X\alpha\beta$, if $\alpha\preceq\beta $ and there exists some $s>0$ such that
 $$
 \tau_\beta(f(\org z\beta))=\xi^s\tau_\alpha(f(\org z\alpha)).
 $$
An easy calculation shows that necessarily $s=\ord X\beta-\ord X\alpha$.   Note that $f$ is homogeneous in the classical sense \iff\ $\recur X{\tuple 0}{\tuple 1}$.

 \begin{lemma}\label{L:arcequiv}
If $\recur X\alpha\beta$, then   
$$
\class{\parc n\beta X} =\class{\parc {n-s}\alpha X}\cdot\lef^{sm-\norm\beta+\norm\alpha}
$$
in $\grotform \fld_\lef$, for all $n>s$, with   $s=\ord X\beta-\ord X\alpha$.
\end{lemma}
\begin{proof}
By \eqref{eq:parceq}, the defining equations of  $\parc n\beta X$ are  
$$
\tau_\beta(f(\org z\beta))\equiv 0\mod \xi^n \quad\text{and}\quad \org x{\tuple n-\beta}=0.
$$
By assumption, the power series in the arc equation equals $\xi^s\tau_\alpha(f(\org z\alpha))$, and so yields the arc equation
\begin{equation}\label{eq:betans}
\tau_\alpha(f(\org z\alpha))\equiv0\mod \xi^{n-s}.
\end{equation}
However,  \eqref{eq:betans} is also   the arc equation of $\parc {n-s}\alpha X$. As the initial condition for $\parc {n-s}\alpha X$ is given by  $\org x{\tuple n-\tuple s-\alpha}$, the difference between the two directed arc schemes lies in the number of   free variables not covered by the respective initial conditions, a number which is easily seen to be $\norm{\tuple s-\beta+\alpha}=sm-\norm\beta+\norm\alpha$, whence the assertion.
\end{proof}

\subsection*{Rationalizing trees}
 We are interested in subtrees of $\Gamma_m$, and  will use the following terminology: by a \emph{tree} we mean a finite, connected partially ordered subset of (\emph{nodes} from) $\Gamma_m$   such that any initial segment is totally ordered. The unique minimum is called the \emph{root} of the tree, and any maximal element is called a \emph{leaf}. By a \emph{branch}, we will mean a chain $[\alpha,\beta]$ from a node $\alpha$ to a leaf $\beta$. 
 
Let $\delta\leq \eta$ be binary vectors, that is to say, tuples with entries $0$ or $1$. We define a transformation $e^\eta_\delta$ on $\Gamma_m $ as follows. For each $i$, let $e_i$ simply be addition with the basis vector $e_i$ on $\Gamma_m $ (note that per our addition convention, each entry stays in whichever state, tagged or untagged,   it was). On the other hand, we let $e_i^*$ be the transformation which tags the $i$-th entry but  leaves the remaining entries unchanged. Given a binary vector $\varepsilon$, we let $e_\varepsilon$ (respectively, $e^*_{\varepsilon}$) be the composition of all $e_i$ (respectively, all $e_i^*$) for which $\varepsilon_i=1$. Note that all these transformations commute with each other. Finally, we let     $e^\eta_\delta$  be the composition of $e_\delta$ and  $e^*_{\eta-\delta}$. For instance, 
\begin{multline*}
\quad e^{(1,1,0,1,0)}_{(0,0,0,1,0)}(2,3^*,1,4,1) =e_{(1,1,0,0,0)}^*e_{(0,0,0,1,0)}(2,3^*,1,4,1)= \\
e_1^*e_2^*e_4(2,3^*,1,4,1)=(2^*,3^*,1,5,1).\qquad \qquad
\end{multline*}
Note that $e^\eta_\delta(\theta)$ has underlying value equal to $v\theta+\delta$. More precisely, taking in account our addition convention, we have 
$$
e^\eta_\delta(\theta)=e^*_{\eta-\delta}(\theta+\delta).
$$
 Note that $e^\eta_\delta$ call fail to be an increasing function (if in the above example we replace $(0,0,0,1,0)$ by $(0,1,0,0,1,0)$ the resulting tuple is $(2^*,4^*,1,5,1)$ which is not comparable with $(2,3^*,1,4,1)$ because of the second entry). A necessary condition is
  \begin{equation}\label{eq:preceq}
(\forall i)[\text{if }\theta_i\text{ tagged then } \eta_i=0] \dan \theta\preceq e^\eta_\delta(\theta). 
\end{equation}
  We will use these transformation mainly through the following result:

\begin{lemma}\label{L:tmprec}
Let $X\sub\affine \fld m$ be a closed subscheme. 
For every   $\theta\in \Gamma_m $, and every binary vector $\eta$, we have an identity
$$
\class{\parc n\theta X}=\sum_{\delta\leq\eta}\class{\parc n{e^\eta_\delta(\theta)}X}
$$
in $\grotform \fld $, for all $n$.
\end{lemma}
\begin{proof}
We apply \eqref{eq:basicsum}   to $Y:=\parc n\theta X$ with respect to the variables $\tilde x_{\theta_i,i}$ such that $\eta_i=1$, yielding
$$
\class{Y}=\sum_{\delta\leq\eta}\class{\bar Y_\delta}.
$$
However, $\bar Y_\delta$ is obtained by inverting, that is to say, tagging $\tilde x_{\theta_i,i}$ if $\delta_i=1$, and equating it  to zero, if $\delta_i=0$. Since the defining arc equations for $\parc n\theta X$ are $f_i(\org x\theta)\equiv 0\mod \xi^n$, for $i=\range 1s$, where $f_1=\dots=f_s=0$ is the defining \zariski\ formula of $X$, the arc equations of $\bar Y_\delta$ are  $f_i(\org x{\theta+\eta-\delta})\equiv 0\mod \xi^n$, for $i=\range 1s$, together with inverting all $\tilde x_{\theta_i,i}$ for which $\delta_i=1$. As these  are the defining arc equations for  $\parc n{e^\eta_{\eta-\delta}(\theta)}X
$, we proved the assertion (note that summing over all $\delta$ is the same as summing over all $\eta-\delta$).
\end{proof}

 We   define by induction   on the  height of a tree in $\Gamma_m$ for it to be a \emph{resolution tree} as follows: any singleton is a resolution tree; if $T$ is a resolution tree, then so is $T'$ which is obtained from $T$ first by choosing a leaf $\gamma$ of $T$ and a binary vector $\eta$ such that whenever an entry $\gamma_i$ is tagged, the corresponding entry $\eta_i$ is zero, and then  by adding on to $T$ all the $e^\eta_\delta(\gamma)$ as new leafs, for $\delta\leq\eta$.  By \eqref{eq:preceq},  the new subset is indeed a tree. In particular, if every entry of some node $\theta\in T$ is tagged and $T$ is a resolution tree, then $\theta$ is necessarily a leaf of $T$. Moreover, if $T'$ is a subtree of $T$, that is to say, all nodes of $T$ greater than or equal to a fixed node, and $T$ is  a resolution tree, then so is $T'$.

\begin{lemma}\label{L:restree}
Let $X\sub\mathbb A^m$ be a closed subscheme and let $T\sub\Gamma_m$ be a subtree with root $\theta$. If $T$ is a resolution tree, then
$$
\class{\parc n\theta X}=\sum_{\gamma\in T\text{ leaf}}\class{\parc n\gamma X}
$$
in $\grotform \fld $, for all $n$.
\end{lemma}
\begin{proof}
By induction on the height of a node, immediate from Lemma~\ref{L:tmprec}.
\end{proof}

For a tree $T\sub\Gamma_m$, we define recursively what it means for it to be  \emph{$X$-rationalizing}: if all but one of its leafs $\gamma$ are  $X$-regular and if $\recur X\theta\gamma$, then $T$ is $X$-rationalizing. Furthermore, if $T$ is $X$-rationalizing,  $\gamma$ a leaf of $T$, and $T'$ an $X$-rationalizing tree with root $\gamma$, then the composite tree obtained by replacing $\gamma $ in $T$ by $T'$ is again $X$-rationalizing.
 
\begin{theorem}\label{T:rattree}
If $T$ is an $X$-rationalizing resolution tree  with root $\theta\in \Gamma_m$, then the geometric Igusa-zeta series $\igugeom X$ is rational over $\grotform \fld_\lef$.
\end{theorem} 
\begin{proof}
Let us first show that if $T$ is a resolution tree with root $\theta$ and for each leaf $\gamma$, the twisted geometric Igusa-zeta series $\igutwist X\gamma$ is rational, then so is $\igutwist X\theta$. Indeed, by Lemma~\ref{L:restree}, we have an identity
\begin{equation}\label{eq:restree}
\class{\parc n\theta X}=\sum_{\gamma\in T\text{ leaf}}\class{\parc n\gamma X}
\end{equation}
Multiplying with $t^n$, and summing over all $n$ then yields
$$
\igutwist X\theta= \sum_{\gamma\in T\text{ leaf}}\igutwist X\gamma,
$$
proving the claim. Hence, we may use the recursive definition of a rationalizing tree and the previous result, to reduce to the case that all but one leaf $\bar\gamma$ of $T$ is   $X$-regular, and $\recur X\theta{\bar\gamma}$. Assume first the \ch\ is zero. Again \eqref{eq:restree} holds, and by Corollary~\ref{C:parcreg}, the directed arcs along all leafs $\gamma\neq\bar\gamma$ are certain multiples of $\lef^{(m-1)n}$, where the multiple is independent from $n$, whereas   by Lemma~\ref{L:arcequiv}, the directed arc class along $\bar\gamma$ is a multiple of the class along $\theta$. More precisely, there exists an element $w\in\grotform \fld_\lef$ such that
$$
\class{\parc n\theta X}=w\cdot\lef^{(m-1)n}+\class{\parc {n-s}\theta X}\cdot\lef^r
$$
for all $n$,   where $s=\ord X{\bar\gamma}-\ord X\theta$ and $r=sm-\norm{\bar\gamma}+\norm\theta$.
Multiplying with $t^n$ and summing   over all sufficiently large $n$, we get an identity
$$
\igutwist X\theta=\frac {Q}{1-\lef^{m-1}t}+\lef^rt^s\igutwist X\theta
$$
for some polynomial $Q$ over $\grotform \fld_\lef$. Solving for $\igutwist X\theta$ then proves the claim, as $s>0$.
\end{proof}
 
\subsection*{Linear rationalization algorithm}
The algorithm that we will use here to construct an $X$-rationalizing resolution tree with root $\tuple 0$, thus establishing the rationality of the geometric Igusa-zeta series of a hypersurface $X$ by Theorem~\ref{T:rattree}, relies on the simple form the singular locus takes. Namely, we say that a hypersurface $X$  containing the origin has \emph{linear singularities}, if its singular locus is a union of coordinate subspaces, where a \emph{coordinate subspaces} is a closed subscheme given by equations $x_{i_1}=\dots=x_{i_s}=0$ for some subset $x_{i_j}$ of the variables. We will apply the algorithm to  hypersurfaces all of whose twisted initial forms have linear singularities.

\subsection*{Single-branch linear rationalization algorithm for power hypersurfaces with an isolated singularity}
In its simplest form, the algorithm works as follows: assume for every twisted initial form $\tin X\theta$ of $X$, there exists a variable $x_i$ such that the basic subset $(\tin X\theta)_{x_i}$   is smooth. We then apply  Lemma~\ref{L:tmprec} with $\eta=e_i$, thus building a binary tree with at each stage exactly one untagged and one tagged leaf, and such that the latter   is moreover $X$-regular, whence requires no further action. We continue this process (on the remaining untagged leaf) until we reach an untagged leaf $\gamma$ with $\recur X{\tuple 0}\gamma$, at which point we can invoke  Theorem~\ref{T:rattree}. If such a leaf $\gamma$ can be found, we say that the algorithm \emph{stops}.

This algorithm will stop on any hypersurface $X$  with an equation of the form
$$
f:=r_1x_1^{a_1}+\dots+r_mx_m^{a_m}
$$
with $a_i>0$ and $r_i\in \fld $; we will refer to such an $X$ as  \emph{power hypersurface}. In \ch\ zero, the origin is an isolated singularity, but in positive \ch, this is only the case if at most one of the powers $a_i$ is divisible by the \ch. In the isolated singularity case, the   algorithm as described above does apply: any   twisted initial form is again a power hypersurface; if it is one of the powers $x_i^{a_i}$, its regular locus, although empty, is obtained by inverting $x_i$, even if $a_i$ is divisible by the \ch; in the remaining case, we can always invert one variable whose power is not divisible by the \ch, yielding a smooth twisted initial form. So remains to show that this algorithm stops, that is to say, will eventually produce a leaf $\gamma$ such that $\recur X{\tuple 0}\gamma$. To see this, note that the set of all $\ord X\theta$, with $\theta$ running over all untagged nodes in the tree, is equal to the union   of all semi-groups $a_i\nat$, for $i=\range 1m$. Therefore, if $e$ is the least common multiple of all $a_i$, it will occur as some $\ord X\gamma$ for some untagged leaf $\gamma$ in this algorithm. However, it is easy to see that $\tin X\gamma=X$, and hence we showed:

\begin{theorem}\label{T:powerhyp}
The geometric Igusa zeta-series $\igugeom X$ of a power hypersurface $X$ with an isolated singularity is rational over $\grotform \fld_\lef$.\qed
\end{theorem}

In the next section,  we will work out in   complete detail the implementation of this algorithm for the power surface $x^2+y^3+z^4=0$. Generalizing these calculations, we will derive the following formula:

\begin{corollary}\label{C:powerhyp}
If $r_1x_1^{a_1}+\dots+r_mx_m^{a_m}=0$ is the equation of the power hypersurface $X$ with an isolated singularity, then there exists a polynomial $Q_X(t)\in \pol{\grotform \fld_\lef}t$ such that
$$
\igugeom X= \frac {Q_X(t)}{(1-\lef^{m-1}t)(1-\lef^Nt^e)}
$$
where $e$ is the least common multiple of   $a_1,\dots,a_m$, and where 
\begin{equation}\label{eq:lefN}
N=e(\frac {a_1-1}{a_1}+\dots+\frac {a_m-1}{a_m}).
\end{equation} 
\end{corollary}

\section{Rationality of the Igusa-zeta series for Du Val surfaces}\label{s:DuVal}
In this final section, we apply the previous rationalization algorithm to the geometric Igusa zeta series of   Du Val surfaces, which over a field $\fld $ of \ch\ different $p\neq 2$, are precisely  the isolated canonical singularities (at the origin $O$). Over $\mathbb C$, they can be realized, up to  analytic isomorphism, as the quotients  $\mathbb A^2/\Gamma$, where $\Gamma\sub\op{SL}_2(\mathbb C)$ is   a finite subgroup. A complete invariant is the dual resolution graph viewed as one of the following Dynkyn diagrams: $A_k$, $D_k$, $E_6$, $E_7$, or $E_8$, and we therefore will denote them simply by the latter letters. The main result of this section is the rationality of their geometric Igusa-zeta series over $\grotform \fld_\lef$, summarized by the following table, where we listed in the last column only the relevant factor in the denominator (the other factor being $(1-\lef^2t)$).

\begin{table}[h]
\caption{Denominator of the Igusa-zeta series   for Du Val surfaces ($p\neq2$). }
\label{tab:duval}     
\begin{tabular}{|c|r|c|c|}
\hline
\multicolumn{2}{|c|}{Du Val surface $X$}&equation  &denominator
 of $\breathe(1-\lef^2t)\igugeom{X}$  
   \\
\hline\hline
\multirow{2}{*}{$\phantom{xx}A_k\phantom{xx}$} &$\breathe k$ odd&\multirow{2}{*}{$x^2+y^2+z^{k+1}$} & $(1-\lef^{2k+1}t^{k+1})$ \\\cline{2-2}\cline{4-4}
&$\breathe  k$ even&&$(1-\lef^{4k+2}t^{2k+2})$\\
\hline
\multicolumn{2}{|c|}{$\breathe D_k$}&$x^2+y^2z+z^{k-1}$ & $(1-\lef^{4k-5}t^{2k-2})$ \\ 
\hline
\multicolumn{2}{|c|}{$\breathe E_6$}&$x^2+y^3+z^4$& $(1-\lef^{23}t^{12})$ \\
\hline
\multicolumn{2}{|c|}{$\breathe E_7$}&$x^2+y^3+yz^3$ & $(1-\lef^{29}t^{18})$ \\
\hline
\multicolumn{2}{|c|}{$\breathe E_8$}&$x^2+y^3+z^5$ & $(1-\lef^{60}t^{30})$ \\
\hline
\end{tabular}
\end{table}

 
\subsection*{The $E_6$-surface}\label{s:E6}
Let us work  step-by-step through the  rationalization algorithm for the   Du Val surface  $E_6$ with equation $x^2+y^3+z^4$.  We take a `short-cut' by observing that the origin $O$ is an isolated singularity, so that we only need   to calculate the class of $\parc n{(1,1,1)}{E_6}=\inverse{\rho_n}O$ by Theorem~\ref{T:redfiber}. By \eqref{eq:twist}, its arc equations are
 $$
 \xi^2\dot x^2+\xi^3\dot y^3+\xi^4\dot z^4\equiv0\mod \xi^n
 $$
 together with the initial conditions $\tilde x_i=\tilde y_i=\tilde z_i=0$ for   $i\geq n-1$. 
 The twisted initial  form   is $x^2$. According to the algorithm, we have a single branching   given by the transformations $e_1^*$ and $e_1$. The twisted initial  form of $e_1^*(1,1,1)=(1^*,1,1)$ is defined by $x_*^2=0$ and hence is empty. So remains the untagged leaf   $e_1(1,1,1)=(2,1,1)$, with arc equations
 $$
 \xi^4\dot x^2+\xi^3\dot y^3+\xi^4\dot z^4\equiv0\mod \xi^n
 $$
and in  addition to the previous initial conditions, also $\tilde x_{n-2}=0$. As the twisted initial  form is   $y^3$, we branch with $e_2^*$ and $e_2$. The twisted initial  form of $e_2^*(2,1,1)=(2,1^*,1)$ is   $y_*^3=0$, whence empty, leaving us with  $e_2(2,1,1)=(2,2,1)$, whose arc equations are
\begin{equation}\label{eq:221}
 \xi^4\dot x^2+\xi^6\dot y^3+\xi^4\dot z^4\equiv0\mod \xi^n
\end{equation}
and an additional initial condition $\tilde y_{n-2}=0$. The new twisted initial  form is $x^2+z^4$. At this point, inverting either variable $x$ or $z$ yields a regular surface. However, instead of choosing one, we may perform a multi-branching step, in which we consider all four possibilities  $e_1e_3$,  $e_1^*e_3$,  $e_1e_3^*$, or  $e_1^*e_3^*$, when applying   \eqref{eq:basicsum}, yielding the four leafs $(3,2,2)$, $(2^*,2,2)$, $(3,2,1^*)$, and $(2^*,2,1^*)$ respectively. The corresponding initial forms are given by $x^2+y^3=0$, $x_*^2=0$, $z^4_*=0$, and $x^2_*+z_*^4=0$. The middle two clearly are empty, and as the last  is smooth, we may invoke Corollary~\ref{C:parcreg}, to get
 $$
 \class{\parc n{(2^*,2,1^*)}{E_6}}=\class{\tin{E_6}{(2^*,2,1^*)}}\cdot \lef^{2n-2+4-5}=\class{x^2_*+z_*^4} \cdot\lef^{2n-3}
 $$
 as $\ord{(2^*,2,1^*)}{E_6}=4$. This leaves the first leaf, $(3,2,2)$, with arc equations   
  $$
 \xi^6\dot x^2+\xi^6\dot y^3+\xi^8\dot z^4\equiv0\mod \xi^n
 $$
 and the two additional initial conditions  $\tilde x_{n-3}=\tilde z_{n-2}=0$. Its twisted initial form  $x^2+y^3$  becomes non-singular if we invert $x$ or $y$, suggesting another multi-branching step. Inverting one and equating the other   to zero leads once more to contradictory equations,  
so we only have to deal with the two leafs $(3^*,2^*,2)$ and $(4,3,2)$. For the former, we may invoke once more Corollary~\ref{C:parcreg}, yielding the class 
$$
\class{x^2_*+y_*^3=0}\cdot\lef^{2n-2+6-7},
$$
as $\ord{(3^*,2,2^*)}{E_6}= 6$. The latter has arc equations 
 $$
 \xi^8\dot x^2+ \xi^9\dot y^3+ \xi^8\dot z^4\equiv0\mod \xi^n
 $$
together with the vanishing of $\tilde x_i$, $\tilde y_i$ and $\tilde z_i$  for $i$ greater than or equal to respectively $n-4$, $n-3$, and $n-2$. Since $(4,3,2)$ has the same twisted initial form as $(2,2,1)$,
we may repeat our previous argument. Tagging both variables gives the leaf $(4^*,3,2^*)$ and
$$
\class{x^2_*+z_*^4}\cdot \lef^{2n-2+8-9}
 $$
 as $\ord{(4^*,3,2^*)}{E_6}=8$. The latter leaf is $(5,3,3)$,  with arc equations
 $$
 \xi^{10}\dot x^2+ \xi^9\dot y^3+\xi^{12}\dot z^4\equiv0\mod \xi^n
 $$
together with  $\tilde x_i, \tilde y_i, \tilde z_i=0$  for $i\geq n-5,n-3,n-3$ respectively. As the twisted initial form is $y^3$, we again branch  over $e_2^*$ and $e_2$ leading to the leaf $(5,4,3)$,  with arc equations 
 $$
 \xi^{10}\dot x^2+\xi^{12}\dot y^3+\xi^{12}\dot z^4\equiv0\mod \xi^n
 $$
together with  $\tilde x_i, \tilde y_i, \tilde z_i=0$,  for $i\geq n-5,n-4,n-3$ respectively. As $x^2$ is the new twisted initial form, we branch   over $e_1^*$ and $e_1$, yielding the leaf $(6,4,3)$,  with arc equations 
 $$
 \xi^{12}\dot x^2+\xi^{12}\dot y^3+\xi^{12}\dot z^4\equiv0\mod \xi^n
 $$
together with   $\tilde x_i, \tilde y_i, \tilde z_i=0$  for $i\geq n-6,n-4,n-3$ respectively. As $X$ itself is the twisted initial form of this leaf, that is to say, $\recur X{\tuple 0}{(6,4,3)}$, our algorithm has come to a halt. Indeed, if we factor out $\xi^{12}$ in the last equation, we get the $(n-12)$-th arc equations. Since we have $\norm{(6,4,3)}=13$ additional initial conditions, we are left with   $3\cdot 12-13=23$ free variables  $\tilde x_i, \tilde y_i, \tilde z_i$  for $n-12\leq i< n-6,n-4,n-3$ respectively, as predicted by Lemma~\ref{L:arcequiv}. Putting everything together, we showed that $\class{\arc n{E_6}}$ is equal to 
\begin{equation}\label{eq:E6arc}
\class{E_6-O} \lef^{2n-2}+2\class{\tin {E_6}{(2^*,2,1^*)}}  \lef^{2n-3}+\class{\tin {E_6}{(3^*,2^*,2)}} \lef^{2n-3}+\class{\arc{n-12} {E_6}} \lef^{23}
\end{equation}
Multiplying with $t^n$, summing over all $n$, and solving for the zeta series yields
$$
\igugeom {E_6} =\frac{Q_{E_6}}{(1-\lef^2t)(1-\lef^{23}t^{12})}
$$
for some polynomial $Q_{E_6}$ over $\grotform \fld_\lef$. A schematic representation of these calculations is given by the following rationalization tree, in which     we equated, for brevity, a leaf to the class of the corresponding directed arc scheme   (giving only its defining polynomial):

\begin{table}[h]
  {\tiny
\xymatrix@!C=4pc{
&&(1,1,1)\ar[d]^{e_1}\ar[dr]^{e^*_1}\\
&&(2,1,1)\ar[d]^{e_2}\ar[dr]^{e_2^*}&(1^*,1,1)=0\\
&&(2,2,1)\ar[d]^{e_1e_3}\ar[dr]^{e_1^*e_3^*}\ar[lld]_{e_1e_3 ^*}\ar[ld]^{e_1^*e_3}&(2,1^*,1)=0\\
(3,2,1^*)=0&(2^*,2,2)=0&(3,2,2)\ar[d]^{e_1e_2}\ar[dr]^{e_1^*e_2^*}\ar[lld]_{e_1^*e_2}\ar[ld]^{e_1e_2^*}&(2^*,2,1^*)\ar@{=}[r]& \class{x^2_*+z^4_*}\lef^{2n-3} \\
(4,2^*,2)=0&(3^*,3,2)=0&(4,3,2)\ar[d]^{e_1e_3}\ar[dr]^{e_1^*e_3^*}\ar[lld]_{e_1e_3 ^*}\ar[ld]^{e_1^*e_3}&(3^*,2^*,2)\ar@{=}[r]& \class{x^2_*+y^3_*}\lef^{2n-3} \\
(5,3,2^*)=0&(4^*,3,3)=0&(5,3,3)\ar[d]^{e_2}\ar[dr]^{e_2^*}&(4^*,3,2^*)\ar@{=}[r]& \class{x^2_*+z^4_*}\lef^{2n-3} \\
&&(5,4,3)\ar[d]^{e_1}\ar[dr]^{e^*_1}&(5,3^*,3)=0\\
&&(6,4,3)\ar@{=}[d]&(5^*,4,3)=0\\
&&\class{\arc {n-12}{E_6}}\lef^{23}&\\
}}
\caption{\label{t:E6}The rationalization tree for the $E_6$-surface}
\end{table}

 It is now also clear how this generalizes to any power hypersurface, yielding a proof of Corollary~\ref{C:powerhyp}.  Indeed, with $e$ the least common multiple of the $a_i$, the algorithm stops at the leaf $\gamma:=(\frac e{a_1},\dots,\frac e{a_m})$, whose order is $e$. During this process, we introduced $\norm\gamma$ many additional initial conditions. As we have $me$ more arc variables for the $n$-th arc as for the $(n-e)$-th arc, this yields $N=me-\norm\gamma$ free variables, explaining   formula~\eqref{eq:lefN}.
 
  Let us apply this algorithm also to the Du Val surfaces  $A_k$ and $E_8$. The former is given by $x^2+y^2+z^{k+1}$. If $k$ is odd, then $e=k+1$ and $N=3(k+1)-(\frac{k+1}2+\frac{k+1}2+1)=2k+1$, and hence 
 $$
 \igugeom{A_k}=\frac{Q_{A_k}}{(1-\lef^2t)(1-\lef^{2k+1}t^{k+1})}.
 $$
  If $k$ is even, then $e=2(k+1)$ and $N=6(k+1)-(k+1+k+1+2)=4k+2$, so that 
   $$
 \igugeom{A_k}=\frac{Q_{A_k}}{(1-\lef^2t)(1-\lef^{4k+2}t^{2(k+1)})}.
 $$
 Finally, since $E_8$ has equation $x^2+y^3+z^5$, the values are $e=30$ and $N=90-(15+10+5)=60$, so that 
 $$
 \igugeom{E_8}=\frac{Q_{E_8}}{(1-\lef^2t)(1-\lef^{60}t^{30})}. 
 $$

 Although a priori $\tin X\theta$ depends on the embedding of $X$ in some affine space, its class  may be  more independent from this embedding. For instance, in \eqref{eq:E6arc}, all but  the two middle terms are independent from an embedding. To which extent does this hold?

\subsection*{The $E_7$-surface}
 This time, the defining equation is $x^2+y^3+yz^3=0$ (and, as before,  assuming that the  \ch\ is not equal to $2$  or $3$), which again has an isolated (canonical) singularity. As this is no longer just a sum of powers, it will lead to a more complicated rationalization tree, given in  Table~\ref{t:E7} below.

  \begin{table}
  {\tiny
\xymatrix@!C=4pc{
&&(1,1,1)\ar[d]^{e_1}\ar[dr]^{e^*_1}\\
&&(2,1,1)\ar[d]^{e_2}\ar[dr]^{e_2^*}&(1^*,1,1)=0\\
&&(2,2,1)\ar[d]^{e_1}\ar[dr]^{e^*_1}&(2,1^*,1)=0\\
&&(3,2,1)\ar[d]^{e_3}\ar[dr]^{e_3^*}&(2^*,2,1)=0\\
&&(3,2,2)\ar[d]^{e_1e_2}\ar[dr]^{e_1^*e_2^*}\ar[lld]_{e_1e_2 ^*}\ar[ld]^{e_1^*e_2}&(3,2,1^*)=0\\
(4,2^*,2)=0&(3^*,3,2)=0&(4,3,2)\ar[d]^{e_1}\ar[dr]^{e^*_1}&(3^*,2^*,2)\ar@{=}[r]& \class{x^2_*+y^3_*}\lef^{2n-3} \\
&&(5,3,2)\ar[d]^{e_2e_3}\ar[dr]^{e_2^*e_3^*}\ar[lld]_{e_2^*e_3}\ar[ld]^{e_2e_3^*}&(4^*,3,2)=0\\
(5,3^*,3)=0&(5,4,2^*)\ar@{=}[d]&(5,4,3)\ar[d]^{e_1}\ar[dr]^{e^*_1}&(5,3^*,2^*)\ar@{=}[r]&\class{y^3_*+y_*z^3_*}\lef^{2n-3}\\
&\class{x^2+yz^3_*}\lef^{2n-3}&(6,4,3)\ar[d]^{e_1e_2}\ar[dr]^{e_1^*e_2^*}\ar[lld]_{e_1^*e_2}\ar[ld]^{e_1e_2^*}&(5^*,4,3)=0\\
(6^*,5,3)=0&(7,4^*,3)=0&(7,5,3)\ar[d]^{e_1e_3}\ar[dr]^{e_1^*e_3^*}\ar[lld]_{e_1^*e_3}\ar[ld]^{e_1e_3^*}&(6^*,4^*,3)\ar@{=}[r]&\class{x^2_*+y^3_*}\lef^{2n-3}\\
(7^*,5,4)=0&(8,5,3^*)\ar@{=}[d] &(8,5,4)\ar[d]^{e_2}\ar[dr]^{e^*_2}&(7^*,5,3^*)\ar@{=}[r]&\class{x^2_*+yz^3_*}\lef^{2n-3}\\
&\class{yz^3_*}\lef^{2n-4}&(8,6,4)\ar[d]^{e_1}\ar[dr]^{e^*_1}&(8,5^*,4)=0\\
&&(9,6,4)\ar@{=}[d]&(8^*,6,4)=0\\
&&\class{\arc {n-18}{E_7}}\lef^{29}&\\
}}
\caption{\label{t:E7}The rationalization tree for the $E_7$-surface}
\end{table}

  Since in any rationalization tree, the sum of all nodes is equal to the root, we get, using Theorem~\ref{T:redfiber}, that
$$
\class{\arc n{E_7}}=q\lef^{2n-3}+ \class{\arc {n-18}{E_7}}\lef^{29}
$$
where   $q$ is equal to 
$$
\class{{E_7}-O}\lef+\class{x^2+yz^3_*}+ \class{yz^3_*}\lef^{-1}+2 \class{x^2_*+y^3_*}+\class{x^2_*+yz^3_*}+\class{y^3_*+y_*z_*^{3}}.
$$
Using the  identities   
\begin{align*}
\class{x^2+yz^3_*}&=\lef^2-\lef=\class{yz^3_*}\\
\class{x^2_*+yz^3_*}&=(\lef-1)^2\\
\class{y^3_*+y_*z_*^{3}}&=(\lef-1)\class{x^2_*+y^3_*}
\end{align*}
we get 
$$
q= \lef\class{{E_7}-O} +(\lef+1)\class{x_*^2+y^3_*}+2\lef(\lef-1).
$$
 Regardless the value of $q$, the usual argument yields the rationality of  $\igugeom {E_7}$, with denominator equal to $(1-\lef^2t)(1-\lef^{29}t^{18})$.

 \vfill\eject
 \subsection*{The $D_k$-surface}
 The general equation of the $D_k$-surface   is $x^2+y^2z+z^{k-1}$ for $k\geq 4$ (assuming the \ch\ is different from $2$). Depending on whether $k$ is odd or even, we have two slightly different rationalization trees, both leading from the root $(1,1,1)$ to $(k-1,k-2,2)$, where
 $$
 \class{\parc n{(k-1,k-2,2)}{D_k}}=\class{\arc{n-2k+2}{D_k}}\cdot\lef^{4k-5}
 $$
 since we have $3(2k-2)-\norm{(k-1,k-2,2)}=4k-5$ free variables. The equations for  the directed arc along the starting value $(1,1,1)$  has arc equations 
 $$
 \xi^2\dot x^2+\xi^3\dot y^2\dot z+\xi^{k-1}\dot z^{k-1}\equiv0\mod \xi^n.
 $$
 It is not hard to see that the initial part of the tree is given by alternating $e_1$ and $e_2$. As a result, in  respectively the first and second term, the power of $\xi$ is each time increased by $2$. This goes on until one of them catches up with the power $\xi^{k-1}$, and this depends on the parity of $k$. So assume first that $k$ is odd. In that case, we arrive at a leaf with value $(\frac{k-1}2,\frac{k-1}2,1)$, whose directed arc has arc equations 
 $$
 \xi^{k-1}\dot x^2+\xi^k\dot y^2\dot z+\xi^{k-1}\dot z^{k-1}\equiv0\mod \xi^n.
 $$
 The remainder of the tree is now as follows, where the middle part gets repeated   until the indicated value is reached:

  \begin{table}[h]
  {\tiny
\xymatrix@!C=4pc{
&&(\frac{k-1}2,\frac{k-1}2,1)\ar[d]^{e_1e_3}\ar[dr]^{e_1^*e_3^*}\ar[ld]^{e_1^*e_3}\ar[lld]_{e_1e_3^*}\\
(\frac{k+1}2,\frac{k-1}2,1^*)=0&(\frac{k-1}2^*,\frac{k-1}2,2)=0&(\frac{k+1}2,\frac{k-1}2,2)\ar[d]^{e_1e_2}\ar[dr]^{e_1^*e_2^*}\ar[lld]_{e_1e_2^*}\ar[ld]^{e_1^*e_2}&(\frac{k-1}2^*,\frac{k-1}2,1^*)\ar@{=}[r]&\class{x_*^2+z_*^{k-1}}\lef^{2n-3}\\
(\frac{k+3}2,\frac{k-1}2^*,2)\ar@{=}[d]&(\frac{k+1}2^*,\frac{k+1}2,2)=0  &(\frac{k+3}2,\frac{k+1}2,2)\ar@{.>}[dd]^{e_1e_2}\ar@{.>}[ddr]^{e_1^*e_2^*}\ar@{.>}[lldd]_{e_1e_2^*}\ar@{.>}[ldd]^{e_1^*e_2} &(\frac{k+1}2^*,\frac{k-1}2^*,2)\ar@{=}[r]&\class{x_*^2+y^2_*z}\lef^{2n-3}\\
\class{y_*^2z}\lef^{2n-4}&&&\\
\class{y_*^2z}\lef^{2n-4}&0&(k-2,k-3,2)\ar[d]^{e_1e_2}\ar[dr]^{e_1^*e_2^*}\ar[lld]_{e_1e_2^*}\ar[ld]^{e_1^*e_2}&\class{x_*^2+y^2_*z}\lef^{2n-3}\\
\class{y_*^2z}\lef^{2n-4}&0&(k-1,k-2,2)&\class{x_*^2+y^2_*z}\lef^{2n-3}\\
}}
\caption{\label{t:Dkodd}Bottom part of the rationalization tree for   $D_k$, when $k$ is odd.}
\end{table}

\vfill\eject
The tree for even $k$   is analogous, where this time, the starting value is $(\frac k2,\frac{k-2}2,1)$,  with arc equations
$$
 \xi^k\dot x^2+\xi^{k-1}\dot y^2\dot z+\xi^{k-1}\dot z^{k-1}\equiv0\mod \xi^n.
 $$
The remainder of the tree, with the middle part again repeated, is

  \begin{table}[h] 
  {\tiny
\xymatrix@!C=4pc{
&&(\frac{k}2,\frac{k-2}2,1)\ar[d]^{e_2e_3}\ar[dr]^{e_2^*e_3^*}\ar[lld]_{e_2e_3^*}\ar[ld]^{e_2^*e_3}\\
(\frac{k}2,\frac{k}2,1^*)=0&(\frac{k}2,\frac{k-2}2^*,2)=0 &(\frac{k}2,\frac{k}2,2)\ar[d]^{e_1}\ar[dr]^{e_1^*}&(\frac{k}2,\frac{k-2}2^*,1^*)\ar@{=}[r]&\class{y_*^2z_*+z_*^{k-1}}\lef^{2n-3}\\
&&(\frac{k+2}2,\frac{k}2,2)\ar[d]^{e_1e_2}\ar[dr]^{e_1^*e_2^*}\ar[lld]_{e_1e_2^*}\ar[ld]^{e_1^*e_2}&(\frac{k}2^*,\frac{k}2,2)=0\\
(\frac{k+4}2,\frac{k}2^*,2)\ar@{=}[d]&(\frac{k+2}2^*,\frac{k+2}2,2)=0  & (\frac{k+4}2,\frac{k+2}2,2)\ar@{.>}[dd]^{e_1e_2}\ar@{.>}[ddr]^{e_1^*e_2^*}\ar@{.>}[lldd]_{e_1e_2^*}\ar@{.>}[ldd]^{e_1^*e_2} &(\frac{k+2}2^*,\frac{k}2^*,2)\ar@{=}[r]&\class{x_*^2+y^2_*z}\lef^{2n-3}\\
\class{y_*^2z}\lef^{2n-4}&&&\\
\class{y_*^2z}\lef^{2n-4}&0&(k-2,k-3,2)\ar[d]^{e_1e_2}\ar[dr]^{e_1^*e_2^*}\ar[lld]_{e_1e_2^*}\ar[ld]^{e_1^*e_2}&\class{x_*^2+y^2_*z}\lef^{2n-3}\\
\class{y_*^2z}\lef^{2n-4}&0&(k-1,k-2,2)&\class{x_*^2+y^2_*z}\lef^{2n-3}\\
}}
\caption{\label{t:Dkodd}Bottom part of the rationalization tree for   $D_k$, when $k$ is even.}
\end{table}

Note that $\class{y_*^2z}=\lef^2-\lef$   appears  $\round k2$ times as an end value, and $\class{x_*^2+y^2_*z}=(\lef-1)^2$ appears $\round {k+2}2$ many times. It follows that $\class{\arc n{D_k}}$ is equal to 
$$
\big(\class{{D_k}-O}\lef+\class{x^2_*+z^{k-1}_*}+\frac{k+1}2(\lef^2-\lef)+(\lef-1)^2\big)\lef^{2n-3}+ \class{\arc {n-2k+2}{D_k}}\lef^{4k-5}
$$
 in the odd case, and to 
$$
\big(\class{{D_k}-O}\lef+\class{y^2_*z_*+z^{k-1}_*}+\frac {k}2(\lef^2-\lef)+(\lef-1)^2\big)\lef^{2n-3}+ \class{\arc {n-2k+2}{D_k}}\lef^{4k-5}
$$ 
 in the even case. In particular,  the geometric Igusa-zeta series $\igugeom{D_k}$ is rational with denominator equal to $(1-\lef^2t)(1-\lef^{4k-5}t^{2k-2})$.

\vfill \eject
\section{Appendix: idempotents in Noetherian local rings}\label{s:app}
Let us call two elements in a ring $A$ \emph{orthogonal} if their product is zero.
Recall that an element $e\in A$ is called \emph{idempotent}, if $e^2=e$; the set of all idempotents in $A$ will be denoted $\op{Idem}(A)$. If $e$ is idempotent, then so is $1-e$. Moreover, $e$ and $1-e$ are orthogonal, and the ideals they generate   are subrings of $A$ with the property that $A\iso eA\oplus(1-e)A$. In fact, $A/eA\iso (1-e)A$. In particular, $A$ is indecomposable, or, equivalently, $\op{Spec}A$ is connected, \iff\ $0$ and $1$ are its only idempotents. We define a partial order relation on $\op{Idem}(A)$ by  $e\leq f$ \iff\ $eA\sub fA$.   Clearly, $0\leq e\leq 1$, for all $e\in\op{Idem}(A)$.   An idempotent is called \emph{primitive} if it cannot be written as the sum of two non-zero orthogonal idempotents. Note that the sum of two orthogonal idempotents is again idempotent.

\begin{lemma}\label{L:idem}
For idempotents $e,f,g\in A$, we have
\begin{enumerate}
  \item\label{i:idemord}   $e\leq f$ \iff\ $e=ef$; 
  \item\label{i:sum} if $e=f+g$ with $f$ and $g$ orthogonal idempotents, then $f,g\leq e$
   \item\label{i:sumzero} $0$ is primitive;
   \item\label{i:primsum} a non-zero idempotent is primitive \iff\ it cannot be written as the sum of two strictly smaller orthogonal idempotents;
      \item\label{i:prim} a non-zero idempotent is  primitive \iff\ it is minimal among all non-zero idempotents.
\end{enumerate}
\end{lemma}
\begin{proof}
One direction in \eqref{i:idemord} is clear, so assume $e\in fA$, say, $e=af$. Squaring this equality yields $e=a^2f=a(af)=ae$. On the other hand, $0=(1-e)e=(1-e)af$, so that $af=eaf=(ae)f=ef$, and hence $e=ef$. For \eqref{i:sum}, multiply with $f$ to get $ef=f+fg=f$, whence $f\leq e$ by \eqref{i:idemord}.
To prove \eqref{i:sumzero}, suppose $e+f=0$ with $e,f$ orthogonal idempotents. By \eqref{i:sum}, we get $e,f\leq 0$, and hence $e=f=0$.  Property~\eqref{i:primsum} follows immediately from \eqref{i:sum}. 
For \eqref{i:prim}, let $e$ be primitive and suppose $0\neq f\leq e$, whence   $f=fe$ by \eqref{i:idemord}. Multiplying $1=f+(1-f)$ with $e$ yields $e=ef+e(1-f)$, so that by primitivity one of the two terms must be zero. As $f=ef\neq0$, we get $e(1-f)=0$, whence $e=ef=f$. Conversely, assume $e$ is minimal among all non-zero idempotents, but assume it is not primitive. Hence, by \eqref{i:primsum}, we can write it as $e=f+g$ with $f,g< e$ orthogonal idempotents. By minimality, $f=g=0$, contradicting that $e\neq0$. 
\end{proof}

\begin{lemma}\label{L:dcc}
If $A$ is Noetherian, then the order relation on $\op{Idem}(A)$ satisfies the ascending and descending chain conditions.
\end{lemma}
\begin{proof}
Since  $A$ is Noetherian, ideal inclusion satisfies the ascending chain condition, and, therefore, so does $\leq$. To prove that $\leq$ also satisfies the descending chain condition, observe that $e\leq f$ is equivalent with $1-f\leq 1-e$ by \eqref{i:idemord}, so that we may apply the ascending chain condition to the latter.
\end{proof}

\begin{proposition}\label{P:primidem}
If $A$ is Noetherian, then there are only finitely many primitive idempotents, say,  $e_1,\dots,e_s$. The $e_i$ are mutually orthogonal, and  any idempotent $e$ can be written  in a unique way as a sum $e=e_{i_1}+\dots+e_{i_t}$ with $1\leq i_1<\dots<i_t\leq s$. In particular, $e_1+\dots+e_s=1$, and $\op{Idem}(A)$ is finite. Moreover, $A\iso \bigoplus_{i=1}^s e_iA$, and each $e_iA\iso A/(1-e_i)A$ is indecomposable.
\end{proposition}
\begin{proof}
We start with proving that any idempotent is the sum of mutually orthogonal primitive idempotents. Let $e\in\op{Idem}(A)$. If $e$ is not primitive, then we can find non-zero orthogonal idempotents $e_1,e_2\in A$ such that $e=e_1+e_2$. By Lemma~\ref{L:idem}\eqref{i:sum}, we have $e_1,e_2\leq e$, whence $e_1,e_2<e$, since $e_1,e_2$ are non-zero.   If either one of $e_i$ is not primitive, we may continue in this way, and split it up as a sum of two orthogonal idempotents both strictly less than $e_i$. This cannot go on indefinitely, for then we would get an infinite strictly descending chain in $\op{Idem}(A)$, contradicting Lemma~\ref{L:dcc}. This shows that any idempotent admits a decomposition into mutually orthogonal primitive idempotents. In particular, 
\begin{equation}\label{eq:decompunit}
1=e_1+\dots+e_s,
\end{equation}
 with all $e_i$ mutually orthogonal primitive idempotents. Suppose $1=f_1+\dots+f_t$ is another such decomposition. Multiplying \eqref{eq:decompunit} with $f_1$, we get 
\begin{equation}\label{eq:f1}
f_1=e_1f_1+\dots+e_sf_1.
\end{equation}
 Since all $e_if_1$ are orthogonal, and since $f_1$ is primitive, all terms in the right hand side of \eqref{eq:f1} must be zero except one, say, $e_1f_1=f_1$. In particular, $f_1\leq e_1$, whence $f_1=e_1$ by Lemma~\ref{L:idem}\eqref{i:prim}. This shows that the $e_i$ are unique. Let $e$ be an arbitrary idempotent. Multiplying \eqref{eq:decompunit} with $e$, we get $e=ee_1+\dots+ee_s$. Since $ee_i\leq e_i$ and $e_i$ is primitive, Lemma~\ref{L:idem}\eqref{i:prim} implies that $ee_i$ is either zero or equal to $e_i$. In conclusion, $e$ is the sum of some of the $e_i$ (in fact, those such that $e_i\leq e$). In particular, if $e$ is primitive and non-zero, then it must be equal to one of the $e_i$, showing that $e_1,\dots,e_s$ are all primitive idempotents. Moreover, since any idempotent is a sum of these, $\op{Idem}(A)$ is finite. The last assertion now follows readily.
 \end{proof}

\begin{theorem}\label{T:stabisosch}
Two affine $X$-schemes  are stably isomorphic over $X$ \iff\ they are isomorphic over $X$.
\end{theorem}
\begin{proof}
Let $B$ and $\tilde B$ be two finitely generated $A$-algebras, such that their corresponding affine schemes are stably isomorphic. Hence, there exists a finitely generated $A$-algebra $C$ such that $D:=C\oplus B$ and $\tilde D:=C\oplus \tilde B$ are isomorphic as $A$-algebras.   In fact, all we will use from these rings is that they are Noetherian, and hence we will show that, if $C\oplus B \iso  C\oplus \tilde B$, for some Noetherian rings $B$, $\tilde B$, and $C$,  then $B\iso\tilde B$.   Let $c_1,\dots,c_s$ be the   primitive idempotents of $C$, which are finite in number by Proposition~\ref{P:primidem}, and likewise $e_1,\dots,e_n$ and $\tilde e_1,\dots,\tilde e_{\tilde n}$ those of $B$ and $\tilde B$ respectively. Hence $\rij d{s+n}:=(c_1,\dots,c_s,e_1,\dots,e_n)$ and $\rij{\tilde d}{s+\tilde n}:=(c_1,\dots,c_s,\tilde e_1,\dots,\tilde e_{\tilde n})$ are the primitive idempotents of $D$ and $\tilde D$ respectively (since their sums are equal to one; see Proposition~\ref{P:primidem}). 

Since the isomorphism $f\colon D\to \tilde D$ must preserve  primitive idempotents,  $n=\tilde n$ and there exists a permutation $\sigma$ of $\{1,\dots,n+s\}$ such that $f(d_i)=\tilde d_{\sigma(i)}$. Fix some $1\leq i\leq s+n$ and let $i':=\sigma(i)$. Since $f(d_i)=\tilde d_{i'}$, the isomorphism $f$ induces an isomorphism between  $d_iD\iso D/(1-d_i)D$ and $\tilde D/(1-\tilde d_{i'})\tilde D\iso \tilde d_{i'}\tilde D$. Assume $i'\leq s$, so that $\tilde d_{i'}\tilde D\iso c_{i'}C\iso d_{i'}D$. In other words, we showed that $d_iD\iso d_{\sigma(i)}D$ whenever $\sigma(i)\leq i$. Letting $i'':=\sigma(i')$, we may repeat this argument, and hence, if $i''\leq s$, we also get $d_iD\iso d_{\sigma^2(i)}D$. In conclusion, if $k$ is the smallest positive power of $\sigma$ such that $\sigma^k(i)>s$ (it is possible that no such power exists), then   $d_iD\iso d_{\sigma^{k-1}(i)}D\iso \tilde e_{\sigma^k(i)-s}\tilde B$. 

In particular, if $i>s$, then there exists such a smallest power $k$ (since the order of $\sigma$ is an upper bound for $k$).  Writing $j:=i-s$ and $j':=\sigma^k(i)-s$, we showed that for each $j=\range 1n$, there exists $j'\in\{1,\dots,n\}$ such that $e_jB\iso \tilde e_{j'}\tilde B$. Furthermore, the assignment $j\mapsto j'$ is injective, for if $j_1'=j_2'$, then for some $k_1$ and $k_2$ we have 
\begin{equation}\label{eq:sigma12}
\sigma^{k_1}(j_1+s)=\sigma^{k_2}(j_2+s), 
\end{equation}
where $k_1$ and $k_2$ are the respective smallest powers of $\sigma$ that map $j_1+s$ and $j_2+s$ to a value greater than $s$. Without loss of generality, assume $k_2\leq k_1$. Hence applying $\sigma^{-k_2}$ to \eqref{eq:sigma12} yields $\sigma^{k_1-k_2}(j_1+s)=j_2+s>s$, so that by minimality, $k_1=k_2$, and hence $j_1=j_2$.
In particular, by the pigeonhole principle, the assignment $j\mapsto j'$ is a bijection, and, therefore, using Proposition~\ref{P:primidem}, we get an isomorphism
$$
B\iso \bigoplus_{j=1}^ne_jB\iso \bigoplus_{j=1}^n\tilde e_{j'}\tilde B\iso \tilde B,
$$
as required.
\end{proof}


\begin{thebibliography}{10}

\bibitem{BLR}
S.~Bosch, W.~L{\"u}tkebohmert, and M.~Raynaud, \emph{N\'eron models},
  Ergebnisse der Mathematik und ihrer Grenzgebiete (3) [Results in Mathematics
  and Related Areas (3)], vol.~21, Springer-Verlag, Berlin, 1990. \MR{MR1045822
  (91i:14034)}

\bibitem{CH}
R.~Cluckers and D.~Haskell, \emph{Euler characteristic and dimension of
  semi-algebraic $p$-adic sets}, preprint, 2000.

\bibitem{CHGrot}
\bysame, \emph{Grothendieck rings of {$\Bbb Z$}-valued fields}, Bull. Symbolic
  Logic \textbf{7} (2001), no.~2, 262--269.

\bibitem{CrawMot}
A.~Craw, \emph{An introduction to motivic integration}, Strings and geometry,
  Clay Math. Proc., vol.~3, Amer. Math. Soc., Providence, RI, 2004,
  pp.~203--225. \MR{MR2103724 (2005k:14027)}

\bibitem{DLIgu}
J.~Denef and F.~Loeser, \emph{Motivic {I}gusa zeta functions}, J. Algebraic
  Geom. \textbf{7} (1998), no.~3, 505--537.

\bibitem{DLArcs}
\bysame, \emph{Germs of arcs on singular algebraic varieties and motivic
  integration}, Invent. Math. \textbf{135} (1999), no.~1, 201--232.

\bibitem{DLDwork}
\bysame, \emph{On some rational generating series occurring in arithmetic
  geometry}, Geometric aspects of Dwork theory. Vol. I, II, Walter de Gruyter
  GmbH \& Co. KG, Berlin, 2004, pp.~509--526.

\bibitem{Eis}
D.~Eisenbud, \emph{Commutative algebra with a view toward algebraic geometry},
  Graduate Texts in Mathematics, vol. 150, Springer-Verlag, New York, 1995.

\bibitem{Hart}
R.~Hartshorne, \emph{Algebraic geometry}, Springer-Verlag, New York, 1977.

\bibitem{HuKaz}
Ehud Hrushovski and David Kazhdan, \emph{The value ring of geometric motivic
  integration, and the {I}wahori {H}ecke algebra of {${\rm SL}\sb 2$}}, Geom.
  Funct. Anal. \textbf{17} (2008), no.~6, 1924--1967, With an appendix by Nir
  Avni.

\bibitem{KraSca}
J.~Kraj{\'{\i}}{\v{c}}ek and T.~Scanlon, \emph{Combinatorics with definable
  sets: {E}uler characteristics and {G}rothendieck rings}, Bull. Symbolic Logic
  \textbf{6} (2000), no.~3, 311--330.

\bibitem{LooMot}
E.~Looijenga, \emph{Motivic measures}, Ast\'erisque (2002), no.~276, 267--297,
  S\'eminaire Bourbaki, Vol.\ 1999/2000. \MR{MR1886763 (2003k:14010)}

\bibitem{Mar}
D.~Marker, \emph{Model theory}, Graduate Texts in Mathematics, vol. 217,
  Springer-Verlag, New York, 2002.

\bibitem{Mats}
H.~Matsumura, \emph{Commutative ring theory}, Cambridge University Press,
  Cambridge, 1986.

\bibitem{SchEC}
H.~Schoutens, \emph{Existentially closed models of the theory of {A}rtinian
  local rings}, J. Symbolic Logic \textbf{64} (1999), 825--845.

\end{thebibliography}


\providecommand{\bysame}{\leavevmode\hbox to3em{\hrulefill}\thinspace}
\providecommand{\MR}{\relax\ifhmode\unskip\space\fi MR }
\providecommand{\MRhref}[2]{%
  \href{http://www.ams.org/mathscinet-getitem?mr=#1}{#2}
}
\providecommand{\href}[2]{#2}

\end{document}